\documentclass[final]{siamltex}

\usepackage{amsfonts,amsmath,amssymb}
\usepackage{mathrsfs,mathtools}
\usepackage{enumerate,enumitem}
\usepackage{hyperref}
\usepackage{esint}
\usepackage{graphicx}
\usepackage{bm}
\usepackage{cite}
 \usepackage[usenames,dvipsnames]{xcolor} 
\usepackage{todonotes}
\usepackage{comment}

\usepackage{float}

\DeclareMathAlphabet{\mathpzc}{OT1}{pzc}{m}{it}

\newcommand{\Rey}{\mathrm{Re}}
\newcommand{\Pra}{\mathrm{Pr}}
\newcommand{\Gra}{\mathrm{Gr}}

\newcommand{\vv}{\mathbf{v}}
\newcommand{\ww}{\mathbf{w}}

\newcommand{\boldg}{\mathbf{g}}
\newcommand{\vvarphi}{\bm{\varphi}}

\newcommand{\diver}{\mathrm{div}\,}
\newcommand{\n}{\mathbf{n}}

\title{A function space approach to the shape
optimization of the Boussinesq system\thanks{The authors acknowledge support from the American Institute of Mathematics (AIM) through the SQuaRE project {\em Optimal mixing and control of heat conductive flows}. Partial support for C.N.R. and L.M. was provided by NSF grants DMS-2012391 and DMS-2309557, respectively. W.H. received partial support from NSF grant DMS-2205117 and AFOSR grant FA9550-23-1-0675. A.C. was partially supported by the UBA grant UBACYT 56BA and the ANPCyT grant PICT 2019-00985.}}
\author{Andrea Ceretani \thanks{Departamento de Matemática, Facultad de Ciencias Exactas y Naturales, Universidad de Buenos Aires (UBA), and Instituto de Investigaciones Matemáticas Luis A. Santaló (IMAS), UBA-CONICET, Buenos Aires, Argentina, \texttt{aceretani@dm.uba.ar}}
\and
Weiwei Hu\footnotemark[3] \and
Lin Mu \thanks{Department of Mathematics, University of Georgia, Athens, GA 30602, USA,
\texttt{weiwei.hu@uga.edu}, \texttt{lin.mu@uga.edu}}
\and Carlos N. Rautenberg \thanks{Department of Mathematical Sciences and the Center for Mathematics and Artificial Intelligence (CMAI), George Mason University, Fairfax, VA, 22030, USA, \texttt{crautenb@gmu.edu}}
}

\begin{document}  
\maketitle
\begin{abstract}
We investigate a shape optimization problem for a heat-conducting fluid governed by a Boussinesq system. The main goal is to determine an optimal domain shape that yields a temperature distribution as uniform as possible. Initially, we analyze the state problem, prove its well-posedness and establish a local boundary regularity result for the weak solution. We then demonstrate the existence of an optimal shape and derive a first-order optimality condition. This requires the derivation and analysis of the adjoint system associated with the Boussinesq model, as well as a rigorous treatment of the directional derivatives of the objective functional under appropriate domain perturbations. Finally, we present numerical experiments that illustrate and support the theoretical findings. 
\end{abstract}

\begin{keywords}
Shape optimization, Navier-Stokes equations, Boussinesq equations, first-order optimality condition.
\end{keywords}

\begin{AMS}
65K10, 
49J20, 
76D05, 
35D30, 
74S05. 
\end{AMS}

\section{Introduction}
In this work we consider the problem of identifying the optimal shape of an incompressible thermal fluid container in order to force temperature to be homogeneous (and identical) in all points. In particular, we consider a two dimensional domain containing a fluid described by the incompressible Boussinesq equations that model velocity, pressure, and temperature of the fluid. The former two variables are modeled by means of a Navier-Stokes system, and the latter one by a convection-diffusion equation. The coupling of the two systems of equations is done via fluid velocity entering in the convective part of the heat equation, and the buoyancy term in the Navier-Stokes system determined by a linear function of the temperature. We endow the differential equations with nonhomogeneous Dirichlet boundary conditions for the temperature, and Dirichlet homogeneous for the fluid velocity.

The goal of modifying the shape of the container is to achieve the highest possible stage of mixing by means of a passive strategy. The scenario described here is common to fluid problems where a quantity of interest (temperature, or density of substance) is diffused and convected/advected by the velocity of the fluid. The problem under investigation is widespread in practical applications, spanning areas such as food industry, energy efficiency in buildings and indoor gas concentration detection. 

In the framework presented in the paper, for a container (domain) $\Omega_\gamma$ we allow only deformations on the bottom boundary $\Gamma_\gamma$ while the rest of the walls remain in place. The structural constraints in this setting to keep the problem realistic are several: (i) Changes of shape of the bottom $\Gamma_\gamma$ should not significantly change the volume $|\Omega_\gamma|$ of $\Omega_\gamma$ (ii) The deformation of bottom $\Gamma_\gamma$ of the container should be described by a  sufficiently smooth function $\gamma$ (iii) The bottom of the container should be a positive distance away of the top container. A few words are in order concerning these constraints; without (i) the optimal shape problem for mixing is rendered ill-posed, in fact, reducing the volume of the container (in general) improves mixing, so infimizing sequences drive volume to zero.  Since we aim to obtain a shape that can be machined or 3D-printed, if we remove (ii) a highly oscillatory bottom would always improve mixing. Finally, the restriction in (iii) preserves the integrity of the interior of the container as a connected set. 

\begin{figure}[ht!]
\centering
\includegraphics[scale=.3]{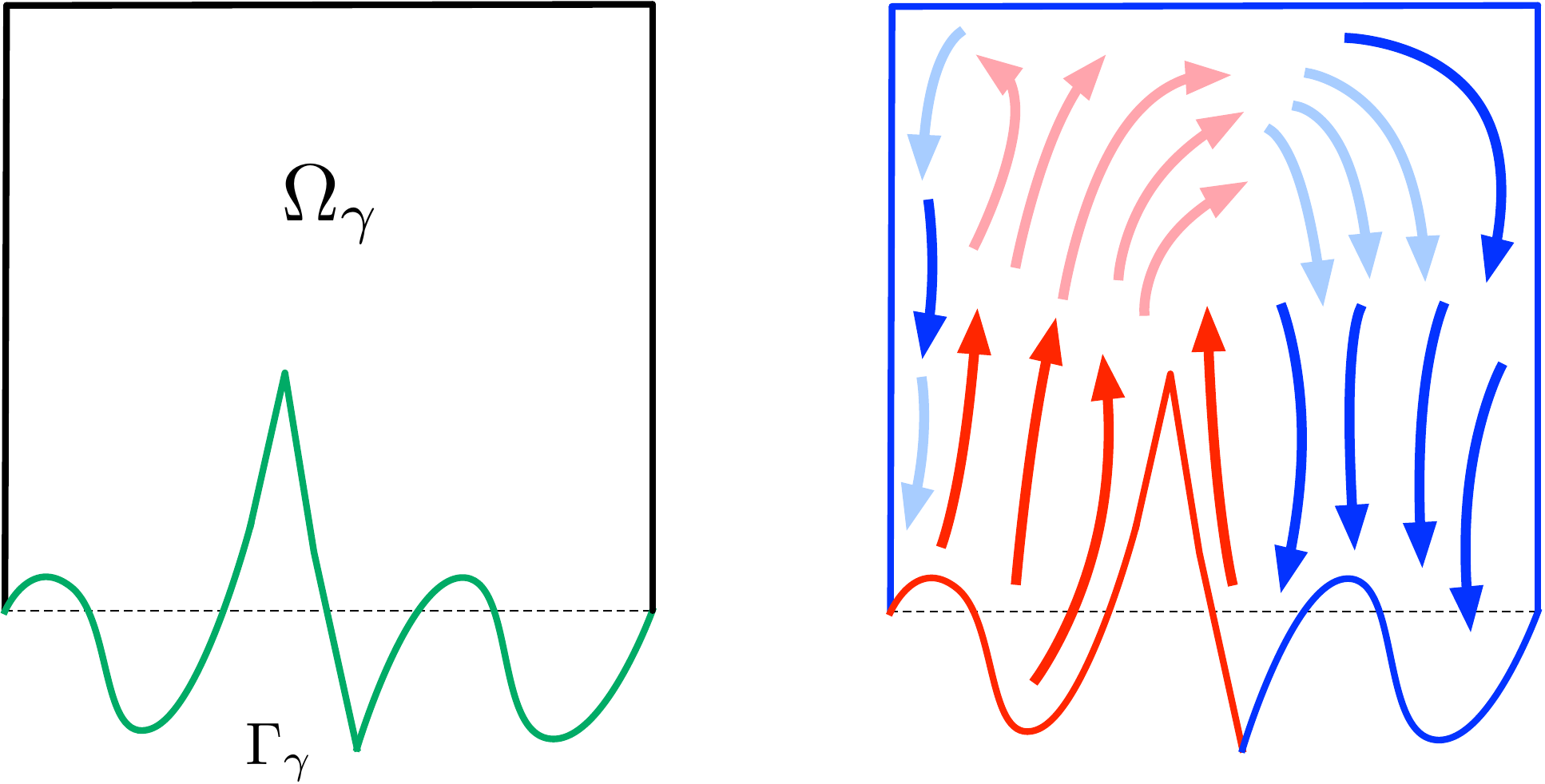}
\label{Fig:Domain}
\caption{Possible admissible domain (left). Fluid temperature and velocity  generated by the buoyancy effect (right).}
\end{figure}

Shape optimization problems with  constraints associated to fluid equations are of significant interest and possess  high level complexities. Involving foundational results of shape optimization, we refer the  {reader} to the monographs \cite{delfour2011shapes} and \cite{sokolowski1992introduction}. For Navier-Stokes systems and Stokes equations, to \cite{MR0609732} and \cite{Girault:1986fk}. Concerning Navier-Stokes in the framework of shape optimization, important works are found in the book \cite{plotnikov2012compressible} by Plotnikov and Sokolowski  and the paper \cite{boisgerault2018shape} by  Zol\'{e}sio and Boisg\'{e}rault. Using low regularity of the boundary and for the Stokes system, the shape optimization problem was considered in \cite{CeHuRa2023}. In \cite{halanay2009shape}, the authors consider the Navier-Stokes system with homogeneous Dirichlet boundary conditions, an existence result and an algorithm is provided. For applied problems involving fluid equations we refer the reader to \cite{mohammadi2004shape}, the monograph \cite{mohammadi2009applied} and references therein. Concerning shape/topology optimization, and modelling of the Boussinesq system, the literature is rather scarce. Notable exceptions are the recent works of \cite{yan2022shape} and \cite{ViCo2022}. In the former the authors consider an obstacle within the domain, and in the latter topological changes of the domain are allowed; see also \cite{CeRa2019,ArCeRa2020}.

Optimization problems constrained by the Navier-Stokes equations are often challenging due to the nonlinear nature of these equations. Factors such as the typically low regularity of the spatial domain, potential coupling with transport equations, and the presence of mixed boundary conditions further complicate their mathematical analysis. Addressing these issues carefully is crucial to ensuring good regularity properties for the fluid-related variables (e.g., velocity, pressure, temperature), which are essential for solving the optimization problem. Over the past decades, significant contributions have been made to the study of problems governed by the Stokes, Navier-Stokes, or Boussinesq equations. For example, see \cite{HuWu2018,HuWu2019,HiRaMoKa1017,DLReTr2007,KuMa2000,BuHeHu2016,BuHu2013} and the references therein.

The structure of the article is as follows. In Section \ref{Sec:Boussinesq}, we introduce the fluid dynamics problem of interest within a fixed container and present the corresponding mathematical model, governed by the Boussinesq equations. For this model, we analyze the existence, uniqueness, and boundary regularity properties of weak solutions. In Section \ref{Sec:Optimization}, we formulate a shape optimization problem aimed at designing an optimal domain that yields a temperature distribution as uniform as possible, and we prove the existence of an optimal shape. We then turn our attention to the derivation of first-order optimality conditions. To this end, we formally derive the adjoint system associated with the Boussinesq system in Section \ref{Sect:Adjoint}, establish its well-posedness, and examine the boundary regularity of its weak solution. In Section \ref{Sec:OptmalityCondition}, we consider a class of domain perturbations that allows us to prove the existence of directional derivatives of the objective functional and to derive a first-order optimality condition. Section \ref{Sec:Numerics} presents a series of numerical experiments that validate and illustrate the theoretical results. Finally, in Section \ref{Sec:Conclusion}, we summarize the main findings and suggest possible directions for future research.

\section{The Boussinesq system}\label{Sec:Boussinesq} This section is devoted to present the mathematical model of the fluid behavior, and its properties. Under usual regularity and smallness assumptions for data, we prove the existence and uniqueness of a weak solution in Theorem \ref{Thm:weak-EU}. In addition, we show that extra regularity for data improves the boundary regularity of the weak solution, see Theorem \ref{Thm:BdryReg}. 

Let $\Omega\subset\mathbb{R}^2$ be an open bounded set with a Lipschitz boundary $\partial\Omega$. Also, let $\Gamma_\gamma\subset\partial\Omega$ be an open and connected boundary portion with non-zero Lebesgue 1-dimensional measure. We consider a heat-conductive incompressible stationary flow described by the velocity of the fluid $\vv$, the pressure $p$, and the temperature $T$. We assume that fluid velocity is zero at $\partial\Omega$, the temperature at $\Gamma_\gamma$ is prescribed by $T_d$ and it is zero on the rest of the walls of the domain. Flow is generated by buoyancy (high temperature fluid raises), and temperature is convected by flow velocity and diffuses in the fluid.
Mathematically, we model the behavior and temperature of the fluid by the Boussinesq system 
\begin{align}
\label{v-edp}\vv\cdot\nabla \vv -\frac{1}{\Rey}\Delta \vv+\nabla p-\frac{\Gra}{\Rey^2}T\boldsymbol{e}=\,&\boldsymbol{g}_1&\text{in}&\quad\Omega,&\\
\label{v-DivFree}\nabla\cdot\vv=\,&0&\text{in}&\quad\Omega,&\\
\label{v-NoSlip}\vv=\,&0&\quad\text{on}&\quad\partial\Omega,
\end{align}
and 
\begin{align}
\label{T-edp}\vv\cdot\nabla T -\frac{1}{\Rey\Pra}\Delta T=\,&g_2&\text{in}&\quad\Omega,&\\
\label{T-Dirichlet}T=\,&T_d&\text{on}&\quad\Gamma_{\gamma},&\\
\label{T-Walls}T=\,&0&\text{on}&\quad\partial\Omega\setminus\Gamma_\gamma,
\end{align}
where ${\bf e}=(0,1)$ is a unit vector in the direction of buoyancy, and ${\bf g}_1$, $g_2$ are associated to possible heat sources and fluid perturbations within the domain, respectively. Moreover, $\Rey$, $\Pra$, $\Gra$ stand for Reynolds, Prandtl, and Grashof numbers. It should be noted that the systems of differential equations \eqref{v-edp}-\eqref{v-NoSlip} and \eqref{T-edp}-\eqref{T-Walls} are fully coupled.

In order to establish the weak form of \eqref{v-edp}-\eqref{T-Walls}, we introduce the spaces
\begin{align*}
V(\Omega):=\,&\{\boldsymbol\vvarphi\in H^{1}(\Omega;\mathbb{R}^2):\diver{\boldsymbol\vvarphi}=0\text{ a.e. in }\Omega, \boldsymbol\vvarphi=0\text{ on }\partial\Omega\text{ in the trace sense}\},\\
H_0^1(\Omega):=\,&\{\varphi\in H^{1}(\Omega):\varphi=0\text{ on }\partial\Omega\text{ in the trace sense}\},
\end{align*}
where $H^{1}(\Omega;\mathbb{R}^d)$ is the usual Sobolev space of functions with values in $\mathbb{R}^d$ (and $H^{1}(\Omega):=H^{1}(\Omega;\mathbb{R}^1)$, which belong to the Lebesgue space $L^2(\Omega)$ together with their first order weak partial derivatives. Throughout the paper, we denote vector fields with boldface letters and scalar quantities in regular font.  Furthermore, we consider them equipped with the norms 
\begin{align*}
\|\boldsymbol\vvarphi\|_{V(\Omega)}=\left(\int_\Omega|\nabla\boldsymbol\vvarphi|^2\,dx\right)^{1/2}\qquad\text{and}\qquad
\|\varphi\|_{H_0^1(\Omega)}=\left(\int_\Omega|\nabla \varphi|^2\,dx\right)^{1/2},
\end{align*}
respectively. The corresponding dual spaces are denoted by $V(\Omega)'$ and $H^{-1}(\Omega)$.

In addition, we consider $\boldg_1\in V(\Omega)'$, $g_2\in H^{-1}(\Omega)$, and we assume that $T_d\in H^1(\Omega)$ such that 
\begin{align*}
T_d\big|_{\partial\Omega\setminus\Gamma_\gamma}=0\quad\text{on}\quad\partial\Omega\setminus\Gamma_\gamma.
\end{align*} 
Therefore, note that we have $T_d\big|_{\partial \Omega}\in H^{1/2}(\partial\Omega)$. 

The weak formulation of \eqref{v-edp}-\eqref{T-Walls} is now stated as follows: Find $(\vv,T)\in V(\Omega)\times H^1(\Omega)$ with $\hat{T}=T-T_d\in H_0^1(\Omega)$ that satisfies:
\begin{align}
\label{eq:weak-v}b_1(\vv,\vv,\vvarphi)+\frac{1}{\Rey}(\nabla\vv,\nabla\vvarphi)_2-\frac{\Gra}{\Rey^2}(T\mathbf{e},\vvarphi)_2=\,&\langle\boldg_1,\vvarphi\rangle&\forall&\,\vvarphi\in V(\Omega),&\\
\label{eq:weak-T}b_2(\vv,T,\varphi)+\frac{1}{\Rey\Pra}(\nabla T,\nabla\varphi)_2=\,&\langle g_2,\varphi\rangle&\forall&\,\varphi\in H_0^1(\Omega),&
\end{align}
where $(\cdot,\cdot)_2$ denotes the inner product in $L^2$, and $b_1$, $b_2$ are defined by
\begin{align}
\label{eq:bv}b_1(\vv,{\boldsymbol w},\vvarphi)=\,&\int_\Omega (\vv\cdot\nabla {\boldsymbol w})\cdot\vvarphi\,dx,&&{\boldsymbol w},\vv,\vvarphi\in H^{1}(\Omega;\mathbb{R}^2),&\\
\label{eq:bT}b_2(\vv,T,\varphi)=\,&\int_\Omega (\vv\cdot\nabla T)\varphi\,dx,&&T\in H^1(\Omega),\,\vv,\vvarphi\in H^{1}(\Omega;\mathbb{R}^2).&
\end{align}
If $(\vv,T)\in V(\Omega)\times H^1(\Omega)$ satisfies \eqref{eq:weak-v}-\eqref{eq:weak-T} with $\hat{T}=T-T_d\in H_0^1(\Omega)$, then we say that $(T,\vv)$ is a weak solution to the Boussinesq system \eqref{v-edp}-\eqref{T-Walls}. The existence and uniqueness of a weak solution can be proved by a fixed point argument, as in \cite{ArCeRa2020}. More precisely, we have the next result. 

\begin{theorem}[Existence and uniqueness of a weak solution to the Boussinesq system]\label{Thm:weak-EU}
Let $\boldg_1\in V(\Omega)'$, $g_2\in H^{-1}(\Omega)$, and $T_d\in H^1(\Omega)$. Then there exist $\varepsilon_1,\varepsilon_2,\varepsilon_3>0$ such that if
\begin{equation}\label{cond:RePrGr}
\Rey\in(0,\varepsilon_1),\qquad
\Pra\in\left(0,\frac{\varepsilon_2}{\Rey}\right),\qquad
\Gra\in(0,\varepsilon_3\Rey),
\end{equation}
and $\boldg_1$, $g_2$, and $T_d$ are sufficiently small, there is unique weak solution $(\vv,T)\in V(\Omega)\times H^1(\Omega)$ to \eqref{v-edp}-\eqref{T-Walls}. 
\end{theorem}

\begin{proof} Notice that \eqref{eq:weak-v}-\eqref{eq:weak-T} can be written as
\begin{align}
\label{eq:weak-v-hat-bis}\frac{1}{\Rey}(\nabla\vv,\nabla\vvarphi)_2=\,&\langle\boldg_1,\vvarphi\rangle-\frac{\Gra}{\Rey^2}(T_d\mathbf{e},\vvarphi)_2-b_1(\vv,\vv,\vvarphi)+\frac{\Gra}{\Rey^2}(\hat{T}\mathbf{e},\vvarphi)_2&\forall&\,\vvarphi\in V(\Omega),&\\
\label{eq:weak-T-hat-bis}\frac{1}{\Rey\Pra}(\nabla \hat{T},\nabla\varphi)_2=\,&\langle g_2,\varphi\rangle-\frac{1}{\Rey\Pra}(\nabla T_d,\nabla\varphi)_2 -b_2(\vv,\hat{T},\varphi)&\forall&\,\varphi\in H_0^1(\Omega).&
\end{align}

Let $R_1:V(\Omega)\to V(\Omega)'$ and $R_2:H_0^1(\Omega)\to H^{-1}(\Omega)$ be defined by
\begin{equation}\label{eq:R1R2}
\langle R_1(\vv),\vvarphi\rangle=\frac{1}{\Rey}(\nabla\vv,\nabla\vvarphi)_2\qquad\text{and}\qquad
\langle R_2(\hat{T}),\varphi\rangle=\frac{1}{\Rey\Pra}(\nabla\hat{T},\nabla\varphi)_2.
\end{equation}
and note that the inverse operators $R_1^{-1}$ and $R_2^{-1}$ exist. Further, let $P_1:V(\Omega)\times H_0^1(\Omega)\to V(\Omega)'$ and $P_2:V(\Omega)\times H_0^1(\Omega)\to H^{-1}(\Omega)$ be defined by
\begin{align}
\label{eq:P1}\langle P_1(\vv,\hat{T}),\vvarphi\rangle=\,&\langle\boldg_1,\vvarphi\rangle-\frac{\Gra}{\Rey^2}(T_d\mathbf{e},\vvarphi)_2-b_1(\vv,\vv,\vvarphi)+\frac{\Gra}{\Rey^2}(\hat{T}\mathbf{e},\vvarphi)_2,&\\
\label{eq:P_2}\langle P_2(\vv,\hat{T}),\varphi\rangle=\,&\langle g_2,\varphi\rangle-\frac{1}{\Rey\Pra}(\nabla T_d,\nabla\varphi)_2-b_2(\vv,\hat{T},\varphi).
\end{align}
Finally, let $F:V(\Omega)\times H_0^1(\Omega)\to V(\Omega)\times H_0^1(\Omega)$ be defined as
\begin{equation}\label{eq:F}
F=(R_1^{-1}P_1,R_2^{-1}P_2).
\end{equation}

Then, a necessary and sufficient condition to be $(\vv,T)$ a weak solution to \eqref{v-edp}-\eqref{T-Walls} is that $(\vv,\hat{T})$ is a fixed point of $F$. The existence and uniqueness of a fixed point of $F$ can be shown similarly as in \cite[Thm 2.1]{ArCeRa2020}, using the Banach fixed point theorem. The proof reduces to observe that we can always find $\tau>0$ and
$\varepsilon _{1},\varepsilon _{2},\varepsilon _{3}>0$ such that if \eqref{cond:RePrGr} holds true, then 
\begin{align}\label{cond:Smallness}
K_{1}<\,1\qquad \text{ and } \qquad K_{2}\leq \, \tau ,
\end{align}
provided that $\boldg_1$, $g_2$, and $T_d$ are sufficiently small, where  
\begin{align}
\label{eq:K1}K_1:=\,&C\left(\tau\left(\Rey+\Rey\Pra\right)+\Rey\Pra\|T_d\|_{H^1(\Omega)}+\frac{\Gra}{\Rey}\right),\\
\label{eq:K2}K_2:=\,&C\left(\tau^2\left(\Rey+\Rey\Pra\right)+\tau\left(\Rey\Pra\|T_d\|_{H^1(\Omega)}+\frac{\Gra}{\Rey}\right)
+\Rey\|\boldg_1\|_{V(\Omega)'}\right.\notag\\
&\left.+\Rey\Pra\|g_2\|_{H^{-1}(\Omega)}+\|T_d\|_{H^1(\Omega)}+\frac{\Gra}{\Rey}\|T_d\|_{H^1(\Omega)}\right),
\end{align} 
and $C$ is a positive constant independent of $\boldg_1$, $g_2$, $T_d$, $\Rey$, $\Pra$, and $\Gra$. We skip the details to avoid repetition, see \cite[Thm 2.1]{ArCeRa2020}. Hence, $F$ is a contraction from the closed ball 
\begin{align*}
\bar{B}_{\tau}:=\{(\boldsymbol\vvarphi,\varphi)\in V(\Omega)\times H_0^1(\Omega)\,:\,\|\boldsymbol\vvarphi\|_{V(\Omega)}+\|\varphi\|_{H_0^1(\Omega)}\leq\tau\},
\end{align*}
into itself. Then, by the Banach's fixed point theorem there
is a unique weak solution to \eqref{v-edp}-\eqref{T-Walls} in $\bar{B}_{\tau}$. 
\end{proof}

Note that Theorem \ref{Thm:weak-EU} yields a weak solution with $H^{1/2}$ boundary regularity. Improved regularity can be achieved by assuming extra regularity on $\Gamma_\gamma$, $\boldg_1$, $g_2$, and $T_d$, as the next theorem shows. 

\begin{theorem}[Boundary regularity]\label{Thm:BdryReg}
Assume that the boundary part $\Gamma_\gamma$ is of class $C^2$. Further, consider $\boldg_1\in L^2(\Omega;\mathbb{R}^2)$, $g_2\in L^2(\Omega)$, and $T_d\in H^{2}(\Omega)$ . If $\Rey$, $\Pra$, $\Gra$ satisfy \eqref{cond:RePrGr} and $\boldg_1$, $g_2$, $T_d$ satisfy the smallness condition in Theorem \ref{Thm:weak-EU}, then the weak solution $(\vv,T)\in V(\Omega)\times H^1(\Omega)$ to \eqref{v-edp}-\eqref{T-Walls} satisfies
\begin{align}\label{eq:v-T-bdry-reg}
\vv\big|_{\Gamma_\gamma^\varepsilon}\in H^{3/2}(\Gamma_\gamma^\varepsilon;\mathbb{R}^2),\qquad\qquad
T\big|_{\Gamma_\gamma^\varepsilon}\in H^{3/2}(\Gamma_\gamma^{\varepsilon}),
\end{align}  
for every relative open subset $\Gamma_\gamma^\varepsilon$ of $\partial\Omega$ strictly contained in $\Gamma_\gamma$, i.e., $\overline{\Gamma_\gamma^\varepsilon}\subset \Gamma_\gamma$. 
\end{theorem}

\begin{proof}
Since $\vv$ satisfies \eqref{eq:weak-v}, we notice that
\begin{align*}
(\nabla\vv,\nabla\vvarphi)_2=\langle{\bf G}_1,\vvarphi\rangle\qquad\forall\,\vvarphi\in V(\Omega),
\end{align*}
where $\mathbf{G}_1\in V(\Omega)'$ is defined by
\begin{align*}
\langle{\bf G}_1,\vvarphi\rangle=\,&\Rey\left(({\bf g}_1,\vvarphi)_2+\frac{\Gra}{\Rey^2}(T{\bf e},\vvarphi)_2-b_1(\vv,\vv,\vvarphi)\right)\\
=\,&\Rey\int_{\Omega}\left({\bf g}_1+\frac{\Gra}{\Rey^2}T{\bf e}+\vv\cdot\nabla\vv\right)\cdot\vvarphi\,dx.
\end{align*}
Given that $\vv\in H^1(\Omega;\mathbb{R}^2)$ and $H^1(\Omega)\hookrightarrow L^q(\Omega)$ for every $2\leq q<\infty$, we observe that $\vv\in L^q(\Omega;\mathbb{R}^2)$ for every $2\leq q<\infty$. In addition, $\nabla\vv\in L^2(\Omega;\mathbb{R}^{2\times 2})$. Then, $\vv\cdot\nabla\vv\in L^r(\Omega;\mathbb{R}^2)$ for every $1<r<2$. Therefore, ${\bf G}_1\in L^r(\Omega;\mathbb{R}^2)$ for every $1<r<2$. 

Consider a relative open subset  $\Gamma_\gamma^\varepsilon$ of $\partial\Omega$, strictly contained in $\Gamma_\gamma$. Let $\Omega_0\subset\Omega$ be an open set such that $\partial\Omega_0\cap(\partial\Omega\setminus\Gamma_\gamma)=\emptyset$ and $\Gamma_\gamma^\varepsilon\subset\partial\Omega_0\cap\Gamma_\gamma$. Notice that the boundary portion $\sigma=\partial\Omega_0\cap\Gamma_\gamma$ of the set $\Omega_0$ is of class $C^2$. Let $\Omega'\subset\Omega_0$ be an open set such that $\partial\Omega'\cap(\partial\Omega_0\setminus\Gamma_\gamma)=\emptyset$, $(\partial\Omega'\cap\partial\Omega_0)\cap\Gamma_\gamma\subsetneq(\partial\Omega_0\cap\partial\Omega)\cap\Gamma_\gamma$, and $\Gamma_\gamma^\varepsilon\subset\partial\Omega'\cap\partial\Omega_0$. Thus, by \cite[Thm. IV.5.1, p. 276]{Ga2011} (see also \cite[Lemma IV.1.1, p. 235]{Ga2011}), we deduce that $\vv\in W^{2,r}(\Omega';\mathbb{R}^2)$ where $1<r<2$. Exploiting the embedding $W^{2,r}(\Omega')\hookrightarrow L^\infty(\Omega')$ we get that $\vv\in L^\infty(\Omega';\mathbb{R}^2)$. In particular, this implies that $\vv\cdot\nabla\vv\in L^2(\Omega';\mathbb{R}^2)$. Therefore, $\mathbf{G}_1\in L^2(\Omega';\mathbb{R}^2)$. Repeating this argument, we deduce that $\vv\in H^2(\Omega'';\mathbb{R}^2)$ for some $\Omega''\subset\Omega'$ such that $\Gamma_\gamma^\varepsilon\subset\partial\Omega''$. Then, by the embedding $H^2(\Omega'')\hookrightarrow H^{3/2}(\partial\Omega'')$, we obtain $
\vv\big|_{\Gamma_\gamma^\varepsilon}\in H^{3/2}(\Gamma_\gamma^\varepsilon;\mathbb{R}^2)$.

Without loss generality, the set $\Omega''$ can be assumed to be of class $C^2$ and such that $\Gamma_\gamma^\varepsilon\subsetneq\partial\Omega''$. Let $\eta\in C^\infty(\bar{\Omega})$ be such that $\eta=1$ on $\Gamma_\gamma^\varepsilon$ and $\eta=0$ on $\partial\Omega''\setminus\Gamma_\gamma$. Notice that $\eta\hat{T}\in H^1_0(\Omega'')$. Since $\hat{T}$ satisfies \eqref{eq:weak-T}, we observe that 
\begin{align*}
(\nabla\hat{T},\nabla\varphi)_2=\langle G_2,\varphi\rangle\qquad\forall\,\varphi\in H_0^1(\Omega),
\end{align*}
where $G_2\in H^{-1}(\Omega)$ is given by
\begin{align*}
\langle G_2,\varphi\rangle=\,&\Rey\Pra\left((g_2,\varphi)_2-\frac{1}{\Rey\Pra}(\nabla T_d,\nabla\varphi)_2-b_2(\vv,T,\varphi)\right).
\end{align*}
A straightforward computation yields that
\begin{align*}
\int_\Omega\nabla(\eta\hat{T})\cdot\nabla\varphi\,dx=\int_{\Omega}\nabla\hat{T}\cdot\nabla(\eta\varphi)\,dx-\int_\Omega\left(2\nabla\eta\cdot\nabla\hat{T}+\hat{T}\Delta\eta\right)\varphi\,dx.
\end{align*}
Thus, 
\begin{align*}
&\int_{\Omega}\nabla(\eta\hat{T})\cdot\nabla\varphi\,dx=\langle G_2,\eta\varphi\rangle-\int_\Omega\left(2\nabla\eta\cdot\nabla\hat{T}+\hat{T}\Delta\eta\right)\varphi\,dx\\
&\qquad=-\int_\Omega\nabla T_d\cdot\nabla\varphi\,dx+\int_\Omega\left(\Rey\Pra (g_2-\vv\cdot \nabla T)-2\nabla\eta\cdot\nabla\hat{T}-\hat{T}\Delta\eta\right)\varphi\,dx,
\end{align*}
for every $\varphi\in H^1_0(\Omega)$. From this we get that
\begin{align*}
\int_{\Omega''}\nabla(\eta\hat{T})\cdot\nabla\varphi\,dx=\,&\langle G,\varphi\rangle\qquad\forall\,\varphi\in H^1_0(\Omega''),
\end{align*}
where $G\in H^{-1}(\Omega'')$ is given by
\begin{align*}
\langle G,\varphi\rangle=-\int_{\Omega''}\nabla T_d\cdot\nabla\varphi\,dx+\int_{\Omega''}\left(\Rey\Pra (g_2-\vv\cdot \nabla T)-2\nabla\eta\cdot\nabla\hat{T}-\hat{T}\Delta\eta\right)\varphi\,dx.
\end{align*}
Exploiting that $T_d\in H^2(\Omega)$, we get 
\begin{align*}
\langle G,\varphi\rangle=\int_{\Omega''}\varphi\Delta T_d\,dx+\int_{\Omega''}\left(\Rey\Pra (g_2-\vv\cdot \nabla T)-2\nabla\eta\cdot\nabla\hat{T}-\hat{T}\Delta\eta\right)\varphi\,dx.
\end{align*}
From this, we observe that $G\in L^2(\Omega'')$ since $\vv\in L^\infty(\Omega'';\mathbb{R}^2)$. In summary, we have that $\eta\hat{T}$ weakly satisfies
\begin{align*}
-\Delta(\eta\hat{T})=\,&G&\text{in}&\quad\Omega'',&\\
\eta\hat{T}=\,&0&\text{on}&\quad\partial\Omega'',&
\end{align*}
where $G\in L^2(\Omega'')$ and $\Omega''$ is of class $C^2$. Then, by \cite[Them 4, p. 334]{Ev2010}, we obtain that $\eta\hat{T}\in H^2(\Omega'')$. By the embedding $H^2(\Omega'')\hookrightarrow H^{3/2}(\partial\Omega'')$ and the fact that $\eta=1$ on $\Gamma_\gamma^\varepsilon$, we deduce that $T\big|_{\Gamma_\gamma^\varepsilon}\in H^{3/2}(\Gamma_\gamma^\varepsilon)$.
\end{proof}

\section{The shape optimization problem}\label{Sec:Optimization}
We now address the existence of an optimal domain to achieve a fluid temperature as uniform as possible. We focus on two dimensional cubes and allow perturbations only on the bottom boundary. More precisely, we consider the following class of domains (see also Fig. \ref{Fig:Admissible1}).

\begin{figure}[ht!]
\centering
\includegraphics[scale=.3]{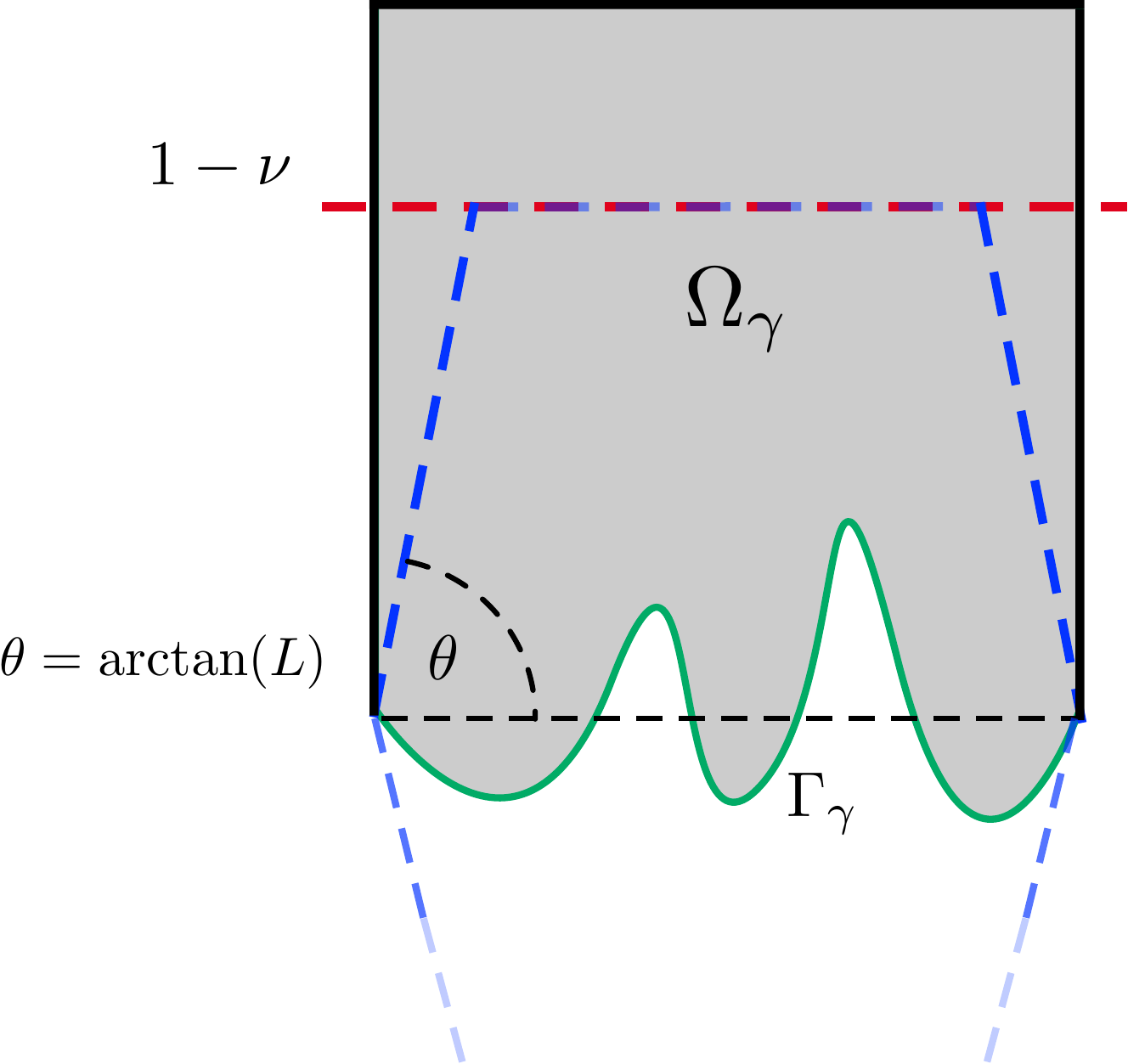}
\label{Fig:Admissible1}
\caption{Admissible domain $\Omega_\gamma\in \mathcal{O}$ in gray, $\Gamma_\gamma$ in {\color{ForestGreen}green}, where $1-\nu$ is represented by the dashed line in {\color{red}red} and the dashed line in {\color{blue}blue} represents the region where $\Gamma_\gamma$ might be located and induced by the obstacle constraint $1-\nu$ and derivative bound $L$.}
\end{figure}

\begin{definition}[Class of Admissible Domains $\mathcal{O}$] 
Let $L>0$ and $\nu\in(0,1)$. We say that $\Omega\subset\mathbb{R}^2$ is an admissible domain if 
\begin{equation}\label{eq:AdmissDomain}
\Omega=\{(x_1,x_2)\in\mathbb{R}^2\,:\,0< x_1< 1,\,\gamma(x_1)< x_2< 1\},
\end{equation}
for some $\gamma\in H_0^1(0,1)$ such that $\gamma\leq 1-\nu$  a.e. in $(0,1)$ and $|\gamma'|\leq L$ a.e. in $(0,1)$. We denote by $\mathcal{O}$ the class of all admissible domains $\Omega$.
\end{definition}

 If $\Omega\in\mathcal{O}$ is given by \eqref{eq:AdmissDomain}, we write $\Omega=\Omega_\gamma$. Notice that every $\Omega_\gamma$ is a Lipschitz domain. Also, let
\begin{align}
\Gamma_\gamma=\,&\{(x_1,x_2)\in\partial{\Omega_{\gamma}}\,:\,x_2=\gamma(x_1)\}.
\end{align}

We are interested in the following problem:
\begin{align}\label{shape}
&\text{Find }\Omega_{\gamma^*}\in \mathcal{O}\text{ such that }
J(T^*,\gamma^*)=\inf_{\Omega_\gamma\in \mathcal{O}}J(T,\gamma),\notag\\
&\text{subject to }(\vv,T)\text{ is the weak solution to }\eqref{v-edp}-\eqref{T-Walls} \text{ in }\Omega_\gamma,
\end{align}
where, 
\begin{align}
\label{J-Def}J(T,\gamma):=\,&\|T-I(T,\gamma)\|^2_{L^2(\Omega_\gamma)}+
\frac{\lambda_1}{2}\|\gamma\,'\|^2_{L^2(0,1)}+\frac{\lambda_2}{2}\left(\int_0^1\gamma(\xi)\,d\xi\right)^2,
\end{align}
for positive constants $\lambda_1$ and $\lambda_2$, and
\begin{align}\label{eq:average}
I(T,\gamma):=\frac{1}{|\Omega_\gamma|}\int_{\Omega_\gamma}T(x)\,dx.
\end{align}

The first term of $J$ corresponds to minimizing the variance of the temperature distribution, the second one is introduced to regularize $\gamma$, and the last term is related to maintain a constant volume for the domain. 

Since $\gamma\leq 1-\nu$ and $|\gamma'|\leq L$ a.e. in $(0,1)$ for every $\Omega_\gamma\in\mathcal{O}$, where $L$  does not depend on $\gamma$, we deduce that there exists $\mathrm{M}>0$ such that $\Omega_{\gamma}\subset\Omega^{\mathrm{M}}:=(0,1)\times(-\mathrm{M},1)$ for every $\Omega_\gamma\in\mathcal{O}$. Hereafter, we assume that the boundary data $T_d$ in the Boussinesq system posed in $\Omega_\gamma$ is the restriction in the trace sense of a function in $H^1(\Omega^\mathrm{M})\cap C_c(\Omega^\mathrm{M})$. As before, we denote the latter also by $T_d$. In addition, we consider that the right-hand sides $\boldg_1$ and $g_2$ in the Boussinesq system posed in $\Omega_\gamma$ are restrictions to $\Omega_\gamma$ of functions in $L^2(\Omega^\mathrm{M};\mathbb{R}^d)$ with $d=2$ and $d=1$, respectively, which we likewise denote as $\boldg_1$ and $g_2$. In the following, we shall use repeatedly that functions ${\bf g}_1\in L^2(\Omega^\mathrm{M};\mathbb{R}^2)$, $g_2\in L^2(\Omega^\mathrm{M})$, $T_d\in H^1(\Omega^\mathrm{M})\cap C_c(\Omega^\mathrm{M})$ {\em satisfy the smallness condition in Theorem \ref{Thm:weak-EU}}. This means that condition \eqref{cond:Smallness} in the proof of Theorem \ref{Thm:weak-EU} holds for $K_1$ and $K_2$ replaced by
\begin{align*}
K_1^\mathrm{M}:=\,&C\left(\tau\left(\Rey+\Rey\Pra\right)+\Rey\Pra\|T_d\|_{H^1(\Omega^\mathrm{M})}+\frac{\Gra}{\Rey}\right),\\
K_2^\mathrm{M}:=\,&C\left(\tau^2\left(\Rey+\Rey\Pra\right)+\tau\left(\Rey\Pra\|T_d\|_{H^1(\Omega^\mathrm{M})}+\frac{\Gra}{\Rey}\right)
+\Rey\|\boldg_1\|_{V(\Omega^\mathrm{M})'}\right.\notag\\
&\left.+\Rey\Pra\|g_2\|_{H^{-1}(\Omega^\mathrm{M})}+\|T_d\|_{H^1(\Omega^\mathrm{M})}+\frac{\Gra}{\Rey}\|T_d\|_{H^1(\Omega^\mathrm{M})}\right).
\end{align*} 
These considerations allow us to deduce a uniform bound for the weak solutions of \eqref{v-edp}-\eqref{T-Walls} in $\Omega_\gamma$, as $\Omega_\gamma$ varies in $\mathcal{O}$. More precisely, we have the following result.    

\begin{lemma}\label{Lem:UniformBound}
Consider $\Omega_\gamma\in\mathcal{O}$. Let $\boldg_1\in L^2(\Omega^\mathrm{M};\mathbb{R}^2)$, $g_2\in L^2(\Omega^\mathrm{M})$, and $T_d\in H^1(\Omega^\mathrm{M})\cap C_c(\Omega^\mathrm{M})$. If $\Rey$, $\Pra$, $\Gra$ satisfy \eqref{cond:RePrGr} and $\boldg_1$, $g_2$, $T_d$ satisfy the smallness condition in Theorem \ref{Thm:weak-EU}, then the weak solution $(\vv,T)\in V(\Omega_\gamma)\times H^1(\Omega_\gamma)$ to \eqref{v-edp}-\eqref{T-Walls} in $\Omega_\gamma$ satisfies 
\begin{align}\label{eq:Bound-uThat}
\|\vv\|_{V(\Omega_\gamma)}+\|\hat{T}\|_{H_0^1(\Omega_\gamma)}\leq\tau,
\end{align} 
where $\hat{T}=T-T_d$, for some $\tau>0$ not depending on $\gamma$.  
\end{lemma}

\begin{proof}
Let $\Omega_\gamma\in\mathcal{O}$. The proof reduces to observe that the constant $C$ in the definition of $K_1^\mathrm{M}$ and $K_2^\mathrm{M}$ can be selected to be independent of $\gamma$, see \eqref{eq:K1}-\eqref{eq:K2}. In fact, $C$ is built upon constants $C_1, C_2>0$ that satisfy
\begin{align*}
\|w\|_{L^2(\Omega_\gamma)}\leq\,& C_1(\Omega_\gamma)\|\nabla w\|_{L^2(\Omega_\gamma)}\leq C_1\|\nabla w\|_{L^2(\Omega_\gamma)}&\forall\,w\in H_0^1(\Omega_\gamma)&,\\
\|w\|_{L^4(\Omega_\gamma)}\leq\,& C_2(\Omega_\gamma)\|w\|_{H^1(\Omega_\gamma)}\leq C_2\|w\|_{H^1(\Omega_\gamma)}&\forall\,w\in H^1(\Omega_\gamma)&,
\end{align*}
for some $C_1(\Omega_\gamma), C_2(\Omega_\gamma)>0$ (see the proof of Theorem 2.1 in \cite{ArCeRa2020}). Also, $C$ depends linearly on integer powers of each of them. The constant $C_1(\Omega_\gamma)$ in the Poincar\'e inequality can be chosen to depend only on the diameter of $\Omega_\gamma$, see \cite[Thm 1.4.3.4]{Gr1985}. Every $\Omega\in\mathcal{O}$ satisfies $\Omega^\mathrm{m}\subset\Omega\subset\Omega^\mathrm{M}$, where $\Omega^\mathrm{m}:=(0,1)\times(1-\nu,1)$. Then, the diameter of $\Omega_\gamma$ is bounded from above and from below, and hence $C_1$ is well-defined. In order to see that $C_2$ is well-defined as well, we first notice that each $\Omega\in\mathcal{O}$ has the cone property. Moreover, the cone determining the cone property of each $\Omega\in\mathcal{O}$ can be selected to be independent of $\Omega$. This is a consequence of that every $\Omega\in\mathcal{O}$ is the unit cube with the bottom boundary replaced by the graph of an absolutely continuous function whose derivative is bounded a.e. in $(0,1)$ for a constant that does not depend on $\Omega$. Then, we observe that the constant $C_2(\Omega_\gamma)$ from the embedding $H^1(\Omega_\gamma)\hookrightarrow L^4(\Omega_\gamma)$ can be chosen to depend only on the cone determining the cone property, see \cite[Proof of Thm 5.4]{Ad1975}. This yields the well-definition of $C_2$. 
\end{proof}

We are now in a position to prove the existence of an optimal domain.

\begin{theorem}[Existence of an optimal domain]
Consider $\boldg_1\in L^2(\Omega^{\mathrm{M}};\mathbb{R}^2)$, $g_2\in L^2(\Omega^{\mathrm{M}})$, and $T_d\in H^1(\Omega^{\mathrm{M}})\cap C_c(\Omega^{\mathrm{M}})$. If $\Rey$, $\Pra$, $\Gra$ satisfy \eqref{cond:RePrGr} and $\boldg_1$, $g_2$, $T_d$ satisfy the smallness condition in Theorem \ref{Thm:weak-EU}, then there exists a solution $\Omega_{\gamma^*}\in \mathcal{O}$ to the shape optimization problem \eqref{shape}.
\end{theorem}

\begin{proof}
Since the class of admissible domains $\mathcal{O}$ is non-empty and $J(T,\gamma)\geq 0$ for every $\Omega_\gamma\in \mathcal{O}$, we deduce the existence of a sequence $\{(T_n,\gamma_n)\}$ such that
\begin{equation}\label{eq:InfSeq}
\lim_{n\to\infty}J(T_n,\gamma_n)=\inf_{\Omega_\gamma\in\mathcal{O}}J(T,\gamma)\geq 0,
\end{equation}
where $T_n=\hat{T}_n+T_d$, and $(\vv_n,T_n)$ is the weak solution to \eqref{v-edp}-\eqref{T-Walls} in $\Omega_{\gamma_n}$. 

Let $({\vv}_n^{\mathrm{M}},\hat{T}_n^{\mathrm{M}})$ be the extension by zero of $(\vv_n,\hat{T}_n)$ outside $\Omega_{\gamma_n}$. Since $(\vv_n,\hat{T}_n)\in H_0^1(\Omega_{\gamma_n};\mathbb{R}^2)\times H_0^1(\Omega_{\gamma_n})$, we observe that $(\vv_n^{\mathrm{M}},\hat{T}_n^{\mathrm{M}})\in H_0^1(\Omega^{\mathrm{M}};\mathbb{R}^2)\times H_0^1(\Omega^{\mathrm{M}})$. In addition,
\begin{equation}\label{eq:uHatvHat-bound}
\|\hat{T}_n^{\mathrm{M}}\|_{H^1(\Omega^\mathrm{M})}+\|\vv^{\mathrm{M}}_n\|_{H^1(\Omega^{\mathrm{M}})}\leq \|\hat{T}_n\|_{H^1(\Omega_{\gamma_n})}+\|\vv_n\|_{H^1(\Omega_{\gamma_n})}.
\end{equation}
Then, by Lemma \ref{Lem:UniformBound}, we get that $\{(\vv_n^{\mathrm{M}},\hat{T}_n^{\mathrm{M}})\}$ is bounded in $H^1(\Omega^{\mathrm{M}};\mathbb{R}^2)\times H^1(\Omega^\mathrm{M})$, and therefore it converges weakly to some $(\vv^{\mathrm{M}},\hat{T}^{\mathrm{M}})\in H^1(\Omega^\mathrm{M};\mathbb{R}^2)\times H^1(\Omega^\mathrm{M})$ along a subsequence. We denote the subsequence also by $\{(\vv^{\mathrm{M}}_n,\hat{T}_n^{\mathrm{M}})\}$, so that
\begin{equation}\label{eq:weak}
\vv^{\mathrm{M}}_n\rightharpoonup\vv^{\mathrm{M}}\quad\text{ in }H^1(\Omega^\mathrm{M};\mathbb{R}^2)\qquad\text{and}\qquad
\hat{T}_n^{\mathrm{M}}\rightharpoonup\hat{T}^{\mathrm{M}}\quad\text{ in }H^1(\Omega^\mathrm{M}).
\end{equation} 

We further notice that there exists $\Omega^*\subset\Omega^\mathrm{M}$ such that $\{\Omega_{\gamma_n}\}$ converges to $\Omega^*$ along a subsequence (again indexed by $n$) in the Hausdorff complementary metric, see \cite[Proposition A3.2, p. 461; and Theorem A3.9, p. 466]{NeSpTi2006}. Moreover, $\chi_{\Omega_{\gamma_n}}\to\chi_{\Omega^*}$ pointwise a.e. in $\Omega^\mathrm{M}$ along a subsequence (again indexed by $n$), see \cite[Theorem A3.2, p.461]{NeSpTi2006} (see also \cite[p.54]{NeSpTi2006}). From this, we deduce that
\begin{equation*}
\Omega_{\gamma^*}=\{(x_1,x_2)\in\mathbb{R}^2\,:\,0<x_1<1   
 ,\,\gamma^\ast(x_1)<x_2<1\},
\end{equation*}                             
where $\gamma^*$ is the a.e. pointwise limit of $\{\gamma_n\}$ in $[0,1]$. Notice that $\gamma^*\leq 1-\nu$ a.e. in $\Omega$. We further observe that $\{\gamma_n\}$ is bounded in $H_0^1(0,1)$ and therefore it converges weakly to some function $\gamma_*\in H_0^1(0,1)$. In particular, this implies that $\gamma_n\rightharpoonup\gamma_*$ in $L^2(0,1)$ and hence 
\begin{equation}\label{eq:gammaConv}
\int_0^1\gamma_n\varphi\,d\xi\to\int_0^1\gamma_*\varphi\,d\xi\qquad\forall\,\varphi\in L^2(0,1).
\end{equation}
Using the Dominated Convergence Theorem, we deduce that \eqref{eq:gammaConv} also holds true if we replace $\gamma_*$ by $\gamma^*$. Thus $\gamma_n\rightharpoonup\gamma_*$ in $L^2(0,1)$ and therefore $\gamma_*=\gamma^*$. Now we observe that the set $A=\{\gamma\in H_0^1(0,1)\,:\,|\gamma'|\leq L\,\text{ a.e. in }(0,1)\}$ is convex and closed, so it is also weakly closed. Hence, since $\gamma_n\rightharpoonup\gamma^*$ in $H_0^1(0,1)$, we get that $\gamma^*\in A$. Therefore, $\Omega^*=\Omega_{\gamma^*}\in\mathcal{O}$.

Let $\varphi\in C_c^\infty(\Omega_{\gamma^*})$. Since $\Omega_{\gamma_n}\to\Omega_{\gamma^*}$ in the Hausdorff complementary metric, we observe that there exists $n^\ast$ such that $\overline{\mathrm{supp}\,\varphi}\subset\Omega_{\gamma_n}$ for every $n\geq n^\ast$; see \cite[Proposition A3.8, p.465]{NeSpTi2006}. Then $\varphi\in C_c^\infty(\Omega_{\gamma_n})$ for every $n\geq n^\ast$ and so
\begin{equation}\label{eq:Basic-Tn}
b_{2,\Omega_{\gamma_n}}(\vv_n,T_n,\varphi)+\frac{1}{\Rey\Pra}(\nabla T_n,\nabla\varphi)_{2,\Omega_{\gamma_n}}=(g_2,\varphi)_{2,\Omega_{\gamma_n}}\qquad\forall\,n\geq n^\ast.
\end{equation}
Let $(\vv^\ast,\hat{T}^\ast)\in H^1(\Omega_{\gamma^*};\mathbb{R}^2)\times H^1(\Omega_{\gamma^*})$ be the restriction of $(\vv^{\mathrm{M}},\hat{T}^{\mathrm{M}})$ to $\Omega_{\gamma^*}$. Since $\vv^{\mathrm{M}}_n\rightharpoonup\vv^{\mathrm{M}}$ in $H^1(\Omega^\mathrm{M};\mathbb{R}^2)$ and the embedding $H^1(\Omega^\mathrm{M})\hookrightarrow L^2(\Omega^\mathrm{M})$ is compact, we have that 
$\vv^{\mathrm{M}}_n\to\vv^{\mathrm{M}}$ in $L^2(\Omega^\mathrm{M};\mathbb{R}^2)$. In addition, using the Dominated Convergence Theorem, we deduce that $\chi_{\Omega_{\gamma_n}}\to\chi_{\Omega^*}$ in $L^p(\Omega^{\mathrm{M}})$ for any $1\leq p<\infty$. In particular, $\chi_{\Omega_{\gamma_n}}\to\chi_{\Omega^*}$ in $L^4(\Omega^{\mathrm{M}})$ and hence $\vv_n^{\mathrm{M}}\chi_{\Omega_{\gamma_n}}\to\vv_n^{\mathrm{M}}\chi_{\Omega^*}$ in $L^2(\Omega^{\mathrm{M}};\mathbb{R}^2)$. We finally notice that if $\varphi^{\mathrm{M}}$ denotes the extension by zero of $\varphi$ outside $\Omega_{\gamma^*}$, then $\varphi^{\mathrm{M}}\in L^\infty(\Omega^{\mathrm{M}})$ and therefore $\vv^{\mathrm{M}}_n\varphi^{\mathrm{M}}\chi_{\Omega_{\gamma_n}}\to\vv^{\mathrm{M}}\varphi^{\mathrm{M}}\chi_{\Omega^*}$ in $L^2(\Omega^\mathrm{M};\mathbb{R}^2)$. Also, from the weak convergence $T^{\mathrm{M}}_n\rightharpoonup T^{\mathrm{M}}$ in $H^1(\Omega^\mathrm{M})$ we find that $\nabla T^{\mathrm{M}}_n\rightharpoonup\nabla T^{\mathrm{M}}$ in $L^2(\Omega^\mathrm{M};\mathbb{R}^2)$. Then,
\begin{align*}
&\lim_{n\to\infty}b_{2,\Omega_{\gamma_n}}(\vv_n,T_n,\varphi)=
\lim_{n\to\infty}\int_{\Omega_{\gamma_n}}(\vv_n\cdot\nabla T_n)\varphi\,dx=
\lim_{n\to\infty}\int_{\Omega^\mathrm{M}}(\vv^{\mathrm{M}}_n\varphi^{\mathrm{M}}\chi_{\Omega_{\gamma_n}})\cdot\nabla T^{\mathrm{M}}_n\,dx\\
&\qquad=\int_{\Omega^{\mathrm{M}}}(\vv^{\mathrm{M}}\varphi^{\mathrm{M}}\chi_{\Omega^*})\cdot\nabla T^\mathrm{M}\,dx=\int_{\Omega_{\gamma^*}}(\vv^\ast\cdot\nabla T^\ast)\varphi\,dx=b_{2,\Omega_{\gamma^*}}(\vv^\ast,T^\ast,\varphi),
\end{align*}
where $T^*=\hat{T}^*+T_d$. In a similar way, we obtain that 
\begin{align*}
\lim_{n\to\infty}(\nabla T_n,\nabla\varphi_n)_{2,\Omega_{\gamma_n}}=(\nabla T^\ast,\nabla\varphi)_{2,\Omega_{\gamma^*}}.
\end{align*}
In fact, from the weak convergence $\hat{T}_n^{\mathrm{M}}\rightharpoonup\hat{T}^{\mathrm{M}}$ in $H^1(\Omega^\mathrm{M})$ we have that $\nabla\hat{T}_n^{\mathrm{M}}\rightharpoonup\nabla\hat{T}^\mathrm{M}$ in $L^2(\Omega^\mathrm{M};\mathbb{R}^2)$. Also, by the Dominated Convergence Theorem, we deduce that $\chi_{\Omega_{\gamma_n}}\nabla\varphi^{\mathrm{M}}\to\chi_{\Omega_{\gamma^*}}\nabla\varphi^{\mathrm{M}}$ in $L^2(\Omega^\mathrm{M};\mathbb{R}^2)$. Then, as before, we obtain the desired convergence. Finally, using the Dominated Convergence Theorem once again, we deduce that
\begin{align*}
\lim_{n\to\infty}(g_2,\varphi)_{2,\Omega_{\gamma_n}}=\,&
\lim_{n\to\infty}\int_{\Omega_{\gamma_n}}g_2\varphi\,dx=
\lim_{n\to\infty}\int_{\Omega^\mathrm{M}}g_2\varphi^{\mathrm{M}}\chi_{\Omega_{\gamma_n}}\,dx\\
=\,&\int_{\Omega^{\mathrm{M}}}g_2\varphi^{\mathrm{M}}\chi_{\Omega^*}\,dx=\int_{\Omega_{\gamma^*}}g_2\varphi\,dx=(g_2,\varphi)_{2,\Omega_{\gamma^*}}.
\end{align*}
Then, taking the limit as $n\to\infty$ in \eqref{eq:Basic-Tn}, we obtain
\begin{equation}\label{eq:Basic-T*}
b_{2,\Omega_{\gamma^*}}(\vv^\ast,T^\ast,\varphi)+\frac{1}{\Rey\Pra}(\nabla T^\ast,\nabla\varphi)_{2,\Omega_{\gamma^*}}=(g_2,\varphi)_{2,\Omega_{\gamma^*}}.
\end{equation}
By a density argument, we deduce that \eqref{eq:Basic-T*} also holds for every $\varphi\in H^1_0(\Omega_{\gamma^*})$. In a similar way we obtain 
\begin{equation}\label{eq:Basic-v*}
b_{1,\Omega_{\gamma^*}}(\vv^\ast,\vv^\ast,{\vvarphi})+\frac{1}{\Rey}(\nabla \vv^\ast,\nabla{\boldsymbol\varphi})_{2,\Omega_{\gamma^*}}-\frac{\Gra}{\Rey^2}(T^*\mathbf{e},\vvarphi)_2=(\boldg_1,\vvarphi)_{2,\Omega_{\gamma^*}},
\end{equation}
for every $\vvarphi\in V(\Omega_{\gamma^*};\mathbb{R}^2)$.

We shall now prove that $T^\ast=T_d$ on $\partial\Omega_{\gamma^\ast}$, $\vv^\ast=0$ on $\partial\Omega_{\gamma^*}$, and that $\diver\vv^\ast=0$ a.e. in $\Omega_{\gamma^*}$. This, together with \eqref{eq:Basic-T*} and \eqref{eq:Basic-v*}, will give us that $(\vv^\ast,T^\ast)$ is the weak solution to \eqref{v-edp}-\eqref{T-Walls} in $\Omega_{\gamma^*}$. 

Given $\Omega_\gamma\in\mathcal{O}$, we define $\Gamma=\partial\Omega_\gamma\setminus\Gamma_{\gamma}$. Notice that $\Gamma$ does not depend on $\gamma$ and $\Gamma\subset\partial\Omega^{\mathrm{M}}$. In particular, this yields that $T_n=0$ on $\Gamma$ for every $n$. From this, and the compact embedding $H^1(\Omega^\mathrm{M})\hookrightarrow L^2(\Gamma)$, we get that $T^*=0$ on $\Gamma$. We shall now prove that $\hat{T}^*=0$ on $\Gamma_{\gamma^*}$, which will give us that $T^*=T_d$ on $\partial\Omega_{\gamma^*}$. Exploiting again the compact embedding $H^1(\Omega^\mathrm{M})\hookrightarrow L^2(\Gamma)$, we get  that $\hat{T}_n^{\mathrm{M}}\to\hat{T}^{\mathrm{M}}$ in $L^2(\Omega^{\mathrm{M}})$ since $\hat{T}_n^{\mathrm{M}}\rightharpoonup\hat{T}^{\mathrm{M}}$ in $H^1(\Omega^{\mathrm{M}})$. Consider an open set $O$ such that $\bar{O}\subset\Omega^{\mathrm{M}}\setminus\Omega_{\gamma^*}$. Similarly as before, from the convergence $\Omega_{\gamma_n}\to\Omega_{\gamma^*}$ in the Hausdorff complementary metric, we get that there exists $n^*$ such that $\bar{O}\subset\Omega^{\mathrm{M}}\setminus\Omega_{\gamma_n}$ for every $n\geq n^*$; see \cite[Proposition A3.8, p.465]{NeSpTi2006}. Thus, $\hat{T}_n^{\mathrm{M}}=0$ a.e. in $\bar{O}$ and therefore $\hat{T}^{\mathrm{M}}=0$ in $\bar{O}$. Given that $O$ is arbitrary, we deduce that $\hat{T}^{\mathrm{M}}=0$ in $\Omega^{\mathrm{M}}\setminus\Omega_{\gamma^*}$. Therefore, $\hat{T}^*=0$ on $\partial\Omega_{\gamma^*}$. The proof that $\vv^*=0$ on $\partial\Omega_{\gamma^*}$ is obtained in a similar way. We shall finally prove that $\diver\vv^\ast=0$ a.e. in $\Omega_{\gamma^*}$, by deducing that 
\begin{equation*}
\int_{\Omega_{\gamma^*}}\varphi\,\diver\vv^\ast\,dx=0\qquad\forall\,\varphi\in C_c^\infty(\Omega_{\gamma^*}).
\end{equation*}
Let $\varphi\in C_c^\infty(\Omega_{\gamma^*})$ and let $\varphi^{\mathrm{M}}$ be its extension by zero outside $\Omega_{\gamma^*}$. Similarly as before, we notice that there exists $n^\ast$ such that $\overline{\supp\varphi}\subset\Omega_{\gamma_n}$ for every $n\geq n^\ast$; see \cite[Proposition A3.8, p.465]{NeSpTi2006} and hence the restriction of $\varphi$ to $\Omega_{\gamma_n}$ makes sense for every $n\geq n^*$. Since
\begin{equation*}
0=\int_{\Omega_{\gamma_n}}\varphi\,\diver\vv_n\,dx=
\int_{\Omega^\mathrm{M}}\varphi^{\mathrm{M}}\,\diver \vv_n^{\mathrm{M}}\chi_{\Omega_{\gamma_n}}\,dx\qquad\forall\,n\geq n^*,
\end{equation*}
it is enough to show that
\begin{equation*}
\lim_{n\to\infty}\int_{\Omega^\mathrm{M}}\varphi^{\mathrm{M}}\,\diver\vv^{\mathrm{M}}_n\chi_{\Omega_{\gamma_n}}\,dx= \int_{\Omega^*}\varphi\,\diver\vv^*\,dx.
\end{equation*}
By the Dominated Convergence Theorem, we deduce that $\varphi^{\mathrm{M}}\chi_{\Omega_{\gamma_n}}\to\varphi^{\mathrm{M}}\chi_{\Omega^*}$ in $L^2(\Omega^{\mathrm{M}})$. Also, since $\vv_n^{\mathrm{M}}\rightharpoonup \vv^{\mathrm{M}}$ in $H^1(\Omega^{\mathrm{M}};\mathbb{R}^2)$, we observe that $\diver \vv_n^{\mathrm{M}}\rightharpoonup\diver \vv^{\mathrm{M}}$ in $L^2(\Omega^{\mathrm{M}})$, and therefore we obtain the desired convergence.

Finally, the above convergences for $\{T_n\}$, $\{\Omega_{\gamma_n}\}$, and $\{\gamma_n\}$ imply that $\Omega_{\gamma^*}$ is an optimal domain by \eqref{eq:InfSeq} and
\begin{equation}\label{eq:LowerSemicont}
J(T^\ast,\gamma^*)\leq\liminf_{n\to\infty}J(T_n,\gamma_n).
\end{equation}
\end{proof}

\section{The formal adjoint system and its regularity}\label{Sect:Adjoint}
This section is devoted to the obtention of an adjoint system and the study of its regularity properties, specially at the boundary. Hereafter, we consider $\boldg_1\in L^2(\Omega^{\mathrm{M}};\mathbb{R}^2)$, $g_2\in L^2(\Omega^{\mathrm{M}})$, and $T_d\in H^1(\Omega^{\mathrm{M}})\cap C_c(\Omega^{\mathrm{M}})$. Also, we assume that $\Rey$, $\Pra$, $\Gra$, $\boldg_1$, $g_2$, and $T_d$ are under the assumptions of Theorem \ref{Thm:weak-EU}. For convenience, we rewrite \eqref{shape} as
\begin{align*}
&\min{J(\vv,\hat{T},\gamma)}\quad\text{ over }\quad(\vv,\hat{T},\gamma)\in V(\Omega_\gamma)\times H_0^1(\Omega_{\gamma})\times H_0^1(0,1),\\
&\text{subject to }\quad e(\vv,\hat{T},\gamma)=0\quad \text{and}\quad\gamma\in H_{\nu,L}.
\end{align*}
Here, 
\begin{align}\label{eq:H}
H_{\nu,L}:=\{\gamma\in H_0^1(0,1)\,:\,\gamma\leq 1-\nu\text{ a.e. in }(0,1)\text{ and }|\gamma'|\leq L\text{ a.e. in }(0,1)\}.
\end{align}
Recall that the objective functional $J$ is defined by 
\begin{align}
\label{J-Def-bis}J(\vv,\hat{T},\gamma)=\,&\|T-I(T,\gamma)\|^2_{L^2(\Omega_\gamma)}+
\frac{\lambda_1}{2}\|\gamma\,'\|^2_{L^2(0,1)}+\frac{\lambda_2}{2}\left(\int_0^1\gamma(\xi)\,d\xi\right)^2,
\end{align}
where $T=\hat{T}+T_d$. Finally, the map $e:V(\Omega_\gamma)\times H_0^1(\Omega_{\gamma})\times H_0^1(0,1)\to V(\Omega_\gamma)'\times H^{-1}(\Omega_{\gamma})$ is given by
\begin{equation*}
\langle e(\vv,\hat{T},\gamma),(\vvarphi,\varphi)\rangle=
\langle e_1(\vv,\hat{T},\gamma),\vvarphi\rangle+
\langle e_2(\vv,\hat{T},\gamma),\varphi\rangle,
\end{equation*}
where 
\begin{align*}
\langle e_1(\vv,\hat{T},\gamma),\vvarphi\rangle:=\,&\int_{\Omega_\gamma}(\vv\cdot\nabla\vv)\cdot\vvarphi\,dx+\frac{1}{\Rey}\int_{\Omega_\gamma}\nabla\vv\cdot\nabla\vvarphi\,dx-\frac{\Gra}{\Rey^2}\int_{\Omega_\gamma}T\mathbf{e}\cdot\vvarphi\,dx\\
&-\int_{\Omega_\gamma}\boldg_1\cdot\vvarphi\,dx,\\
\langle e_2(\vv,\hat{T},\gamma),\varphi\rangle:=\,&\int_{\Omega_\gamma}(\vv\cdot\nabla T)\varphi\,dx+\frac{1}{\Rey\Pra}\int_{\Omega_\gamma}\nabla T\cdot\nabla\varphi\,dx-\int_{\Omega\gamma}g_2\varphi\,dx.
\end{align*}

Consider the reduced objective functional $\hat{J}:H_0^1(0,1)\to \mathbb{R}$, defined by 
\begin{equation*}
\hat{J}(\gamma)=J(\vv(\gamma),\hat{T}(\gamma),\gamma),
\end{equation*}
where $e(\vv(\gamma),\hat{T}(\gamma),\gamma)=0$. We now proceed formally for the derivation of the adjoint system; this step is later made rigorous on the next section. Suppose that the map $\gamma\mapsto \hat{J}(\gamma)$ is differentiable and that we can follow the adjoint state formalism (see \cite{HiPiUlUl2008} for example): That is that $\hat{J}'(\gamma)$ admits the representation 
\begin{align}\label{eq:AbstractJDerivative}
\hat{J}'(\gamma)=\,&
e_\gamma(\vv(\gamma),\hat{T}(\gamma),\gamma)^*(\ww,S)+J_\gamma(\vv(\gamma),\hat{T}(\gamma),\gamma),
\end{align}
where $(\ww,S)$ satisfies the adjoint equation
\begin{align}\label{eq:AbstractAdjointEq}
A(\vv(\gamma),\hat{T}(\gamma))^{*}(\ww,S)=
-(J_{\vv}(\vv(\gamma),\hat{T}(\gamma),\gamma),
J_{\hat{T}}(\vv(\gamma),\hat{T}(\gamma),\gamma)),
\end{align}
and 
\begin{align*}
A(\vv(\gamma),\hat{T}(\gamma)):=(e_{\vv}(\vv(\gamma),\hat{T}(\gamma),\gamma),
e_{\hat{T}}(\vv(\gamma),\hat{T}(\gamma),\gamma)).
\end{align*}
The next two lemmae address structure and properties of this adjoint system. More precisely, Lemma \ref{Lem:AdjointSystem} and Lemma \ref{Lem:Adjoint} establish basic existence of adjoint equation together with uniqueness and regularity. 
\begin{lemma}\label{Lem:AdjointSystem}
Let $\Omega_\gamma\in\mathcal{O}$. Consider $\boldg_1\in L^2(\Omega^{\mathrm{M}};\mathbb{R}^2)$, $g_2\in L^2(\Omega^{\mathrm{M}})$, and $T_d\in H^1(\Omega^{\mathrm{M}})\cap C_c(\Omega^{\mathrm{M}})$. Assume that $\Rey$, $\Pra$, $\Gra$ satisfy \eqref{cond:RePrGr} and $\boldg_1$, $g_2$, $T_d$ satisfy the smallness condition in Theorem \ref{Thm:weak-EU}. Then the  adjoint system associated to \eqref{v-edp}-\eqref{T-Walls} in $\Omega_\gamma$ defined by \eqref{eq:AbstractAdjointEq} is given by
\begin{align}
\label{w-edp}-\vv\cdot\nabla \ww+\ww\cdot\nabla\vv +S\,\nabla T-\frac{1}{\Rey}\Delta \ww+\nabla q=\,&0&\text{in}&\quad\Omega_\gamma,&\\
\label{w-DivFree}\nabla\cdot\ww=\,&0&\text{in}&\quad\Omega_\gamma,&\\
\label{w-NoSlip}\ww=\,&0&\quad\text{on}&\quad\partial\Omega_\gamma,
\end{align}
and 
\begin{align}
\label{S-edp}-\vv\cdot\nabla S -\frac{1}{\Rey\Pra}\Delta S-\frac{\Gra}{\Rey^2}\ww\cdot{\bf e}=\,&T-I(T,\gamma)&\text{in}&\quad\Omega_\gamma,&\\
\label{S-Dirichlet}S=\,&0&\text{on}&\quad\partial\Omega_{\gamma},&
\end{align}
where $(\vv,T)$ is the weak solution to \eqref{v-edp}-\eqref{T-Walls} in $\Omega_\gamma$ given by Theorem \ref{Thm:weak-EU}. 
\end{lemma}

\begin{proof}Allthrough the proof, $(\vv,\hat{T})$ denotes a generic element of $V(\Omega_\gamma)\times H_0^1(\Omega_\gamma)$ and $(\vv(\gamma),T(\gamma))$ refers to the weak solution to \eqref{v-edp}-\eqref{T-Walls}, where $T(\gamma)=\hat{T}(\gamma)-T_d$. We first compute the right-hand side of \eqref{eq:AbstractAdjointEq}. Since $J$ does not depend on $\vv$, we notice that $J_{\vv}=0$. Also, by the chain rule and standard calculations, we find that
\begin{align*}
\langle J_{\hat{T}}(\vv,\hat{T},\gamma),\varphi\rangle=\,&\int_{\Omega_\gamma}\left(T-I(T,\gamma)\right)\left(\varphi-I(\varphi,\gamma)\right)\,dx=\int_{\Omega_\gamma}\left(T-I(T,\gamma)\right)\varphi\,dx.
\end{align*}

We now focus on the computation of the differential operator $A$. For this, we observe that
\begin{align*}
\langle e_{1,\vv}(\vv,\hat{T},\gamma)\vvarphi,\ww)\rangle=\,&
\frac{1}{\Rey}\int_{\Omega_\gamma}\nabla\vvarphi\cdot\nabla\ww\,dx+
\int_{\Omega_\gamma}(\vv\cdot\nabla\vvarphi)\cdot\ww\,dx+
\int_{\Omega_\gamma}(\vvarphi\cdot\nabla\vv)\cdot\ww\,dx,\\[0.15cm]
\langle e_{1,\hat{T}}(\vv,\hat{T},\gamma)\varphi,\ww)\rangle=\,&
-\frac{\Gra}{\Rey^2}\int_{\Omega_\gamma}\varphi\,{\mathbf e}\cdot\ww\,dx,\\[0.15cm]
\langle e_{2,\vv}(\vv,\hat{T},\gamma)\vvarphi,S\rangle=\,&
\int_{\Omega_\gamma}(\vvarphi\cdot\nabla\hat{T})S\,dx+\int_{\Omega_\gamma}(\vvarphi\cdot\nabla T_d)S\,dx,\\[0.15cm]
\langle e_{2,\hat{T}}(\vv,\hat{T},\gamma)\varphi,S\rangle=\,&
\int_{\Omega_\gamma}(\vv\cdot\nabla \varphi) S\,dx+\frac{1}{\Rey\Pra}\int_{\Omega_\gamma}\nabla \varphi\cdot\nabla S\,dx.
\end{align*}

Then,
\begin{align*}
&\langle A(\vv(\gamma),\hat{T}(\gamma))(\vvarphi,\varphi),(\ww,S)\rangle=
b_1(\vv(\gamma),\vvarphi,\ww)+
b_1(\vvarphi,\vv(\gamma),\ww)+b_2(\vvarphi,\hat{T}(\gamma),S)\\
&\quad+b_2(\vvarphi,T_d,S)+\frac{1}{\Rey}(\nabla\ww,\nabla\vvarphi)_2+b_2(\vv(\gamma),\varphi,S)+\frac{1}{\Rey\Pra}(\nabla S,\nabla \varphi)_2
-\frac{\Gra}{\Rey^2}(\ww,\varphi\,{\mathbf e})_2.
\end{align*}

Since $\vv(\gamma)$ is divergence-free in $\Omega_\gamma$, and $\vv(\gamma),\vvarphi,\ww\in H^1(\Omega_\gamma;\mathbb{R}^2)$, we notice that $b_1(\vv(\gamma),\vvarphi,\ww)=-b_1(\vv(\gamma),\ww,\vvarphi)$; see \cite[Lemma IX.2.1, p. 591]{Ga2011}. Also, $b_1(\vvarphi,\vv(\gamma),\ww)=b_1(\ww,\vv(\gamma),\vvarphi)$. Similarly, $b_2(\vv(\gamma),\varphi,S)=-b_2(\vv(\gamma),S,\varphi)$.
Then,
\begin{align*}
&\langle A(\vv(\gamma),\hat{T}(\gamma))(\vvarphi,\varphi),(\ww,S)\rangle=
-b_1(\vv(\gamma),\ww,\vvarphi)+
b_1(\ww,\vv(\gamma),\vvarphi)+
b_2(\vvarphi,T(\gamma),S)\\
&\quad+
\frac{1}{\Rey}(\nabla\ww,\nabla\vvarphi)_2-b_2(\vv(\gamma),S,\varphi)+\frac{1}{\Rey\Pra}(\nabla S,\nabla \varphi)_2-\frac{\Gra}{\Rey^2}(\ww,\varphi\,{\mathbf e})_2.
\end{align*}
Thus, the weak form of the adjoint system is stated as follows: Find $(\ww,S)\in V(\Omega_\gamma)\times H_0^1(\Omega_\gamma)$ that satisfies
\begin{align}
\label{eq:weak-w}&-b_1(\vv(\gamma),\ww,\vvarphi)+b_1(\ww,\vv(\gamma),\vvarphi)+b_2(\vvarphi,T(\gamma),S)+\frac{1}{\Rey}(\nabla\ww,\nabla\vvarphi)_2=0,&\\
\label{eq:weak-S}&b_2(\vv(\gamma),S,\varphi)-\frac{1}{\Rey\Pra}(\nabla S,\nabla\varphi)_2+\frac{\Gra}{\Rey^2}(\ww,\varphi\,{\mathbf{e}})_2=-\langle J_{\hat{T}}(\vv(\gamma),\hat{T}(\gamma),\gamma),\varphi\rangle,
\end{align}
for every $(\vvarphi,\varphi)\in V(\Omega_\gamma)\times H_0^1(\Omega_\gamma)$. Only remains to observe that this is just the weak form of \eqref{w-edp}-\eqref{S-Dirichlet}.
\end{proof}

We now focus on the existence and uniqueness of a weak solution to the adjoint system \eqref{w-edp}-\eqref{S-Dirichlet}. As before, we shall prove it by a fixed point argument. One key step of the proof concerns the existence of the uniform bound (independent of $\gamma$) for weak solutions to the primal Boussinesq system given by Lemma \ref{Lem:UniformBound}. More precisely, we have the next result.

\begin{lemma}[Existence, uniqueness and boundary regularity of the weak solution to the adjoint system]\label{Lem:Adjoint}
Let $\Omega_\gamma\in\mathcal{O}$, and $\boldg_1\in L^2(\Omega^{\mathrm{M}};\mathbb{R}^2)$, $g_2\in L^2(\Omega^{\mathrm{M}})$, and $T_d\in H^1(\Omega^{\mathrm{M}})\cap C_c(\Omega^{\mathrm{M}})$. Then there exist $\varepsilon_1,\varepsilon_2,\varepsilon_3>0$ such that if
\begin{equation}\label{cond:RePrGr-bis}
\Rey\in(0,\varepsilon_1),\qquad
\Pra\in\left(0,\frac{\varepsilon_2}{\Rey}\right),\qquad
\Gra\in(0,\varepsilon_3\min(\Rey,\Rey^2)),
\end{equation}
and $\boldg_1$, $g_2$, and $T_d$ are sufficiently small, then there exists a unique weak solution $(\ww,S)\in V(\Omega_\gamma)\times H_0^1(\Omega_\gamma)$ to \eqref{w-edp}-\eqref{S-Dirichlet} in $\Omega_\gamma$. Moreover, if in addition $T_d\in H^2(\Omega^{\mathrm{M}})$ and the boundary part $\Gamma_\gamma$ is of class $C^2$, then the weak solution $(\ww,S)\in V(\Omega_\gamma)\times H^1_0(\Omega_\gamma)$ to \eqref{w-edp}-\eqref{S-Dirichlet} satisfies
\begin{align}\label{eq:w-S-bdry-reg}
\ww\big|_{\Gamma_\gamma^\varepsilon}\in H^{3/2}(\Gamma_\gamma^\varepsilon;\mathbb{R}^2),\qquad\qquad
S\big|_{\Gamma_\gamma^\varepsilon}\in H^{3/2}(\Gamma_\gamma^{\varepsilon}),
\end{align}  
for every relative open subset $\Gamma_\gamma^\varepsilon$ of $\partial\Omega$ strictly contained in $\Gamma_\gamma$. 
\end{lemma}

\begin{proof}
Similarly as in the proof of Theorem \ref{Thm:weak-EU}, we observe that $(\ww,S)$ a solution to \eqref{eq:weak-w}-\eqref{eq:weak-S} if and only if $(\ww,S)$ is a fixed point of $G$, where $G:V(\Omega_\gamma)\times H_0^1(\Omega_\gamma)\to V(\Omega_\gamma)\times H_0^1(\Omega_\gamma)$ is defined as
\begin{equation}\label{eq:G}
G=(R_1^{-1}Q_1,R_2^{-1}Q_2).
\end{equation}
Here, $R_1:V(\Omega_\gamma)\to V(\Omega_\gamma)'$ and $R_2:H_0^1(\Omega_\gamma)\to H^{-1}(\Omega_\gamma)$ are given as in \eqref{eq:R1R2}.
Also, $Q_1:V(\Omega_\gamma)\times H_0^1(\Omega_\gamma)\to V(\Omega_\gamma)'$ and $Q_2:V(\Omega_\gamma)\times H_0^1(\Omega_\gamma)\to H^{-1}(\Omega_\gamma)$ are given by
\begin{align}
\label{eq:Q1}\langle Q_1(\ww,S),\vvarphi\rangle=\,&b_1(\vv(\gamma),\ww,\vvarphi)-b_1(\ww,\vv(\gamma),\vvarphi)-b(T(\gamma),S,\vvarphi),&\\
\label{eq:Q2}\langle Q_2(\ww,S),\varphi\rangle=\,&(g(\gamma),\varphi)_2+b_2(\vv(\gamma),S,\varphi)+\frac{\Gra}{\Rey^2}(\ww,\varphi\,{\mathbf{e}})_2.
\end{align}
The rest of the proof runs similarly as for Theorem \ref{Thm:weak-EU}. We give some details below. 

For $(\ww_1,S_1), (\ww_2,S_2)\in V(\Omega_\gamma)\times H_0^1(\Omega_\gamma)$, we have
\begin{align*}
&\|G(\ww_1,S_1)-G(\ww_2,S_2)\|_{V(\Omega_\gamma)\times H_0^1(\Omega_\gamma)}\\
&\qquad\qquad\leq 
C\,\Rey\left(\|\vv(\gamma)\|_{V(\Omega_\gamma)}+\Pra\frac{\Gra}{\Rey^2}\right)\|\ww_1-\ww_2\|_{V(\Omega_\gamma)}\\
&\qquad\qquad\quad +C\,\Rey\left(\Pra\|\vv(\gamma)\|_{V(\Omega_\gamma}+\|T(\gamma)\|_{H^1(\Omega_\gamma)}\right)\|S_1-S_2\|_{H_0^1(\Omega_\gamma)}.
\end{align*}
Also, we know from Theorem \ref{Thm:weak-EU} that 
\begin{equation*}
\|\vv(\gamma)\|_{V(\Omega_\gamma)}+\|\hat{T}(\gamma)\|_{H_0^1(\Omega_\gamma)}\leq\tau,
\end{equation*}
for some $\tau>0$, provided that $\Rey$, $\Pra$, $\Gra$ satisfy \eqref{cond:RePrGr-bis} for some $\varepsilon_1$, $\varepsilon_2$, $\varepsilon_3>0$ and $\boldg_1$, $g_2$, and $T_d$ are sufficiently small. Then, 
\begin{equation*}
\|\vv(\gamma)\|_{V(\Omega_\gamma)}+\|T(\gamma)\|_{H^1(\Omega_\gamma)}\leq C\tau+\|T_d\|_{H^1(\Omega_\gamma)},
\end{equation*}
Thus,
\begin{align}\label{eq:Bound-Contraction}
\|G(\ww_1,S_1)-G(\ww_2,S_2)\|_{V(\Omega_\gamma)\times H_0^1(\Omega_\gamma)}
\leq \kappa_1\|(\ww_1,S_1)-(\ww_2,S_2)\|_{V(\Omega_\gamma)\times H_0^1(\Omega_\gamma)},
\end{align}
where
\begin{align*}
\kappa_1:=C\,\Rey\left(\tau+\tau\,\Pra+\|T_d\|_{H^1(\Omega^\mathrm{M})}+\Pra\|T_d\|_{H^1(\Omega^\mathrm{M})}+\Pra\frac{\Gra}{\Rey^2}\right).
\end{align*}

In an analogous manner, we find that
\begin{align}\label{eq:Bound-Ball}
\|G(\ww,S)\|_{V(\Omega_\gamma)\times H_0^1(\Omega_\gamma)}\leq \kappa_2,
\end{align}
for every $(\ww,S)\in\bar{B}_\rho$, where
\begin{align*}
\kappa_2:=\,&C\,\Rey\left(\tau+\tau\,\Pra+\Pra\frac{\Gra}{\Rey^2}+\|T_d\|_{H^1(\Omega^\mathrm{M})}+\Pra\|T_d\|_{H^1(\Omega^\mathrm{M})}\right)\rho\\
&+C\,\Rey\Pra\,(\tau+\|T_d\|_{H^1(\Omega^\mathrm{M})}),
\end{align*} 
and
\begin{equation*}
\bar{B}_\rho:=\{(\ww,S)\in V(\Omega_\gamma)\times H_0^1(\Omega_\gamma)\,:\,\|\ww\|_{V(\Omega_\gamma)}+\|S\|_{H_0^1(\Omega_\gamma)}\leq\rho\},
\end{equation*}
for $\rho>0$.

It follows from \eqref{eq:Bound-Contraction} and \eqref{eq:Bound-Ball} that $G$ is a contraction from $\bar{B}_\rho$ into itself provided 
\begin{align}\label{eq:Cond-FixedPoint}
\kappa_1<1\qquad\text{and}\qquad\kappa_2\leq\rho,
\end{align}
which can be achieved by an appropriate selection of $\varepsilon_1$, $\varepsilon_2$, $\varepsilon_3$, and sufficiently small data $\boldg_1$, $g_2$, $T_d$. Therefore, $G$ admits a unique fixed point in $\bar{B}_\rho$. The boundary regularity of the weak solution $(\ww,S)$ given by \eqref{eq:w-S-bdry-reg} is obtained analogously as for $(\vv,T)$ in Theorem \ref{Thm:BdryReg}. 
\end{proof}

\section{Differentiability properties}\label{Sec:OptmalityCondition}

In this section we establish a differentiability property of the reduced cost functional $\hat{J}(\gamma)$ by considering specific perturbations of a reference domain $\Omega_\gamma$, and derive a first-order optimality condition to the shape optimization problem \eqref{shape}. More precisely, we shall compute the Gateaux derivative of $\hat{J}(\gamma)$ through formula \eqref{eq:AbstractJDerivative}. To this end, we must carefully define the set of admissible domains since it is closely related to the boundary regularity properties of the weak solutions to the primal and adjoint systems \eqref{v-edp}-\eqref{T-Walls} and \eqref{w-edp}-\eqref{S-Dirichlet}. Specifically, the computation of $e_\gamma(\vv(\gamma),\hat{T}(\gamma),\gamma)$ requires $H^{3/2}$ regularity of the weak solution to the aforementioned fluid problems on certain domain boundary part, as we shall see below. By Theorem \ref{Thm:BdryReg} and Lemma \ref{Lem:AdjointSystem}, this regularity is achieved if we focus on the following subclass of the set of admissible domains $\mathcal{O}$. 

\begin{figure}[ht!]
\centering
\includegraphics[scale=.3]{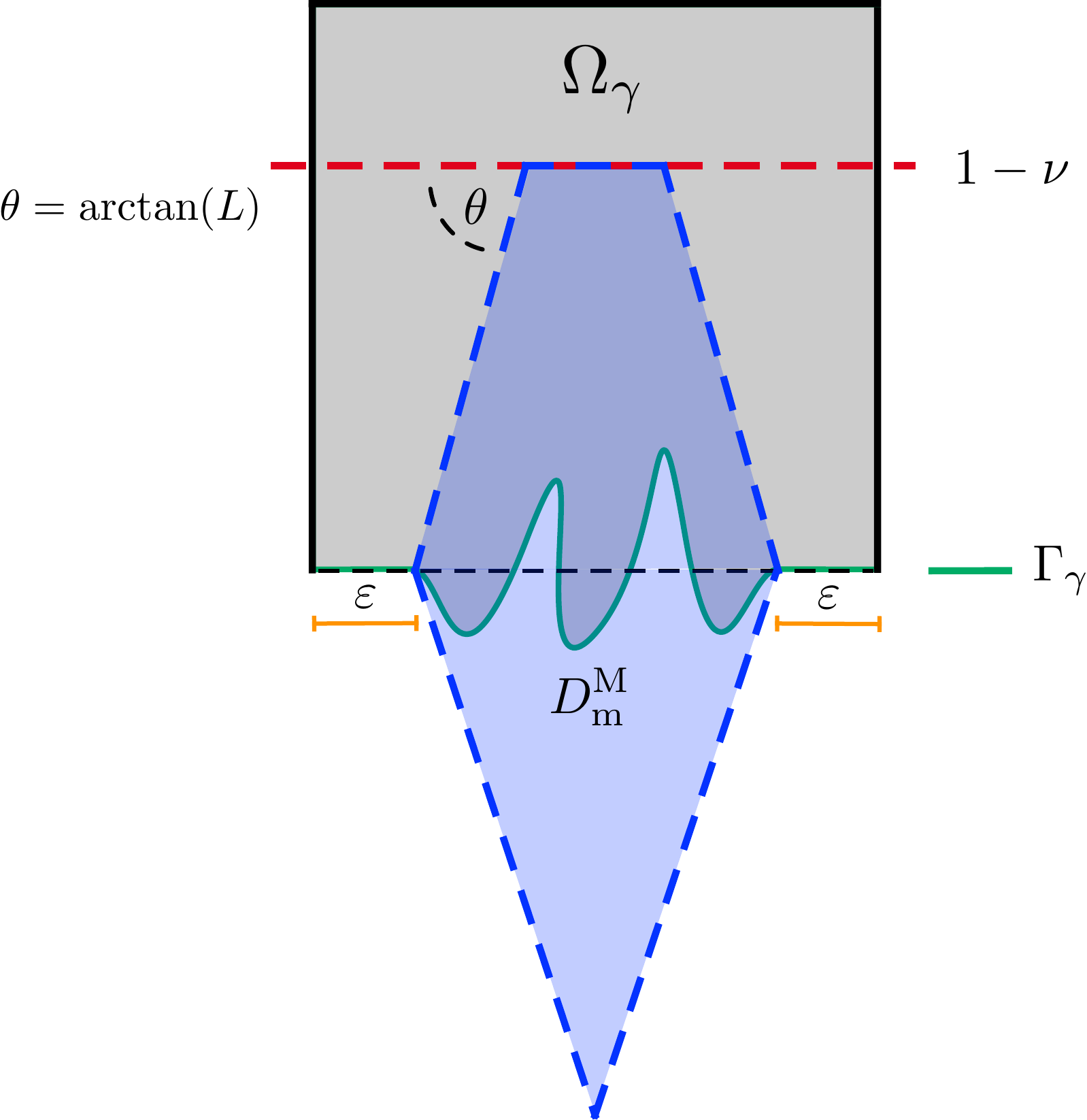}
\label{Fig:Admissible2}
\caption{Admissible domain $\Omega_\gamma\in \mathcal{O}^\varepsilon$ in gray, $\Gamma_\gamma$ in {\color{ForestGreen}green}, where $1-\nu$ is represented by the dashed line in {\color{red}red} and the dashed line in {\color{blue}blue} is the boundary of $D^{\mathrm{M}}_{\mathrm{m}}$ where changes of the boundary are located, $\Gamma_\gamma$ might be located and induced by the obstacle constraint $1-\nu$ and derivative bound $L$.}
\end{figure}

\begin{definition}[$\varepsilon$-admissible domains]
Let $L>0$ and $\varepsilon,\nu\in(0,1)$. We say that $\Omega^\varepsilon\subset\mathbb{R}^2$ is an $\varepsilon$-admissible domain if 
\begin{equation}\label{eq:EpsAdmissDomain}
\Omega^\varepsilon=\{(x_1,x_2)\in\mathbb{R}^2\,:\,0< x_1< 1,\,\gamma(x_1)< x_2< 1\},
\end{equation}
for some $\gamma\in H_{\nu,L}\cap V_\varepsilon$, where $H_{\nu,L}$ was defined in \eqref{eq:H} and
\begin{multline*}
V_\varepsilon:=\{\gamma\in H^3(0,1)\,:\,\gamma^{(\alpha)}=0\text{ a.e. in }(0,\varepsilon)\cup(1-\varepsilon,1)\text{ for }\alpha=0,1,2,\text{ and }\bar{\gamma}=0\},
\end{multline*}
with,
\begin{align}\label{eq:UnitaryVolume}
\bar{\gamma}:=\int_0^1\gamma\,d\xi.
\end{align}
\end{definition}

We denote by $\mathcal{O}^\varepsilon$ the class of all $\varepsilon$-admissible domains $\Omega^\varepsilon$. As before, if $\Omega^\varepsilon\in\mathcal{O}$ is given by \eqref{eq:EpsAdmissDomain}, we write $\Omega^\varepsilon=\Omega^\varepsilon_\gamma$. Notice that $\mathcal{O}^\varepsilon\subset\mathcal{O}$. In particular, this implies that each $\Omega^\varepsilon_\gamma\in\mathcal{O}^\varepsilon$ is a Lipschitz domain. Moreover, the bottom boundary portion $\Gamma_\gamma$ of $\Omega_\gamma^\varepsilon$ is of class $C^2$. Then, for data under the assumptions of Theorem \ref{Thm:BdryReg}, the weak solution $(\vv(\gamma),T(\gamma))$ to the Boussinesq system \eqref{v-edp}-\eqref{T-Walls} in $\Omega_\gamma^\varepsilon$ has $H^{3/2}$ regularity on $\Gamma_\gamma^\varepsilon$, where
\begin{align}\label{eq:UnitaryVolumen}
\Gamma_\gamma^\varepsilon:=\{(x_1,x_2)\in\Gamma_\gamma\,:\,\varepsilon<x_1<1-\varepsilon\}.
\end{align} 
Finally, notice that the condition $\bar{\gamma}=0$ implies that every $\Omega_\gamma^\varepsilon\in\mathcal{O}^\varepsilon$ has a unitary measure. This significantly simplifies the computations of the shape derivative (see \eqref{eq:Jgamma} below), at the reasonably cost of looking for an optimal domain among those having a constant measure. 

\subsection{Shape derivatives}
Additionally, we consider domain perturbations in the form 
$s\mapsto \Omega_s$ given by $\Omega_s=\mathcal{T}_s(\Omega)$
where $\Omega$ is a reference domain and $s\mapsto \mathcal{T}_s$ are diffeomorphisms defined by 
\begin{align*}
\frac{d\mathcal{T}_s}{ds}&=\mathrm{V}(s)\circ \mathcal{T}_s\qquad s\in I,\\
\mathcal{T}_0&=\mathrm{Id}.
\end{align*}
Here, $I$ is a real interval such that $0\in I$ and $V\in C^\infty(I;C^\infty(\overline{\Omega^{\mathrm{M}}};\mathbb{R}^2))$
has compact support in $\Omega^\mathrm{M}\setminus\overline{\Omega}^\mathrm{m}$. Note that $\partial\Omega\setminus\Gamma_\gamma$ is left unchanged via $\mathcal{T}_s$. Additionally, we denote by $(\vv_s,p_s,\hat{T}_s)$ the weak solution in $H^1(\Omega_s;\mathbb{R}^2)\times L^2_0(\Omega_s)\times H^1_0(\Omega_s)$ to the Boussinesq in $\Omega_s$, and consider the following transported variables, defined on $\Omega$,
\begin{equation*}
\vv^s:=\mathbb{P}_s^{-1}(\vv_s),\quad p^s:=p_s\circ \mathcal{T}_s, \quad \text{and}\quad \hat{T}^s:=\hat{T}_s\circ \mathcal{T}_s.
\end{equation*}
The operator $\mathbb{P}_s$ is the \textit{Piola transform}, which is defined as  
\begin{equation}\label{eq:Piola}
\mathbb{P}_s(\vv):\Omega_s\to \mathbb{R}^2,\qquad \mathbb{P}_s(\vv):=(C_s \cdot\vv) \circ  \mathcal{T}_s^{-1},
\end{equation}
for functions $\vv:\Omega\to\mathbb{R}^2$, where $C_s:=J_s^{-1} D\mathcal{T}_s$, $J_s=\det (D\mathcal{T}_s)$, and $D\mathcal{T}_s$ is the Jacobian of $x\mapsto \mathcal{T}_s(x)$.  
It is worth mentioning here that the Piola transform $\mathbb{P}_s$ is an isomorphism between $ V (\Omega)$ and $ V (\Omega_s)$. Also, analogously to the approach in \cite{boisgerault2018shape} we observe that $I\supset I_0 \ni s\mapsto\vv^s\in H^1(\Omega;\mathbb{R}^2)$, $I\supset I_0 \ni s\mapsto p^s\in L^2(\Omega)$, and $I\supset I_0 \ni s\mapsto\hat{T}^s\in H_0^1(\Omega)$ are differentiable for some $I_0$. In contrast to \cite{boisgerault2018shape}, we do not have continuity of $s\mapsto\vv^s$ when considered with values in $H^2$ as in general $\vv^s\notin H^2(\Omega;\mathbb{R}^2)$. However, as we only consider perturbations locally on $\Gamma_\gamma$, we consider the following.

We denote by $D^{\mathrm{M}}_{\mathrm{m}}$ the open set where perturbations of the domain are located. Specifically, we define this set as $D^\mathrm{M}_\mathrm{m}=\Omega^\mathrm{M}\setminus\overline{\Omega}^\mathrm{m}$;
and for a particular admissible domain $\Omega$, we define $\Omega^{\mathrm{M}}_{\mathrm{m}}=\Omega\cap D^{\mathrm{M}}_{\mathrm{m}}$. We say that $\vv$ is \textit{shape differentiable} in $H^1(\Omega;\mathbb{R}^2)$ if the following two conditions hold true:
\begin{itemize}
	\item[i)] $s\mapsto \vv_s\circ \mathcal{T}_s\in H^1(\Omega;\mathbb{R}^2)$ is differentiable at $s=0$. The derivative is denoted as $\dot{\vv}$ and it is called the \textit{material derivative}.
	\item[ii)] The restriction of velocity profile $\vv_0$ associated to the initial domain $\Omega$ to $\Omega^{\mathrm{M}}_{\mathrm{m}}$ has $H^2$ regularity, i.e., $\vv_0|_{\Omega^{\mathrm{M}}_{\mathrm{m}}}\in H^2(\Omega^{\mathrm{M}}_{\mathrm{m}};\mathbb{R}^2)$.
\end{itemize}
Analogously, the shape differentiability of $\hat{T}$ on $H^1_0(\Omega)$ is defined and consequently, { if $\vv$ and $\hat{T}$ are shape differentiable in $H^1(\Omega;\mathbb{R}^2)$ and $H^1_0(\Omega)$, respectively, then} it follows that $p$ is shape differentiable in $L^2(\Omega)$: here  i) and ii) 
hold mutatis mutandis by reducing one order of differentiability on the mentioned Sobolev spaces. Further, the \textit{shape derivative} of $\vv$ is denoted as $\vv'$ and given by
\begin{equation*}
	\vv'=\dot{\vv}-\nabla \vv\cdot \mathrm{V}(0).
\end{equation*}
Note that even though $\nabla\vv$ appears in the expression of the shape derivative, we observe that $\vv'\in H^1(\Omega;\mathbb{R}^2)$, as $\mathrm{V}$ vanishes outside $D^{\mathrm{M}}_{\mathrm{m}}$.
Analogously, we can define the shape derivative of $T$. 
Our goal now is to express the shape derivative $\vv'$ as the derivative of the extension  of $\vv_s$, i.e., as the derivative of the map  $s\mapsto E\vv_s$ where $E$ is a continuous extension from $\Omega_s$ to $\Omega^{\mathrm{M}}$ such that $E:H^1(\Omega;\mathbb{R}^2)\to H^1(\Omega^{\mathrm{M}};\mathbb{R}^2)$ with a norm uniformly bounded with respect to $s$. Also, let $E_0:L^2(\Omega)\to L^2(\Omega^{\mathrm{M}})$ be the extension by zero operator. In particular, we define
\begin{equation*}
	\mathbf{V}_s\circ \mathcal{T}_s=E(\vv_s\circ \mathcal{T}_s),\quad \tau_s\circ \mathcal{T}_s=E(\hat{T}_s\circ \mathcal{T}_s) \quad\text{and}\quad 	P_s\circ \mathcal{T}_s=E_0(p_s\circ \mathcal{T}_s).
\end{equation*}
It can be proved that $J\ni s\mapsto \mathbf{V}_s\in H^1(\Omega^{\mathrm{M}};\mathbb{R}^2)$ and $J\ni s\mapsto \tau_s\in H^1(\Omega^{\mathrm{M}})$ are continuously differentiable: This follows given that $J\ni s\mapsto \mathbf{V}_s\circ \mathcal{T}_s\in H^1(\Omega^{\mathrm{M}};\mathbb{R}^2)$ and $J\ni s\mapsto \tau_s\circ \mathcal{T}_s\in H^1(\Omega^{\mathrm{M}})$ are continuously differentiable, $\mathrm{V}:I\to \Omega^{\mathrm{M}}$ is zero outside $\Omega^{\mathrm{M}}_{\mathrm{m}}$,  $\vv_s\in H^2(\Omega^{\mathrm{M}}_{\mathrm{m}};\mathbb{R}^2)$, and  $T_s\in H^2(\Omega^{\mathrm{M}}_{\mathrm{m}})$  with a uniform norm with respect to $s$; see \cite{boisgerault2018shape} and \cite[Prop 2.38, p.71]{sokolowski1992introduction}.  Hence, we observe that
\begin{equation*}
	\vv'=\frac{\partial\mathbf{V}_s}{\partial s}(0)\Big|_{\Omega}, \quad \text{where} \quad\frac{\partial\mathbf{V}_s}{\partial s}(0)=\frac{\partial\mathbf{V}_s \circ \mathcal{T}_s}{\partial s}(0)-\nabla \mathbf{V}_0\cdot\mathrm{V}(0).
\end{equation*}
and 
\begin{equation*}
	\hat{T}'=\frac{\partial\tau_s}{\partial s}(0)\Big|_{\Omega}, \quad \text{where} \quad\frac{\partial\tau_s}{\partial s}(0)=\frac{\partial\tau_s \circ \mathcal{T}_s}{\partial s}(0)-\nabla \tau_0\cdot\mathrm{V}(0).
\end{equation*}
It follows that $\vv$ and $\hat{T}$ are shape differentiable in $H^1(\Omega;\mathbb{R}^2)$ and $H^1(\Omega)$ respectively, as the regularity of $s\mapsto \vv^s$ and $s\mapsto \hat{T}^s$ imply the regularity of $s\mapsto \vv_s\circ \mathcal{T}_s$ and $s\mapsto \hat{T}_s\circ \mathcal{T}_s$; and due to the additional regularity proven before. Furthermore, this implies that $p$ is {shape differentiable} in $L^2(\Omega)$. 

It follows by standard methods that the shape derivative $(\vv',p',\hat{T}')$ of $(\vv,p,\hat{T})\in {H^1(\Omega;\mathbb{R}^2)\times L^2(\Omega)\times H^1_0(\Omega)}$ satisfies (weakly) the following  system 
\begin{align*}
\vv'\cdot\nabla \vv+\vv\cdot\nabla\vv' -\frac{1}{\Rey}\Delta \vv'+\nabla p'-\frac{\Gra}{\Rey^2}T'\boldsymbol{e}=\,&0&\text{in}&\quad\Omega,&\\
\nabla\cdot\vv'=\,&0&\text{in}&\quad\Omega,&\\
\vv'=\,&0&\quad\text{on}&\quad\partial\Omega\setminus \Gamma_\gamma^\varepsilon,&\\
\vv'=\,&-(\nabla \vv\,\n) (\mathrm{V}(0)\cdot \n)& \quad\text{on}&\quad \Gamma_\gamma^\varepsilon
\end{align*}
and 
\begin{align*}
\vv'\cdot\nabla {T}+\vv\cdot\nabla {T}' -\frac{1}{\Rey\Pra}\Delta {T}'=\,&0&\text{in}&\quad\Omega,&\\
\hat{T}'=\,&0&\text{on}&\quad\partial\Omega\setminus\Gamma_\gamma^\varepsilon ,&\\
\hat{T}'=\,&-(\nabla \hat{T}\,\n) (\mathrm{V}(0)\cdot \n)& \quad\text{on}&\quad \Gamma_\gamma^\varepsilon,
\end{align*}
where $T=\hat{T}+T_d$. Thus, the adjoint formalism of Section \ref{Sect:Adjoint} is rigorous in view of the existence and structure of the shape derivatives. 
 
\subsection{Functional perturbations of the boundary}
From now on, we consider perturbations of a given reference domain $\Omega^\varepsilon_\gamma$, in the form
\begin{equation*}
\Omega_\gamma\mapsto\Omega_{\gamma+sh},
\end{equation*}
where $h\in H_{\nu,L}\cap V_\varepsilon$ and $s\geq 0$ is sufficiently small to have $\gamma+sh\in H_{\nu,L}\cap H_\varepsilon$. This sort of perturbations is achieved by considering a family of diffeomorphism $\mathcal{T}_s=\mathcal{T}_{s,\gamma,h}$, defined as before, with constant $\mathrm{V}(s)=\mathrm{V}_{\gamma,h}$ given by
\begin{equation}\label{eq:Diffeo}
\mathrm{V}_{\gamma,h}(x_1,x_2)=\left(0,\frac{1-x_2}{1-\gamma(x_1)}h(x_1)\right).
\end{equation}
Notice that, in this case, $\mathcal{T}_{s,\gamma,h}(x_1,x_2)=(\mathrm{Id}+s\mathrm{V}_{\gamma,h})(x_1,x_2),$
so that
\begin{equation}\label{eq:subclass}
\mathcal{T}_{s,\gamma,h}(\Omega_\gamma)=\Omega_{\gamma+sh}.
\end{equation}
Observe that the choice of $V_{\gamma,h}$ in \eqref{eq:Diffeo} to obtain the characterization \eqref{eq:subclass} is not unique. Also note that points are only moved in the $x_2$-direction. 
We are now in a position to compute the Gateaux derivative of objective functional through \eqref{eq:AbstractJDerivative}.  

\begin{theorem}[Gateaux derivative of the objective functional]\label{Thm:ShapeDer}Let $\Omega_\gamma^\varepsilon\in\mathcal{O}^\varepsilon$.
Consider $\boldg_1\in L^2(\Omega^\mathrm{M};\mathbb{R}^2)$, $g_2\in L^2(\Omega^\mathrm{M})$, and $T_d\in H^{2}(\Omega^\mathrm{M})\cap C_c(\Omega^\mathrm{M})$. Assume that $\Rey$, $\Pra$, $\Gra$ satisfy \eqref{cond:RePrGr-bis} and $\boldg_1$, $g_2$, $T_d$ satisfy the smallness condition in Lemma \ref{Lem:Adjoint}. Then,
\begin{align}\label{eq:ShapeDerivative}
\hat{J}'(\gamma)h=\,&
-\int_\varepsilon^{1-\varepsilon}\left(F(\xi,\gamma(\xi))
+\lambda_1\gamma''(\xi)\right)h(\xi)\,d\xi,
\end{align}
for every $h\in H_{\nu,L}\cap V_\varepsilon$, where 
\begin{align}\label{F-Def}
F(\xi,\gamma(\xi)):=\,&\frac{1}{\Rey}\left(\frac{\partial\vv}{\partial{\bf n}}\cdot\frac{\partial\ww}{\partial{\bf n}}\right)(\xi,\gamma(\xi))+
\frac{1}{\Rey\Pra}\left(\frac{\partial T}{\partial{\bf n}}\frac{\partial S}{\partial{\bf n}}\right)(\xi,\gamma(\xi))\notag\\
&+
\frac{1}{2}|T(\xi,\gamma(\xi))-I(T,\gamma)|^2,
\end{align}
and $(\vv,T)$ and $(\ww,S)$ are the weak solutions to the primal \eqref{v-edp}-\eqref{T-Walls} and the adjoint \eqref{w-edp}-\eqref{S-Dirichlet} systems in $\Omega_\gamma^\varepsilon\in \mathcal{O}^\varepsilon$, respectively.
\end{theorem}

\begin{proof}
As in the proof of Lemma \ref{Lem:AdjointSystem}, we denote by $(\vv,\hat{T})$ to a generic element of $V(\Omega^\varepsilon_\gamma)\times H_0^1(\Omega^\varepsilon_\gamma)$. We begin by computing $e_{\gamma}(\vv,\hat{T},\gamma)^*$. For this, we notice that
\begin{align*}
&\langle e_{1,\gamma}(\vv,\hat{T},\gamma)h,\ww\rangle\\
&\qquad=
\int_{\partial\Omega_\gamma^\varepsilon}(\vv\cdot\nabla\vv)\cdot\ww\,(V(\gamma,h)\cdot{\bf n})\,d\sigma+\frac{1}{\Rey}\int_{\partial\Omega_\gamma^\varepsilon}(\nabla\vv\cdot\nabla\ww)(V(\gamma,h)\cdot{\bf n})\,d\sigma\\
&\qquad\quad-\frac{\Gra}{\Rey^2}\int_{\partial\Omega_\gamma^\varepsilon}(T\mathbf{e}\cdot\ww)(V(\gamma,h)\cdot{\bf n})\,d\sigma
-\int_{\partial\Omega_\gamma^\varepsilon}(\boldg_1\cdot\ww)(V(\gamma,h)\cdot{\bf n})\,d\sigma\\
&\qquad=\frac{1}{\Rey}\int_{\Gamma_\gamma^\varepsilon}(\nabla\vv\cdot\nabla\ww)(V(\gamma,h)\cdot{\bf n})\,d\sigma\\
&\qquad=\frac{1}{\Rey}\int_{\Gamma_\gamma^\varepsilon}\left(\frac{\partial\vv}{\partial{\bf n}}\cdot\frac{\partial\ww}{\partial{\bf n}}\right)(V(\gamma,h)\cdot{\bf n})\,d\sigma,
\end{align*}
since $\vv=\ww=0$ on $\partial\Omega_\gamma^\varepsilon$ and $V(\gamma,h)=0$ on $\partial\Omega_\gamma^\varepsilon\setminus\Gamma_\gamma^\varepsilon$. 
The unitary normal exterior vector field ${\bf n}$ to $\Gamma_\gamma^\varepsilon$ is given by
\begin{align*}
{\bf n}(\xi,\gamma(\xi))=\left(\frac{\gamma'(\xi)}{\sqrt{1+(\gamma'(\xi))^2}},\frac{-1}{\sqrt{1+(\gamma'(\xi))^2}}\right),\qquad\qquad\xi\in(\varepsilon,1-\varepsilon).
\end{align*}
Then,
\begin{align*}
V(\xi,\gamma(\xi))\cdot{\bf n}(\xi,\gamma(\xi))=\frac{-h(\xi)}{\sqrt{1+(\gamma'(\xi))^2}},\qquad\qquad \xi\in(\varepsilon,1-\varepsilon).
\end{align*}
Hence,
\begin{align*}
\langle e_{1,\gamma}(\vv,\hat{T},\gamma)h,\ww\rangle=
-\int_\varepsilon^{1-\varepsilon}h(\xi)F_1(\xi,\gamma(\xi))\,d\xi,
\end{align*}
where $F_1(\xi,\gamma(\xi)):=\frac{1}{\Rey}\left(\frac{\partial\vv}{\partial{\bf n}}\cdot\frac{\partial\ww}{\partial{\bf n}}\right)(\xi,\gamma(\xi))$. Similarly,
\begin{align*}
\langle e_{2,\gamma}(\vv,\hat{T},\gamma)h,S\rangle=-
\int_\varepsilon^{1-\varepsilon}h(\xi)F_2(\xi,\gamma(\xi))\,d\xi,
\end{align*}
where $F_2(\xi,\gamma(\xi)):=\frac{1}{\Rey\Pra}\left(\frac{\partial T}{\partial{\bf n}}\frac{\partial S}{\partial{\bf n}}\right)(\xi,\gamma(\xi))$.
Then, 
\begin{align}\label{eq:partial-e-gamma}
e_\gamma(\vv,\hat{T},\gamma)^*(\ww,S)h=
-\int_\varepsilon^{1-\varepsilon}h(\xi)(F_1(\xi,\gamma(\xi))+F_2(\xi,\gamma(\xi)))\,d\xi.
\end{align}

Now we focus on $J_{\gamma}(\vv,\hat{T},\gamma)$. A direct calculation yields
\begin{align}\label{eq:Jgamma}
J_{\gamma}(\vv,\hat{T},\gamma)h=\,&
-I(T,\gamma)I(T-I(T,\gamma),\gamma)+\frac{1}{2}\int_{\partial\Omega_\gamma^\varepsilon}|T-I(T,\gamma)|^2V(\gamma,h)\cdot{\bf n}\,d\sigma\notag\\
&+\lambda_1\int_0^1\gamma'(\xi)h'(\xi)\,d\xi.
\end{align}
We have used here that $I_\gamma(T,\gamma)h=I(T,\gamma)$ since every domain in $\mathcal{O}^\varepsilon$ has a unitary measure.  Exploiting this again, we notice that $I(T-I(T,\gamma),\gamma)=0$. Then, we get
\begin{align*}
J_{\gamma}(\vv,\hat{T},\gamma)h=
\frac{1}{2}\int_{\Gamma_\gamma^\varepsilon}|T-I(T,\gamma)|^2V(\gamma,h)\cdot{\bf n}\,d\sigma-\lambda_1\int_\varepsilon^{1-\varepsilon}\gamma''(\xi)h(\xi)\,d\xi,
\end{align*}
where we have used that $\gamma\in H^2(0,1)$.
Then, similarly as before, we deduce that
\begin{align}\label{eq:partial-J-gamma}
J_{\gamma}(\vv,\hat{T},\gamma)h=
-\int_\varepsilon^{1-\varepsilon}h(\xi)F_3(\xi,\gamma(\xi))\,d\xi-\lambda_1\int_\varepsilon^{1-\varepsilon}\gamma''(\xi)h(\xi)\,d\xi.
\end{align}
where $F_3(\xi,\gamma(\xi)):=\frac{1}{2}|T(\xi,\gamma(\xi))-I(T,\gamma)|^2$.

Using \eqref{eq:partial-e-gamma} and \eqref{eq:partial-J-gamma} in \eqref{eq:AbstractJDerivative}, we obtain \eqref{eq:ShapeDerivative}.
\end{proof}

The existence of a solution to the shape optimization problem \eqref{shape}, now posed for the subclass $\mathcal{O}^\varepsilon$ of admissible domains, can be proved exactly as for Theorem \ref{Thm:weak-EU}. Moreover, both existence results hold true under the same assumptions. Then, we finally have the following optimality condition for the shape optimization problem \eqref{shape}.

\begin{theorem}[Optimality condition]
Consider $\boldg_1\in L^2(\Omega^\mathrm{M};\mathbb{R}^2)$, $g_2\in L^2(\Omega^\mathrm{M})$, and $T_d\in H^{2}(\Omega^\mathrm{M})\cap C_c(\Omega^\mathrm{M})$. Assume that $\Rey$, $\Pra$, $\Gra$ satisfy \eqref{cond:RePrGr-bis} and $\boldg_1$, $g_2$, $T_d$ satisfy the smallness condition in Lemma \ref{Lem:Adjoint}. If $\Omega_{\gamma^*}$ solves problem \eqref{shape} in $\mathcal{O}^\varepsilon$, then 
\begin{align}\label{eq:OptCond}
\int_\varepsilon^{1-\varepsilon}\left(F(\xi,\gamma^*(\xi))
+\lambda{\gamma^*}''(\xi)\right)(\gamma^*(\xi)-h(\xi))\,d\xi\geq0,
\end{align}
for all $h\in H_{\nu,L}\cap V_\varepsilon$, where $F$ is given by \eqref{F-Def}.
\end{theorem}

\section{Algorithm and numerical simulations}\label{Sec:Numerics}
In this section, we propose an algorithm to find a solution $\gamma^*\in H_{\nu,L}\cap V_\varepsilon$ to the variational inequality \eqref{eq:OptCond}, and present several numerical tests. 

For numerical purposes we treat $\varepsilon=0$ and we consider the following regularization of the gradient induced by the Gateaux derivative in \eqref{eq:ShapeDerivative}:
\begin{equation}\label{eq:PertShapeDer}
DJ(\gamma)(\xi)=-F(\xi,\gamma(\xi))-\lambda_1{\gamma}''(\xi)+\lambda_2\int_0^1\gamma(\zeta)\,d\zeta+\lambda_3(\gamma(\xi)-1+\nu)^+,
\end{equation}
where $\lambda_2,\lambda_3>0$ and $\nu\in(0,1)$ are given, and $F$ is defined by \eqref{F-Def}. This regularization is put in place in order to relax the volume and ``bottom not touching top'' constraints. By choosing large values of $\lambda_2$ and $\lambda_3$, the last two terms in \eqref{eq:PertShapeDer} penalize the fact that $\gamma$ does not have a zero average in the interval $[0,1]$ and that $\gamma$ does not fall below $1-\nu$. 
This regularization is consistent with the following cost functional:
\begin{equation*}
J(T,\gamma)=J_1(T,\gamma)+J_2(\gamma),
\end{equation*}
where
\begin{align}
J_1(T,\gamma)=\,&\|T-I(T,\gamma)\|^2_{L^2(\Omega_\gamma)},\\
J_2(\gamma)=\,&\frac{\lambda_1}{2}\|\gamma\,'\|^2_{L^2(0,1)}+\frac{\lambda_2}{2}\left(\int_0^1\gamma(\xi)\,d\xi\right)^2+\frac{\lambda}{2}\int_0^1(\gamma(\xi)-1+\nu)^{+2}\,d\xi,
\end{align}
As before, here $I(T,\gamma)$ is given by \eqref{eq:average}.
The proposed algorithm consists on solving
\begin{equation}
DJ(\gamma)=0\quad\text{in}\quad(0,1),\qquad \gamma(0)=\gamma(1)=0,
\end{equation}
by the gradient descent scheme according to
\begin{equation}
\gamma_{n+1}=\gamma_n-\tau (-\Delta)^{-1}DJ(\gamma_n),\label{eq:gd}
\end{equation} 
where $\tau>0$ and $\gamma_0\in H^2(\Omega)\cap H^1_0(\Omega)$ such that $\int_0^1 \gamma_0(\zeta)\mathrm{d}\zeta=0$ are given; and $(-\Delta)^{-1}$ is the inverse of the Laplace operator associated with zero boundary conditions. We end the iterative procedure when the solution $\varphi_n$ to
\begin{equation}\label{eq:Poisson}
-\varphi_n''=DJ(\gamma_n)\quad\text{in}\quad(0,1), \qquad \varphi_n(0)=\varphi_n(1)=0,
\end{equation}
is sufficiently small, providing we are close to the desired function $\gamma$.

In the following, we report several numerical tests to validate the proposed algorithm and our theoretical results. In all the subsequent tests, the boundary data $T_d$ is given by
\begin{equation*}
T_d(x_1,x_2)=\alpha x_1(1-x_1)(1-x_2),\qquad x_1,x_2\in[0,1],
\end{equation*}
for some $\alpha>0$. Further, we shall take the parameters as
$\Pra = 0.7, \Gra = 1, \Rey = 1, \alpha = 10,\lambda_1 = 0.5, \lambda_2 = 1.5E4,\lambda_3 = 1E3, \nu= 0.1,\tau = 1E-3$. The computational mesh $\mathcal{T}_h$ is generated as a structured triangular grid with mesh size $h = 0.03.$ Taylor-Hood element (${\bf V}_h\times Q_h$) of order $k=2$ and discontinuous Galerkin finite element methods of order $k=1$ ($W_h$) will be used to solve fluid problems ($({\bf v},p)$ and $({\bf w},q)$) and heat equations ($T$ and $S$). 
Here the finite element spaces are denoted as
\begin{eqnarray*}
    {\bf V}_h &=& \{{\bf v}\in [H_0^1(\Omega_\gamma)]^2\,:\, {\bf v}\vert_T\in [\mathbb{P}_2(K)]^2,\; K\in\mathcal{T}_h\},\\
    Q_h &=& \{q\in H^1(\Omega_\gamma)\,:\, q\vert_T\in \mathbb{P}_1(K),\; K\in\mathcal{T}_h\}\cap L_0^2(\Omega_\gamma),\\
    W_h &=&\{w\in L^2(\Omega_\gamma)\,:\, w\vert_T\in\mathbb{P}_1(K),\; K\in\mathcal{T}_h\}.
\end{eqnarray*}
The finite difference method has been used to solve the one-dimensional Poisson equation \eqref{eq:Poisson}. As an initial state, we will test $\gamma_0$ as a straight line bottom and a curvy bottom. Fig.~\ref{fig:case1-case3} illustrates the initial domain $\Omega_\gamma$ to start our optimization process. 

\begin{figure}[H]
    \centering
    \begin{tabular}{ccc}
   \includegraphics[width=.3\textwidth]{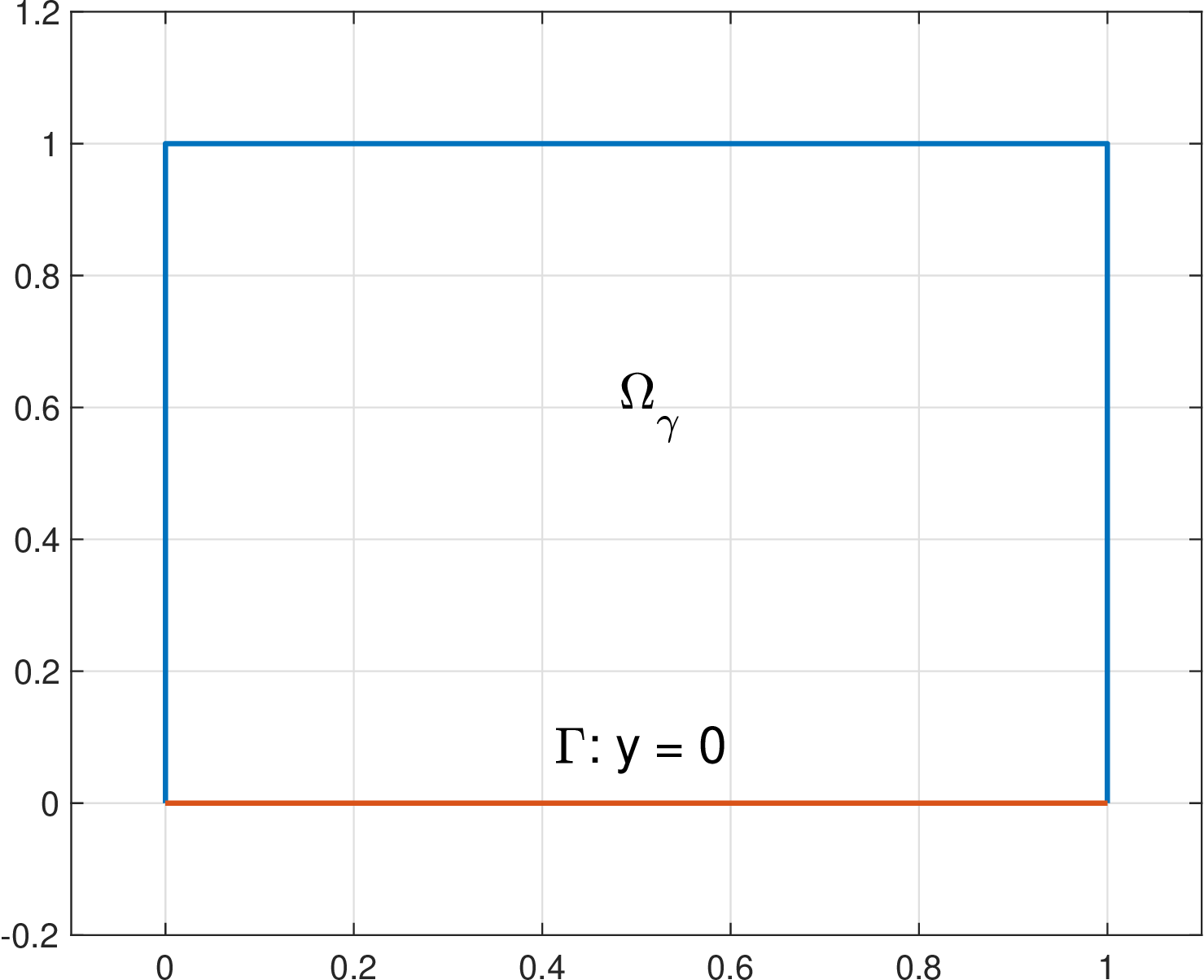}&\includegraphics[width=.3\textwidth]{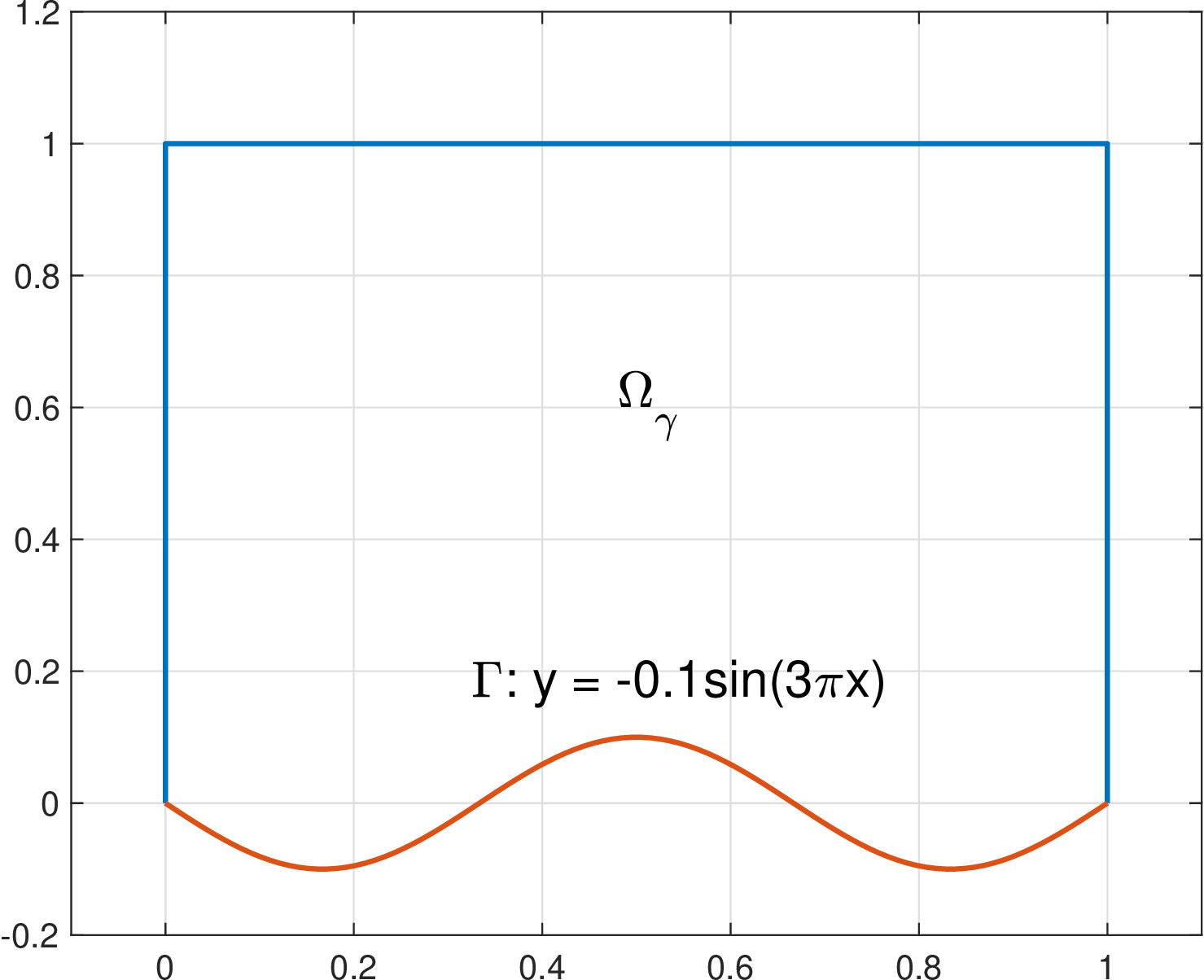}&\includegraphics[width=.3\textwidth]{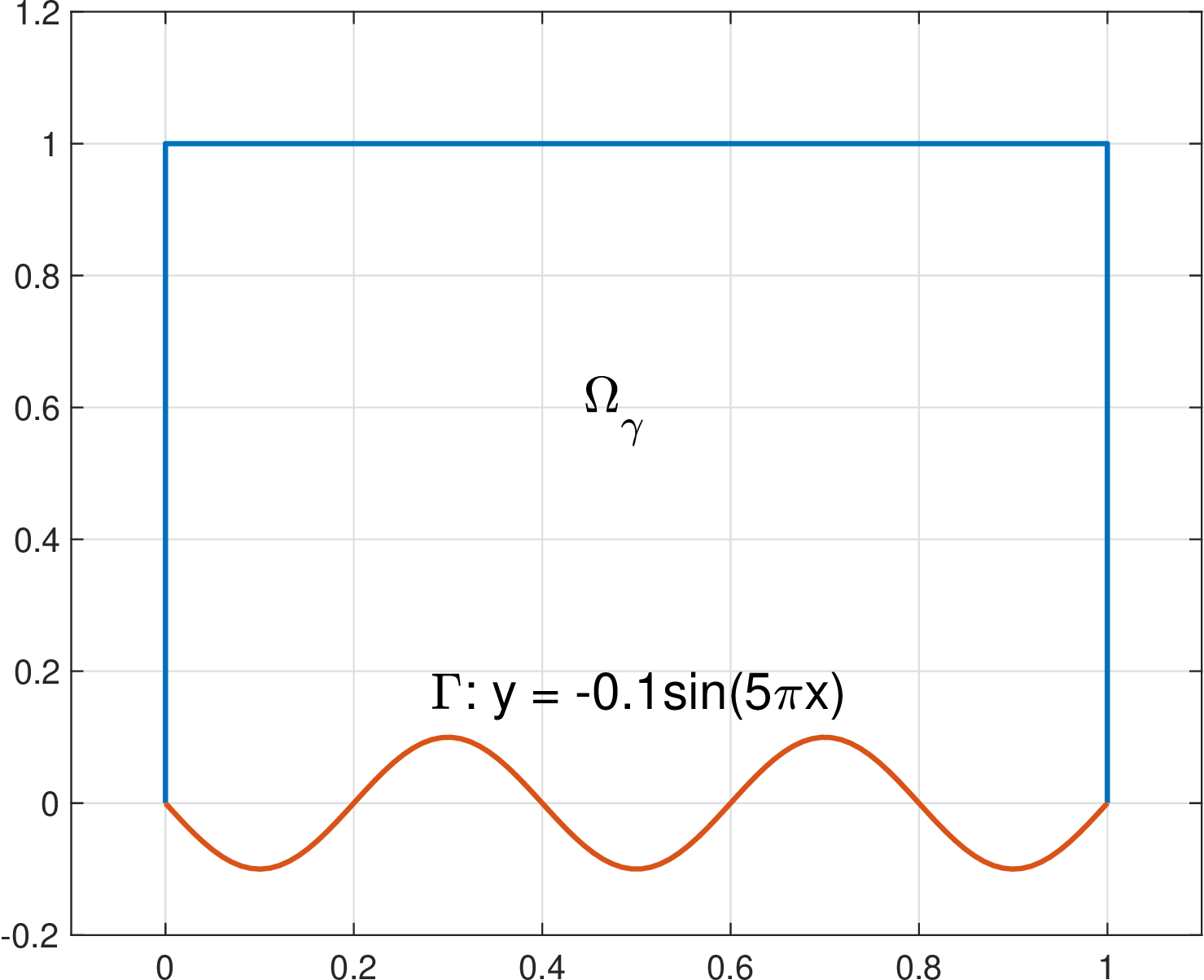}  \\
         (a). Case 1 & (b). Case 2 & (c). Case 3        
    \end{tabular}
    \caption{Illustration of the initial domain $\Omega_\gamma$ in Section~\ref{Sect:Num-1}-Section~\ref{Sect:Num-3}.}
    \label{fig:case1-case3}
\end{figure}

\subsection{Case 1: initial state with $\gamma: y = 0$}\label{Sect:Num-1}

\begin{figure}[H]
    \centering
    \begin{tabular}{ccc}
    \includegraphics[width=0.3\linewidth]{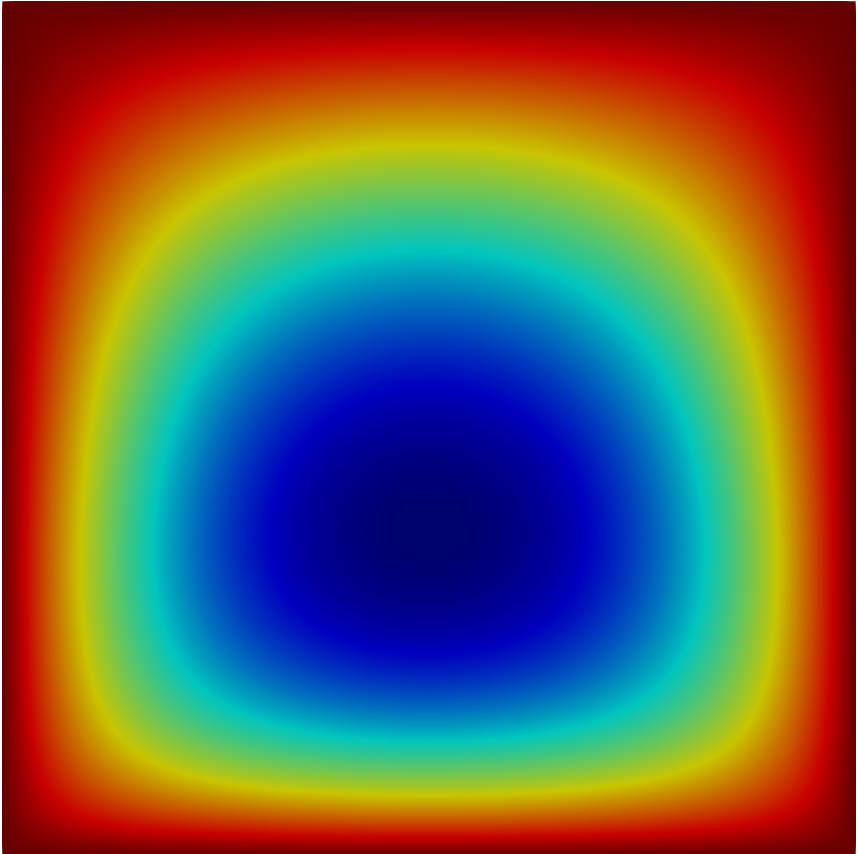}
    &\includegraphics[width=0.3\linewidth]{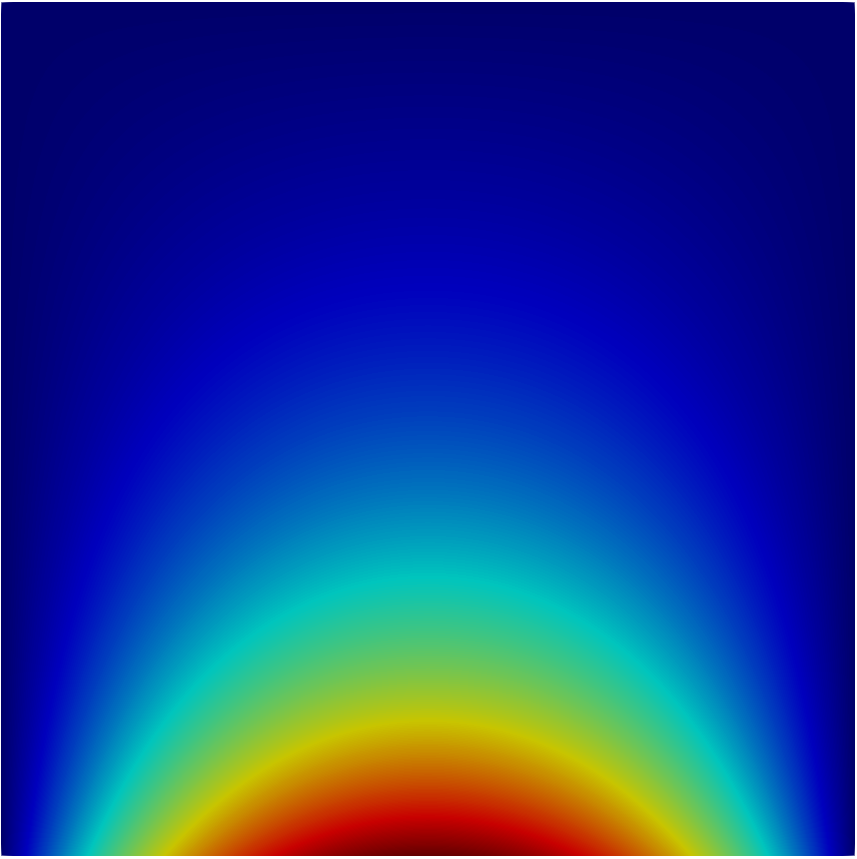}
    &\includegraphics[width=0.3\linewidth]{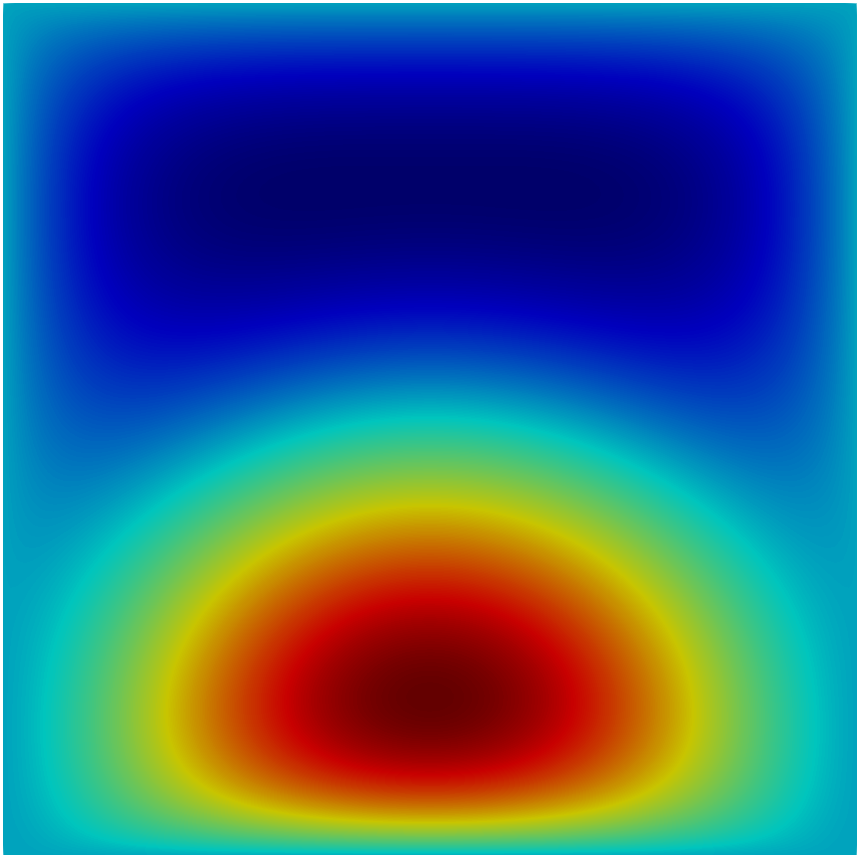}\\
    (a). $\hat{T}$ & (b). $T=\hat{T}+T_d$ &(c). $S$
    \end{tabular}
    \caption{Case\ref{Sect:Num-1}: Plot for the numerical solution on the initial domain $\Omega_0.$}
    \label{fig:Num-1-1}
\end{figure}

\begin{figure}[H]
    \centering
    \begin{tabular}{cccc}
         \includegraphics[width=0.21\linewidth]{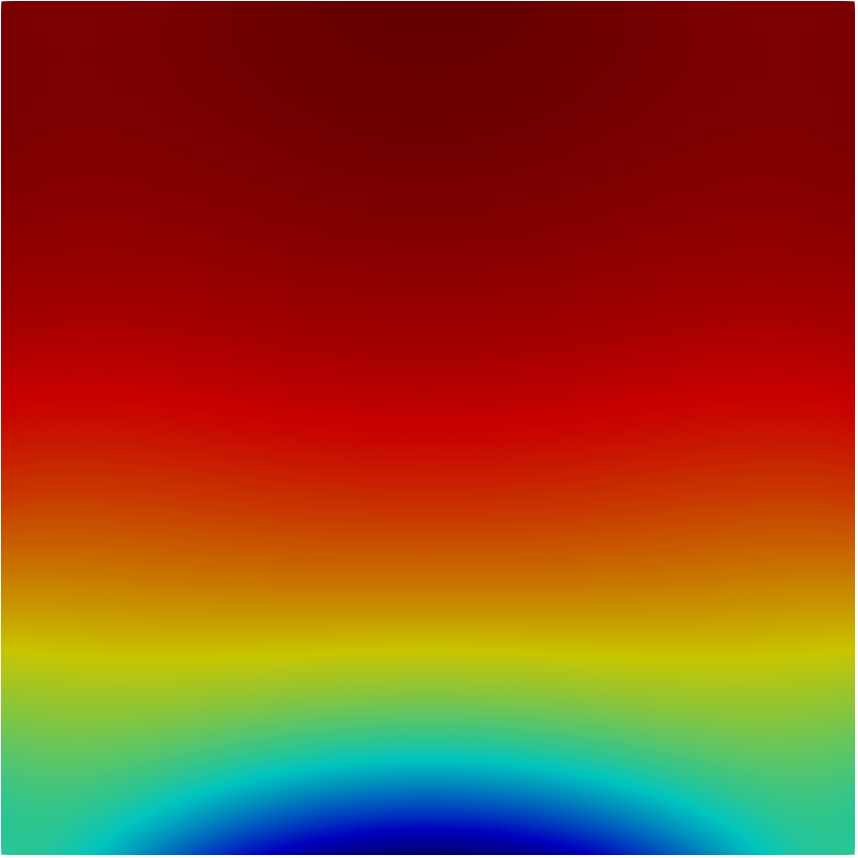}
    &\includegraphics[width=0.21\linewidth]{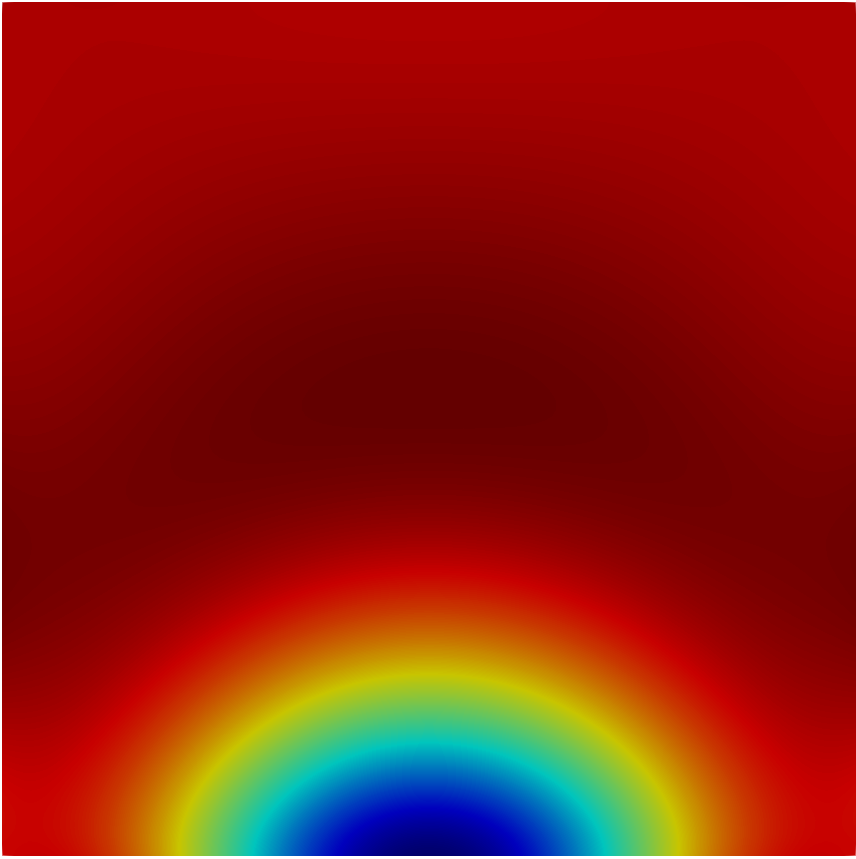}
        &\includegraphics[width=0.21\linewidth]{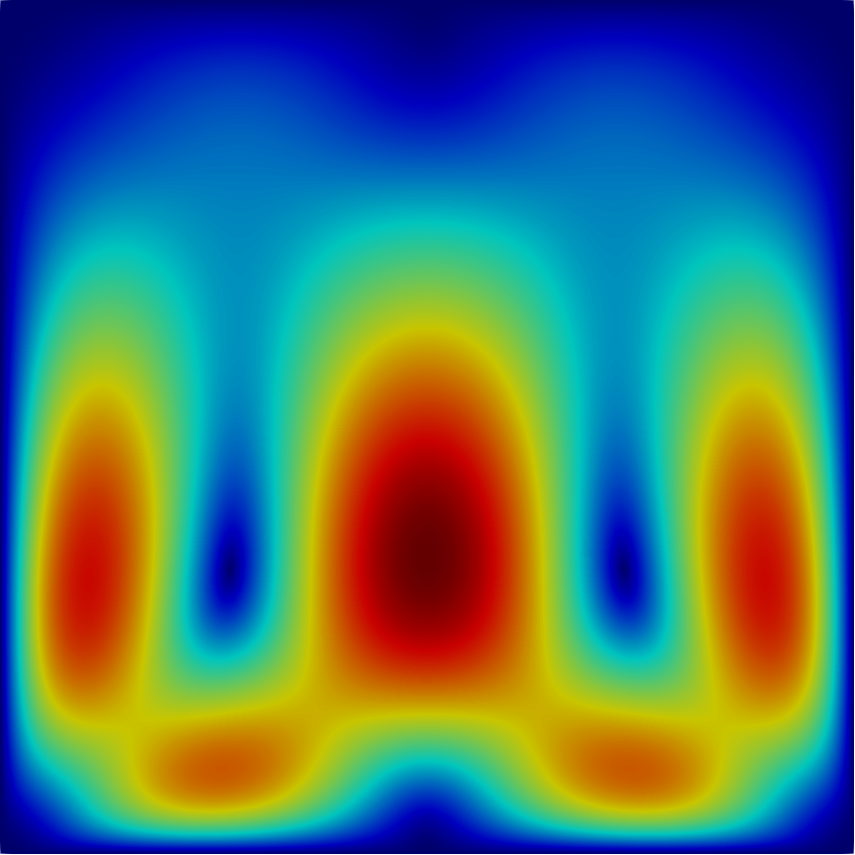}
   & \includegraphics[width=0.21\linewidth]{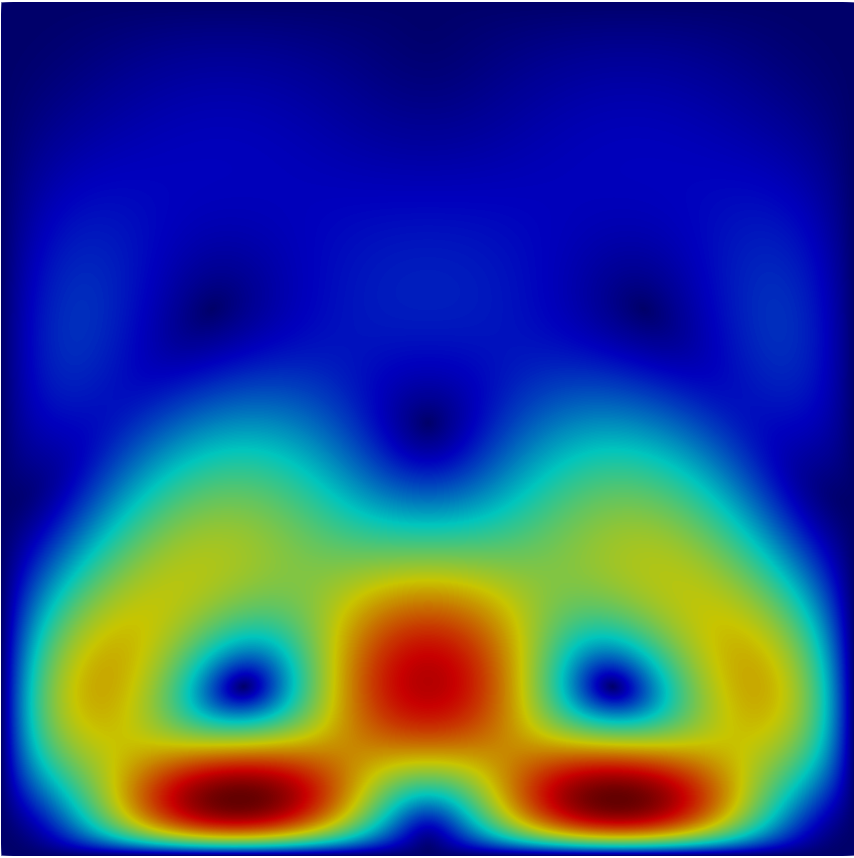}
      \\
  (a). $p$ & (b). $q$ &(c). $|{\bf v}|$ &(d). $|{\bf w}|$
    \end{tabular}
    \caption{Case\ref{Sect:Num-1}: Plot for the numerical solution on the initial domain $\Omega_0.$}
    \label{fig:Num-1-2}
\end{figure}

Firstly, we start with the optimization by a straight line bottom, which is given by $y = 0$. In the initial domain $\Omega_\gamma$, we calculate the numerical solutions corresponding to the primal system $({\bf v},p,T)$ and the adjoint system $({\bf w},q,S)$. The numerical solutions are plotted in Fig.~\ref{fig:Num-1-1}-Fig.~\ref{fig:Num-1-2}. In the initial state, the temperature $\hat{T}$ is enforced with a homogeneous Dirichlet boundary condition, as shown in Fig.~\ref{fig:Num-1-1}a. Since we lift the boundary condition and enforce the homogeneous boundary condition for $\hat{T}$,  the temperature of the bottom of the domain for the practical simulation is calculated by $\hat{T}+T_d$, which is plotted in Fig.~\ref{fig:Num-1-1}b. One can observe that high temperature happens in the center of the bottom to mimic a source in this location. Besides, the temperature corresponding to the adjoint system $S$ is plotted in Fig.~\ref{fig:Num-1-1}c. Again, the homogeneous Dirichlet boundary condition has been used for unknown $S$. For the two fluid equations, we plot the numerical solutions $({\bf v},p)$ and $({\bf w},q)$ in Fig.~\ref{fig:Num-1-2}. As expected, the flow ${\bf v}$ contributes to mix the liquid. 

\begin{figure}[H]
    \centering
    \includegraphics[width=0.8\linewidth]{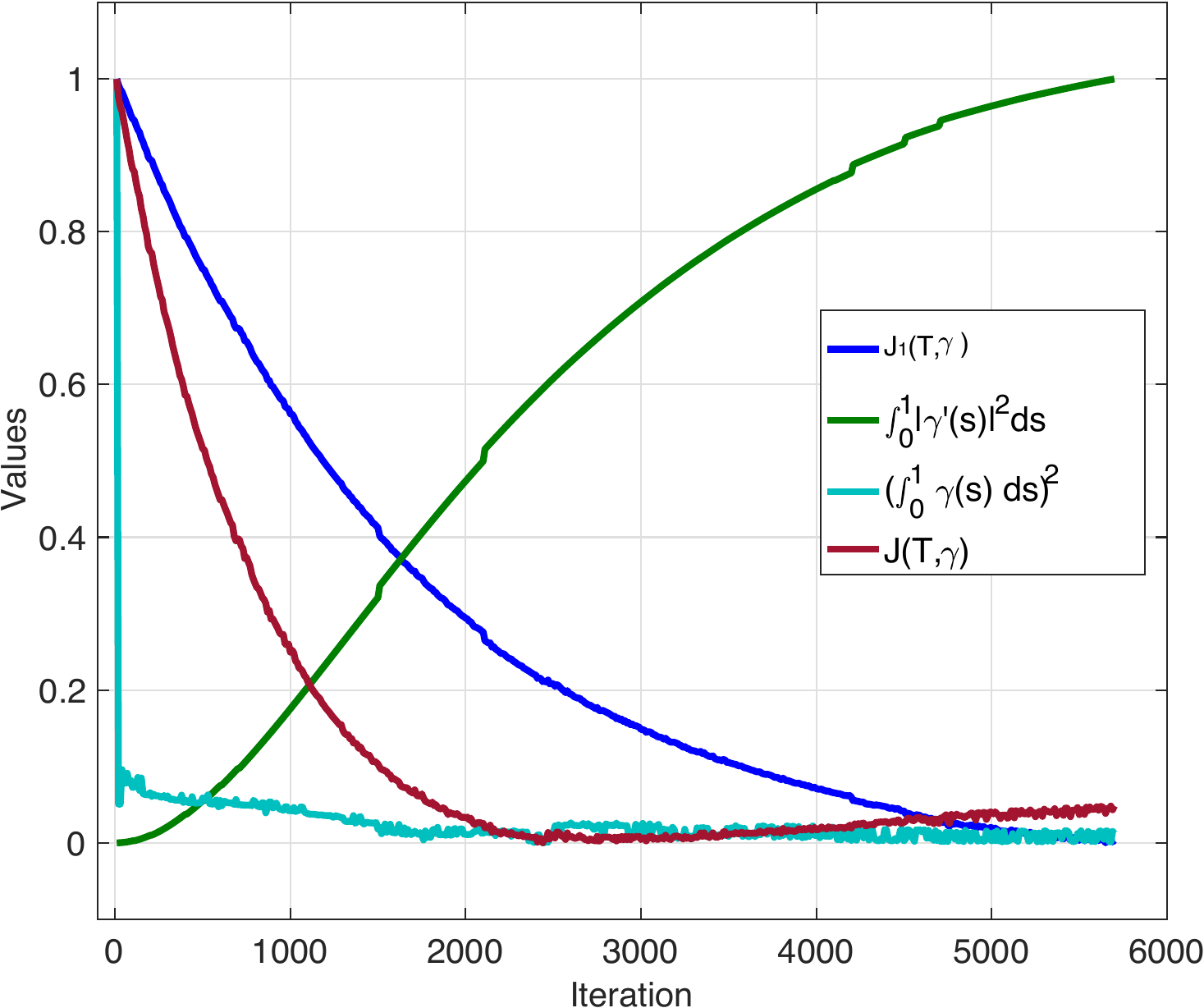}
    \caption{Case~\ref{Sect:Num-1}: Convergence test of the cost functional. 
    }
    \label{fig:Num-1-3}
\end{figure}
By using a gradient descent scheme, the boundary curve $\gamma_n$ will be updated at each iteration $n$. The cost functional has been plotted in Fig.~\ref{fig:Num-1-3}. In this plot, we scaled each term in the cost functional and re-arrange their values in the range $[0,1]$. In the first test, the cost functional $J_1$ is decaying at each iteration; however, the contribution from the curvature $\int_0^1|\gamma'(\xi)|^2d\xi$ is increasing, which means the optimization is changing the geometry of the straight line. Among all the iterations, the area of the domain $\Omega_\gamma$ remains the same. This can be justified by the quantity $(\int_0^1\gamma_n(\xi) d\xi)^2\approx 1E-9$ for all the iterations. We did not report the contribution from $\int_0^1(\gamma(\xi)-1-\nu)^{+2}d\xi$ since it takes value $0$ across all the iterations. Besides, we also observe the decaying of the total cost functional until Iter$=3000$. After this iteration, the cost functional is straggling to balance each term in the scheme, but stays almost the same value without significant decreasing despite decaying of $J_1$.

\begin{figure}[H]
    \centering
    \includegraphics[width=0.8\linewidth]{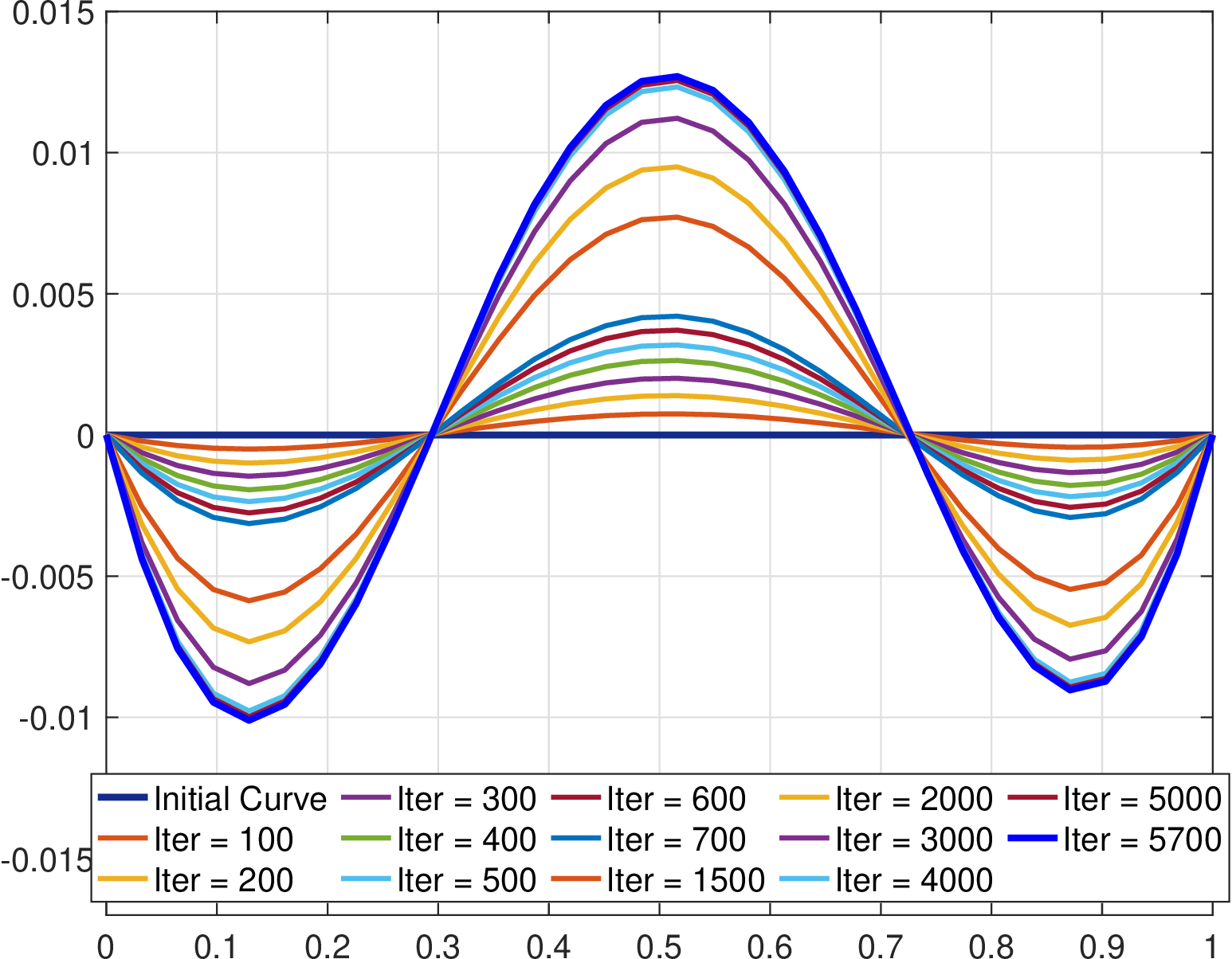}
    \caption{Case\ref{Sect:Num-1}: Plot of optimized curves.}
    \label{fig:Num-1-4}
\end{figure}

\begin{figure}[H]
    \centering
    \begin{tabular}{ccc}
    \includegraphics[width=0.3\linewidth]{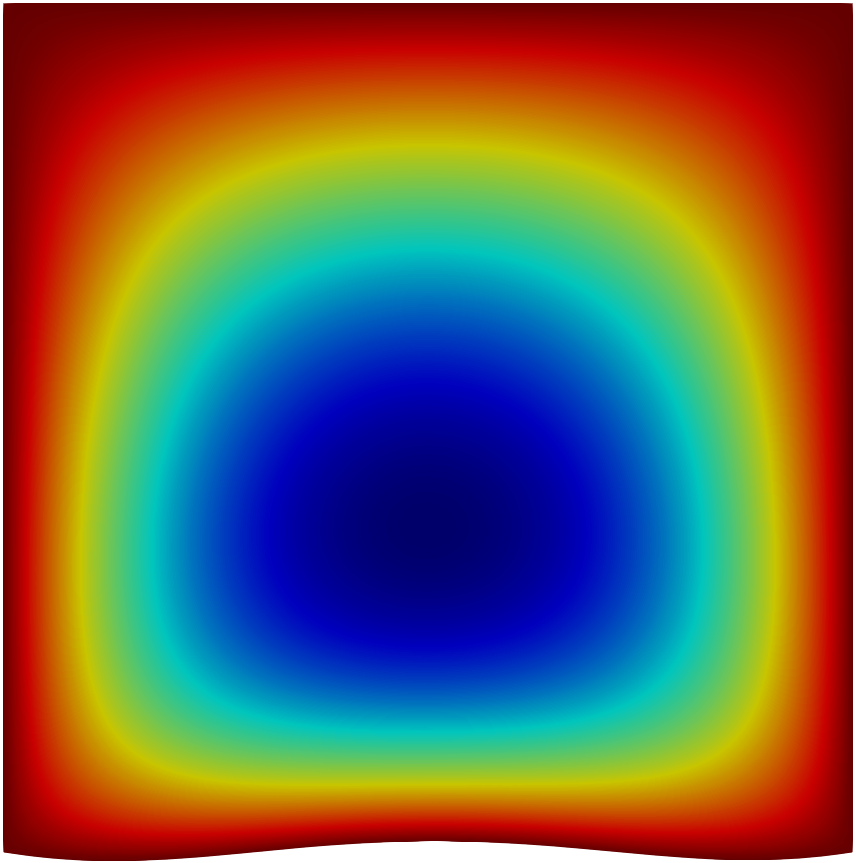}
    &\includegraphics[width=0.3\linewidth]{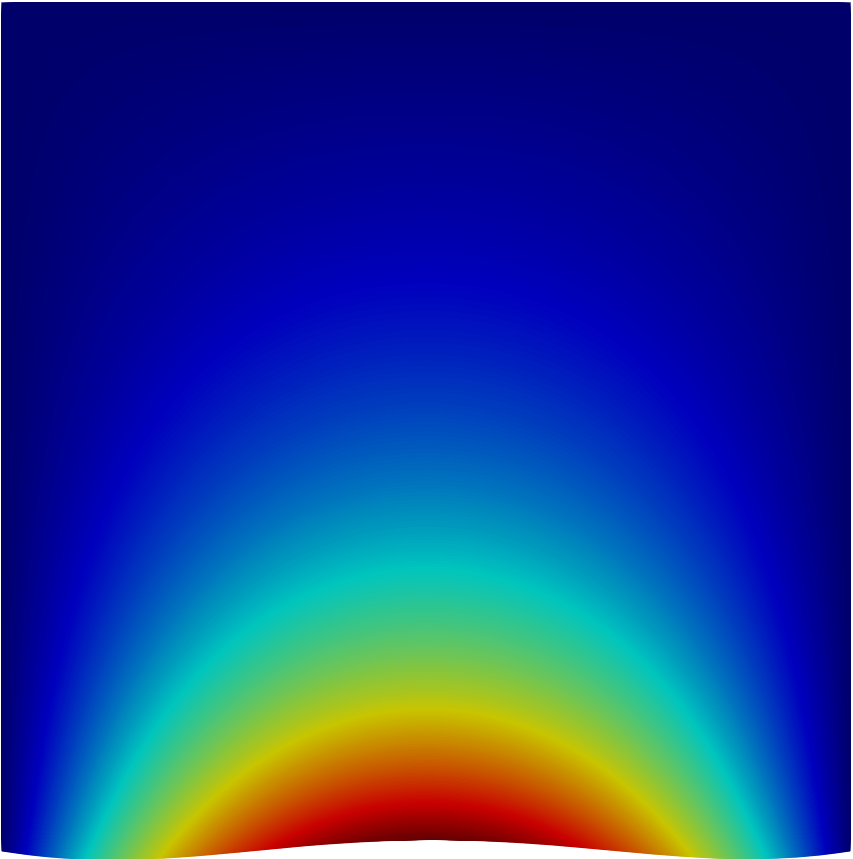}
    &\includegraphics[width=0.3\linewidth]{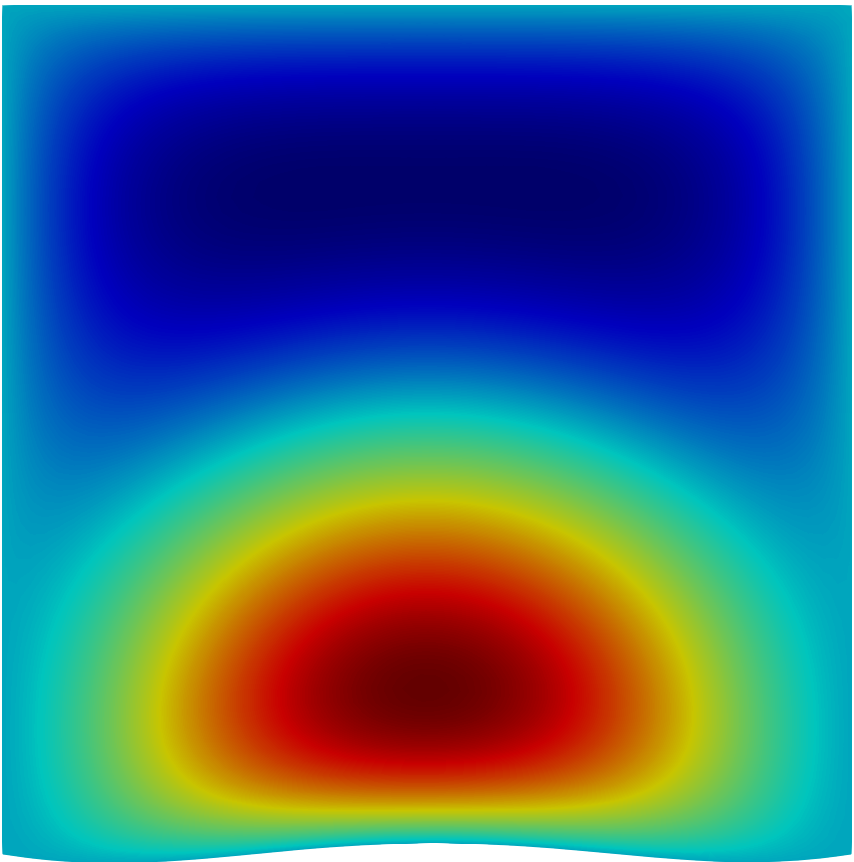}\\
    (a). $\hat{T}$ & (b). $T=\hat{T}+T_d$ &(c). $S$
    \end{tabular}
    \caption{Case\ref{Sect:Num-1}: Plot for the numerical solution on the final domain $\Omega_{5700}.$}
    \label{fig:Num-1-5}
\end{figure}

The curves $\gamma_n$ have been plotted in Fig.~\ref{fig:Num-1-4} for several iterations. It shows that the optimization is trying to lift the straight line and the lifting is converging to the blue curve. We plot the numerical solutions for the final domain $\Omega_{5700}$ in Fig.~\ref{fig:Num-1-5}-Fig.~\ref{fig:Num-1-6}.

\begin{figure}[H]
    \centering
    \begin{tabular}{cccc}
         \includegraphics[width=0.21\linewidth]{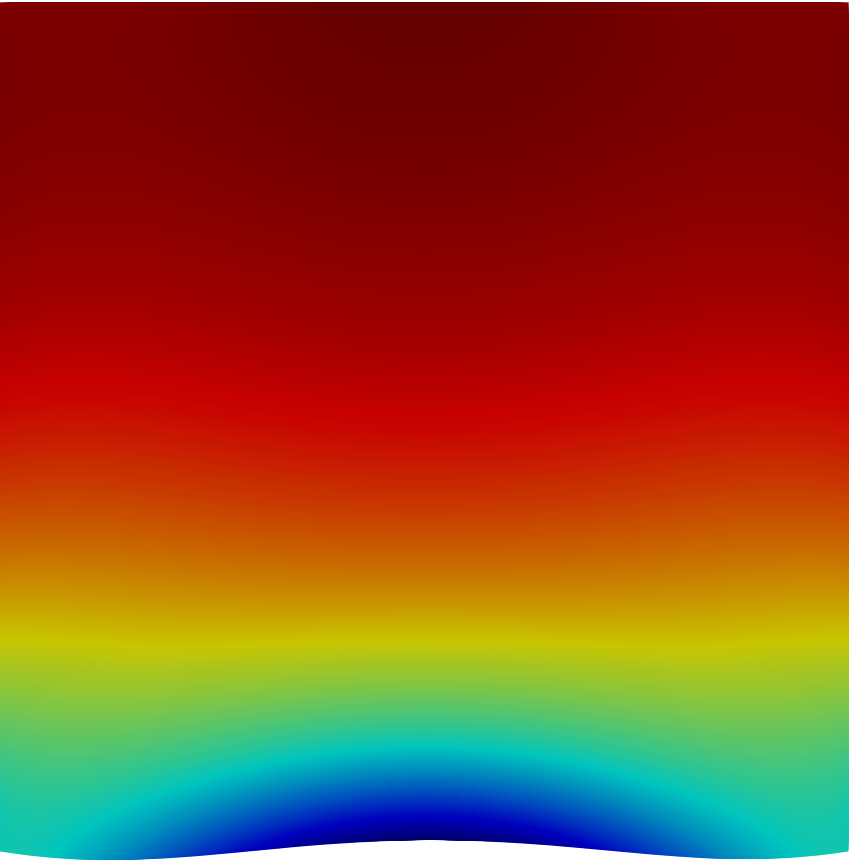}
    &\includegraphics[width=0.21\linewidth]{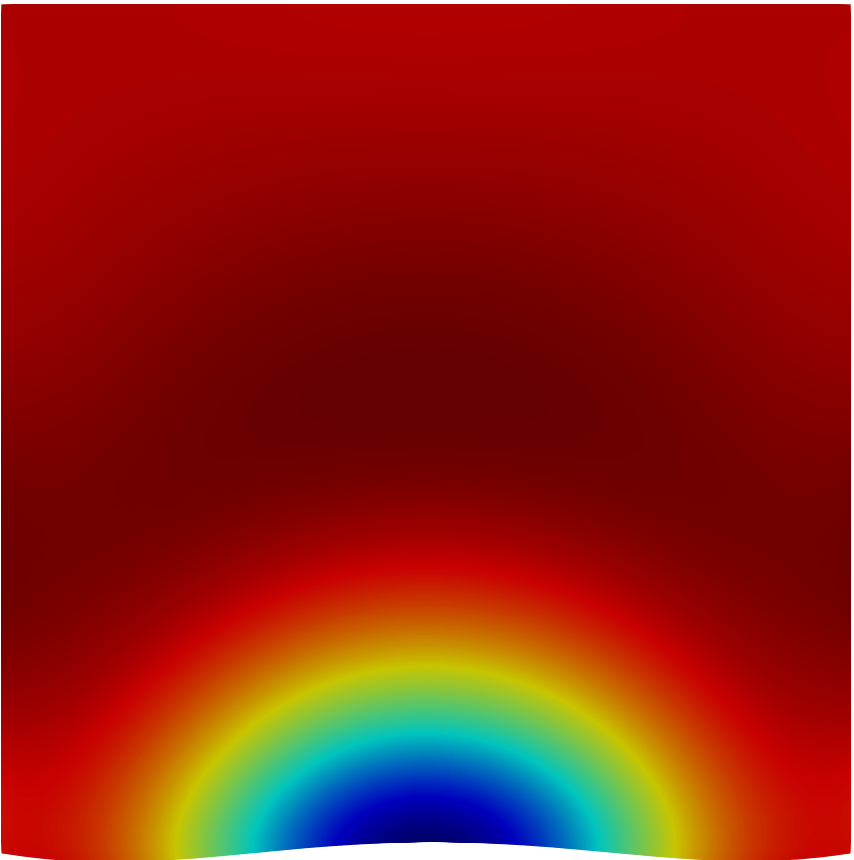}
        &\includegraphics[width=0.21\linewidth]{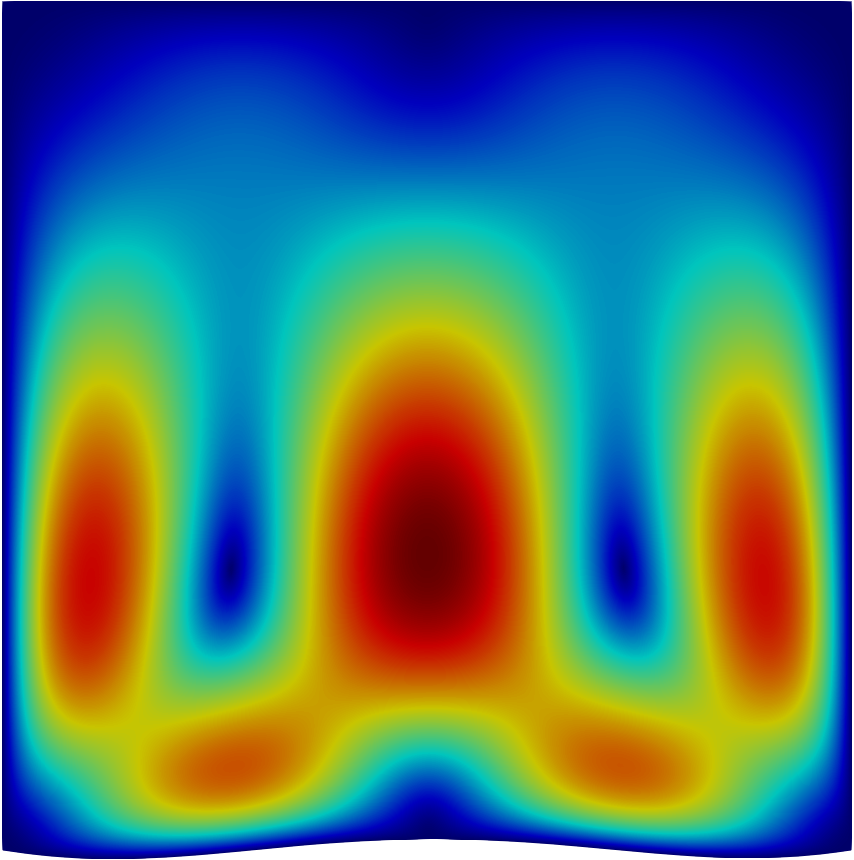}
   & \includegraphics[width=0.21\linewidth]{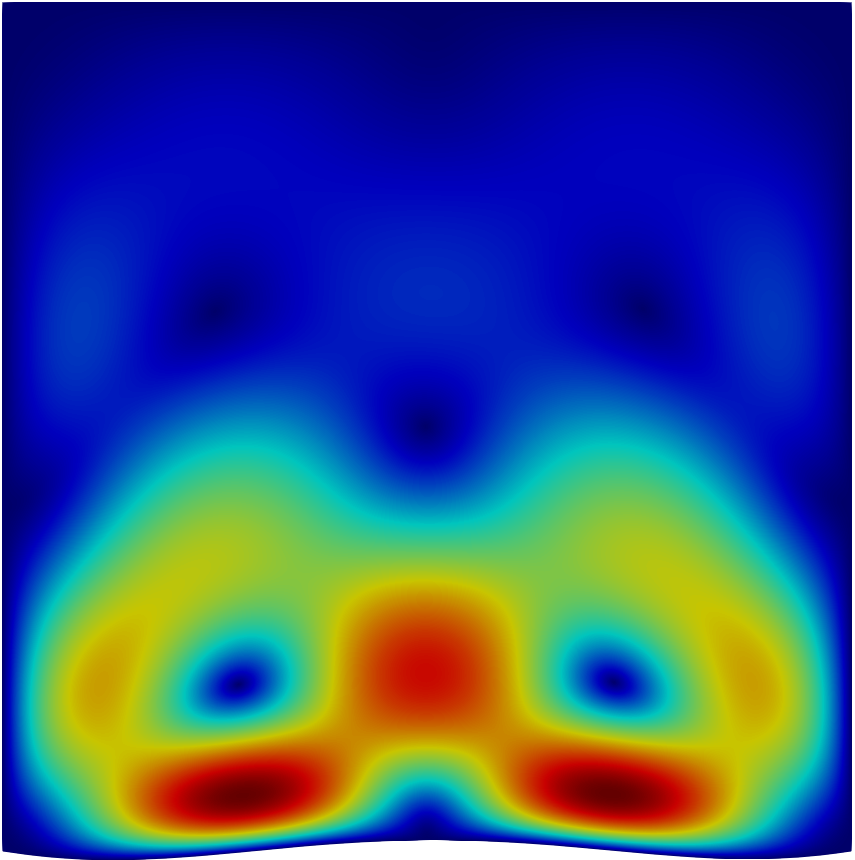}
      \\
  (a). $p$ & (b). $q$ &(c). $|{\bf v}|$ &(d). $|{\bf w}|$
    \end{tabular}
    \caption{Case\ref{Sect:Num-1}: Plot for the numerical solution on the final domain $\Omega_{5700}.$}
    \label{fig:Num-1-6}
\end{figure}

\subsection{Case 2: initial state with $\gamma: y = -0.1\sin(3\pi x)$}\label{Sect:Num-2}
As shown in the above test, the final curve is approaching to a $\sin$ curve and we shall try the initial curve $\gamma_0$ in the gradient descent scheme. The numerical solutions have been plotted in Fig.~\ref{fig:Num-2-1}-Fig.~\ref{fig:Num-2-2}. Similar conclusions as Section~\ref{Sect:Num-1} can be obtained.

\begin{figure}[H]
    \centering
    \begin{tabular}{ccc}
    \includegraphics[width=0.3\linewidth]{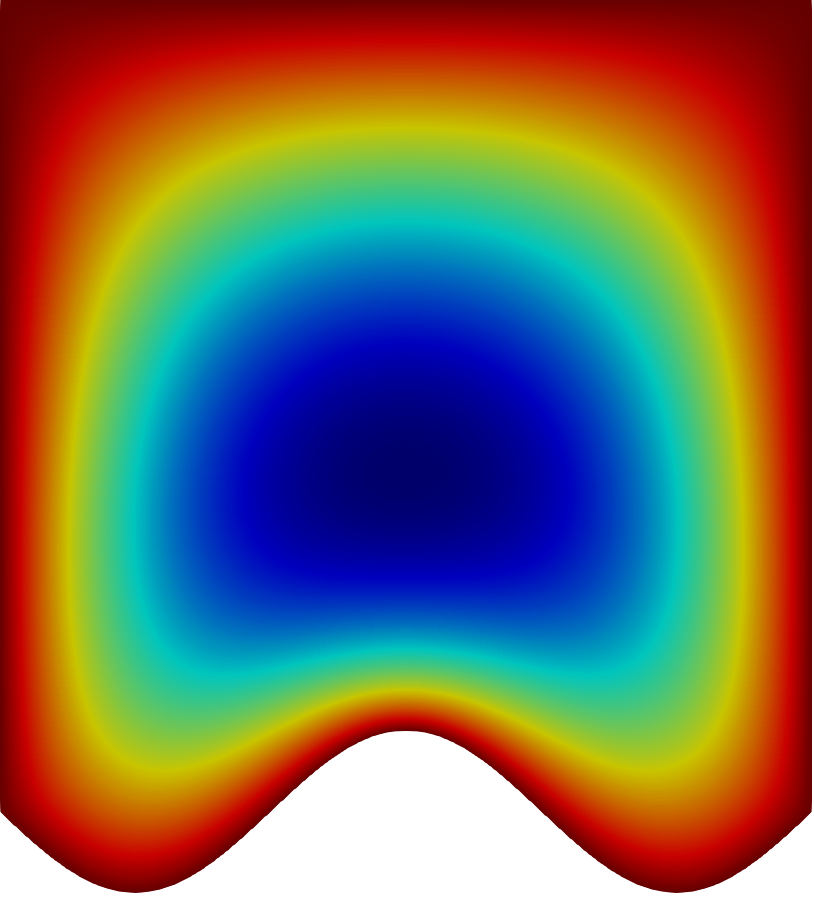}
    &\includegraphics[width=0.3\linewidth]{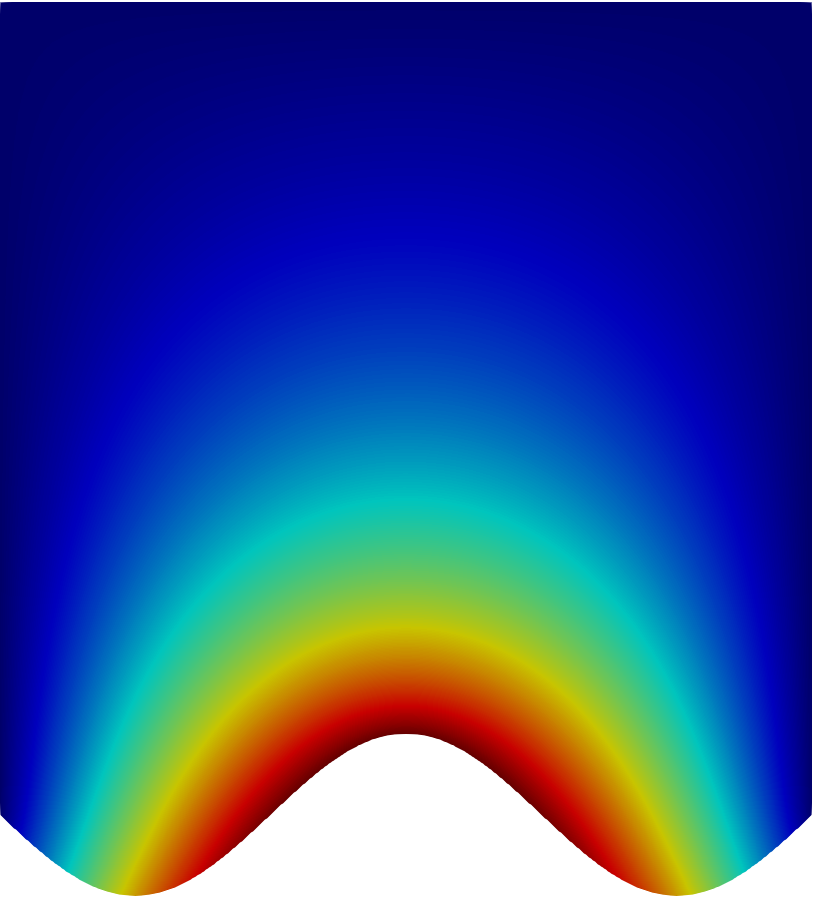}
    &\includegraphics[width=0.3\linewidth]{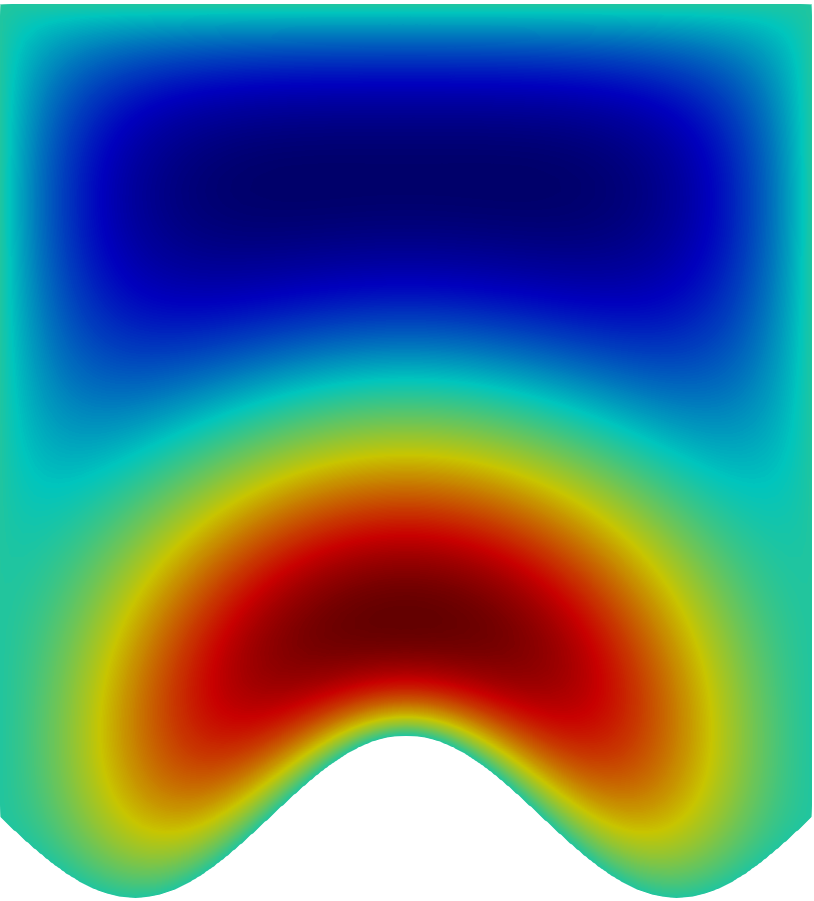}
    \\
    (a). $\hat{T}$ & (b). $T=\hat{T}+T_d$ &(c). $S$
    \end{tabular}
    \caption{Case\ref{Sect:Num-2}: Plot for the numerical solution on the initial domain $\Omega_0.$}
    \label{fig:Num-2-1}
\end{figure}

\begin{figure}[H]
    \centering
    \begin{tabular}{cccc}
         \includegraphics[width=0.21\linewidth]{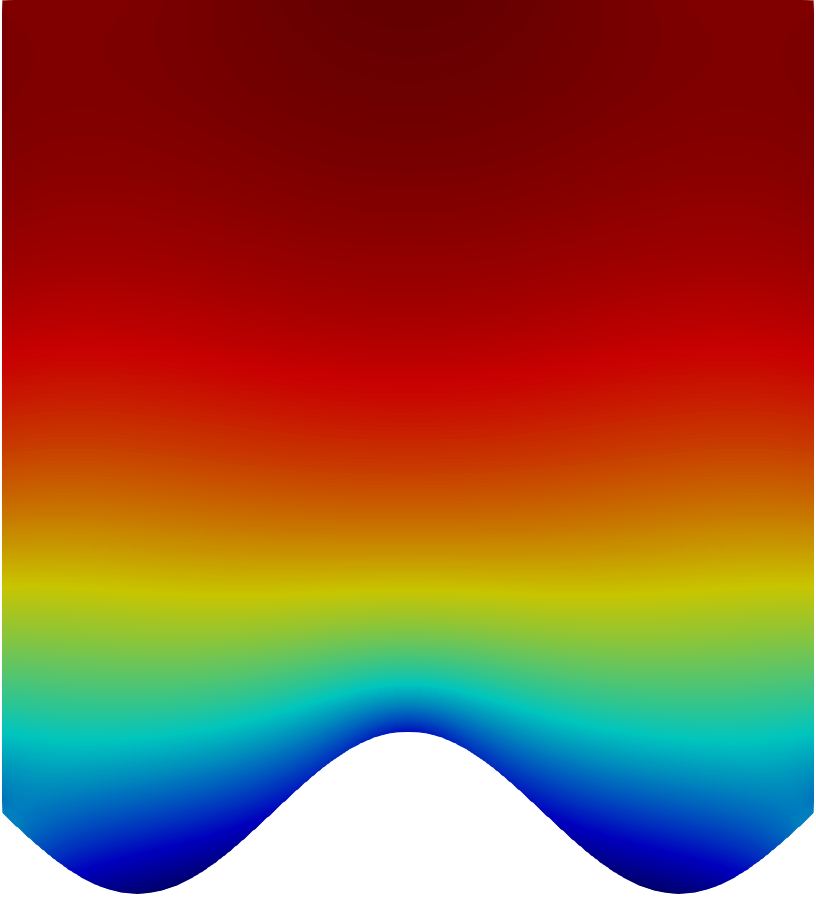}
    &\includegraphics[width=0.21\linewidth]{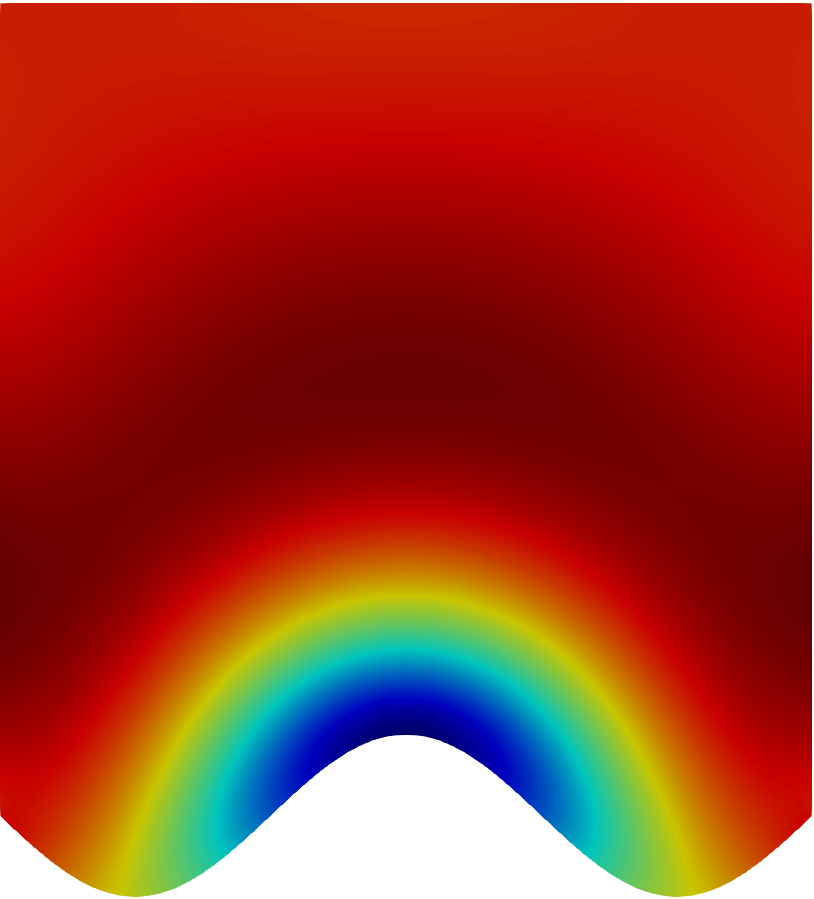}
        &\includegraphics[width=0.21\linewidth]{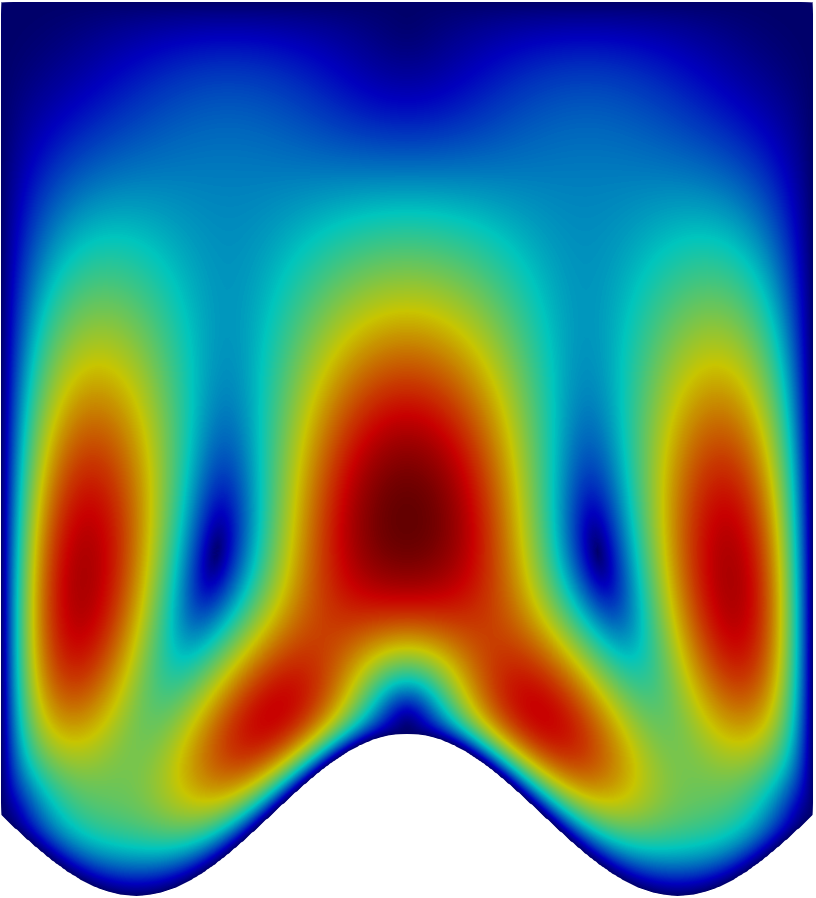}
   & \includegraphics[width=0.21\linewidth]{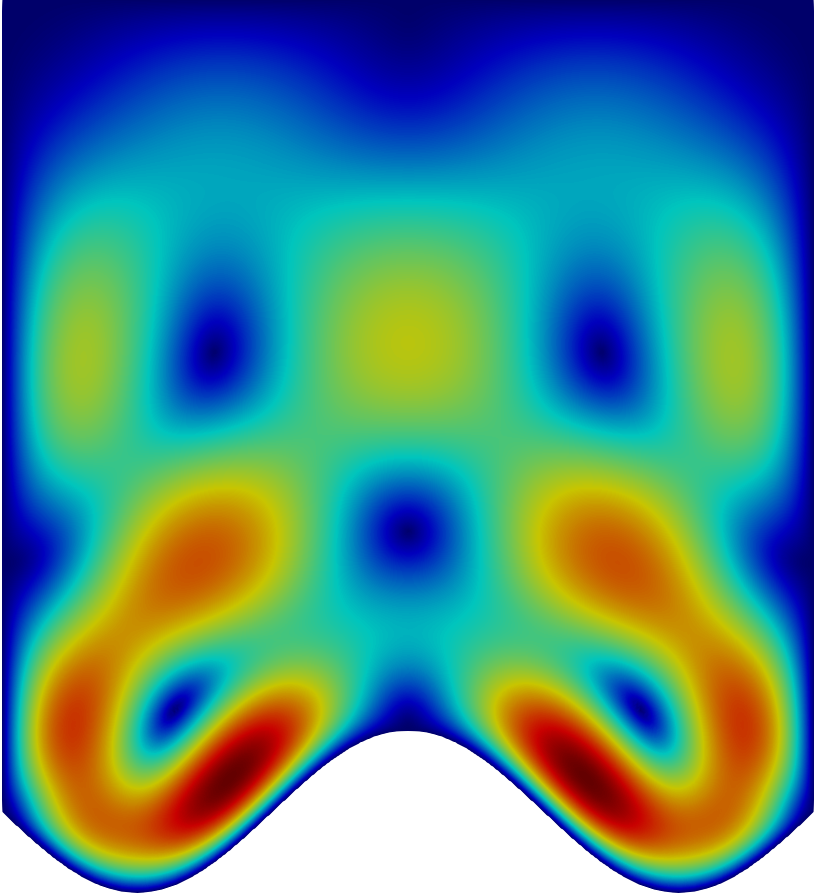}
      \\
  (a). $p$ & (b). $q$ &(c). $|{\bf v}|$ &(d). $|{\bf w}|$
    \end{tabular}
    \caption{Case\ref{Sect:Num-2}: Plot for the numerical solution on the initial domain $\Omega_0.$}
    \label{fig:Num-2-2}
\end{figure}

The decaying behavior for the cost functional at each iterations has been plotted in Fig.~\ref{fig:Num-2-3}. In contrast to Fig.~\ref{fig:Num-1-3}, the cost functional $J_1$ is increasing but the curvature contribution is decreasing. Also the total cost functional $J$ is decreasing until Iter$=4000$. It shows that our proposed cost functional $J$ is effective to balance each term. 

\begin{figure}[H]
    \centering
    \includegraphics[width=0.8\linewidth]{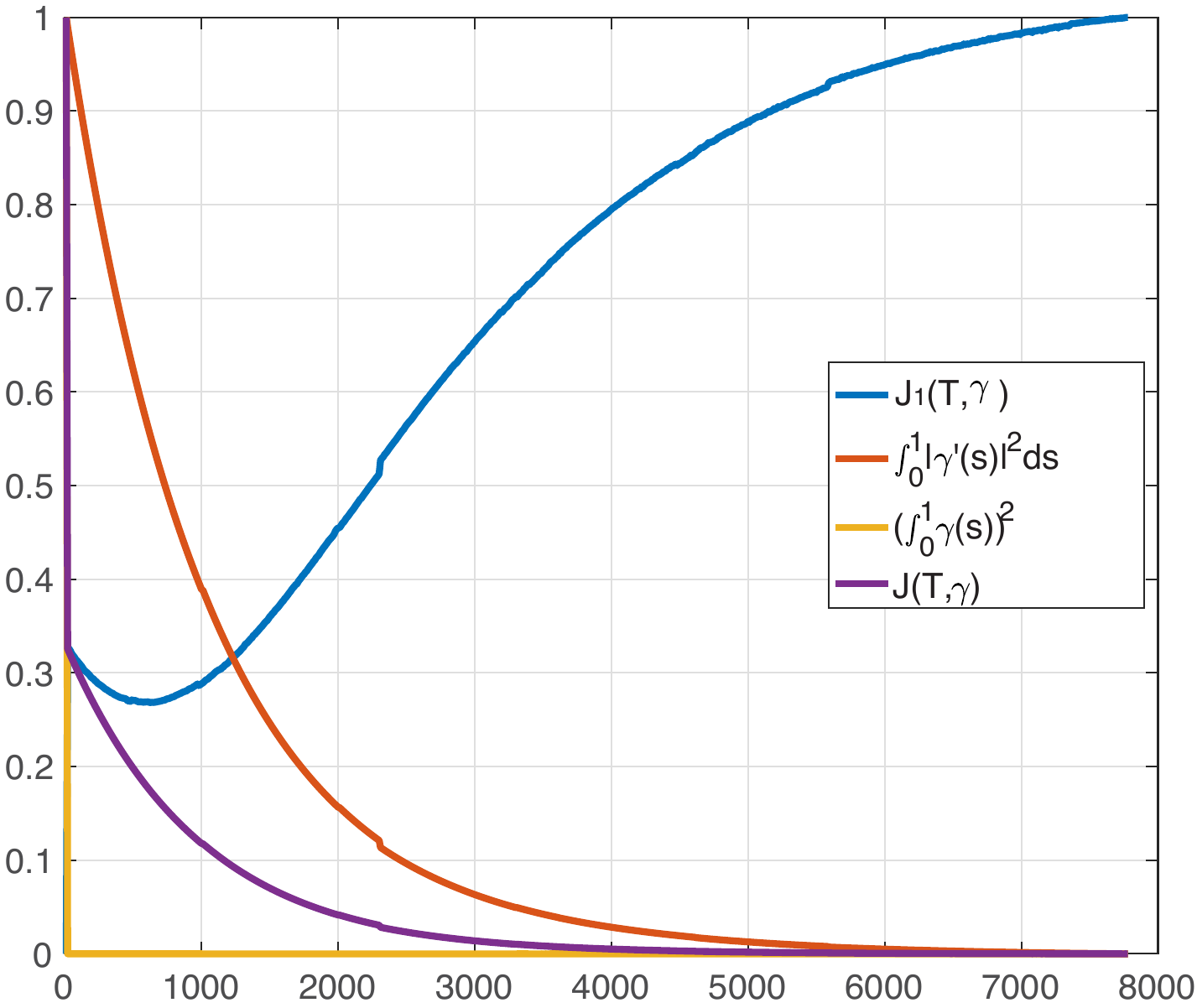}
    \caption{Case~\ref{Sect:Num-2}: Convergence test of the cost functional.
    }
    \label{fig:Num-2-3}
\end{figure}

Then, we plot the curves $\gamma_n$ for several iterations $n$ in Fig.~\ref{fig:Num-2-4}. In this test, the optimization algorithm is straightening the curve bottom to the blue curve with lower amplitude. In addition, the convergence behavior can also be observed in this figure.

\begin{figure}[H]
    \centering
    \includegraphics[width=0.8\linewidth]{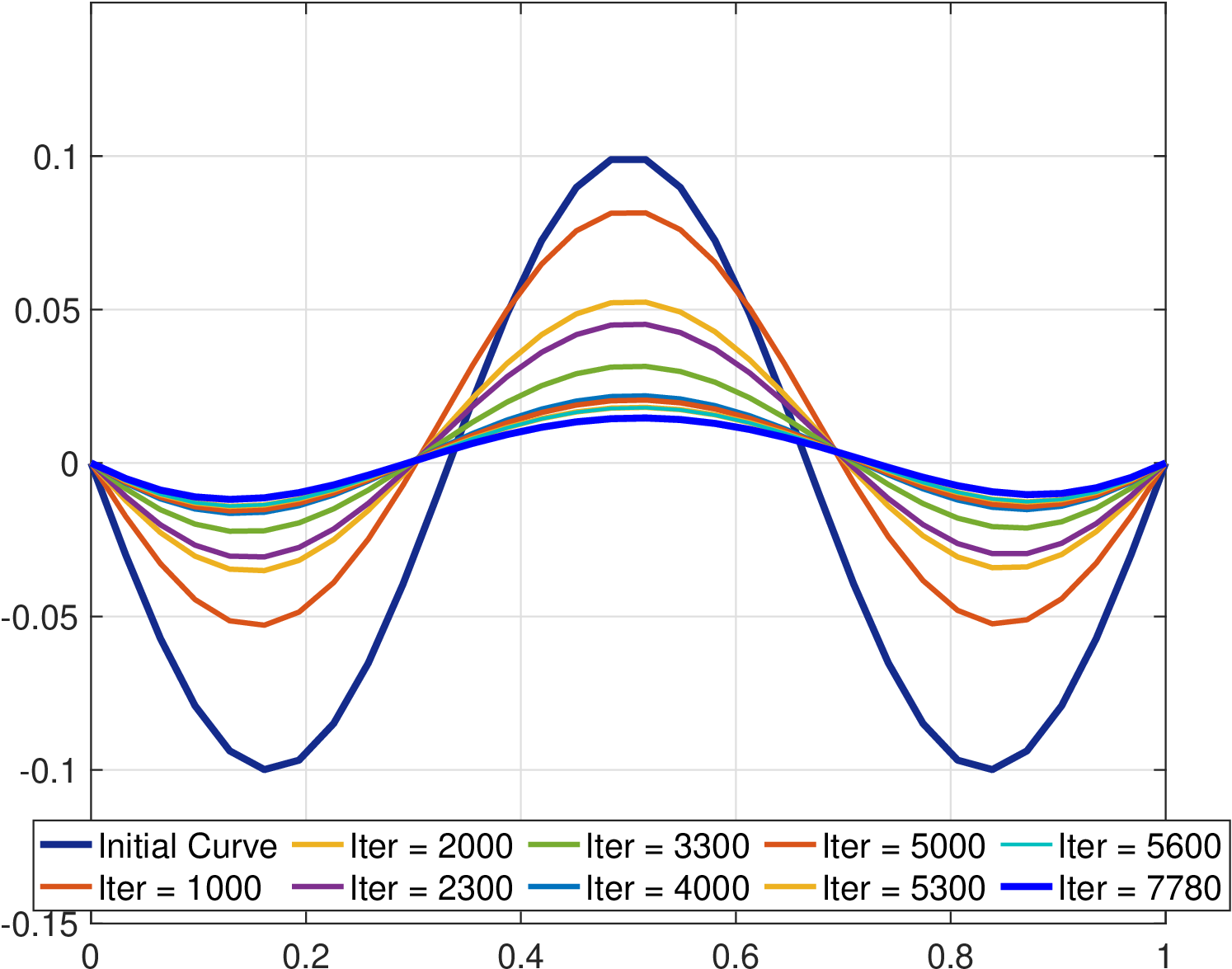}
    \caption{Case\ref{Sect:Num-2}: Plot of optimized curves.}
    \label{fig:Num-2-4}
\end{figure}

Then the numerical solutions on the final domain $\Omega_{5600}$ have been plotted in Fig.~\ref{fig:Num-2-5}-Fig.~\ref{fig:Num-2-6}.
\begin{figure}[H]
    \centering
    \begin{tabular}{ccc}
    \includegraphics[width=0.3\linewidth]{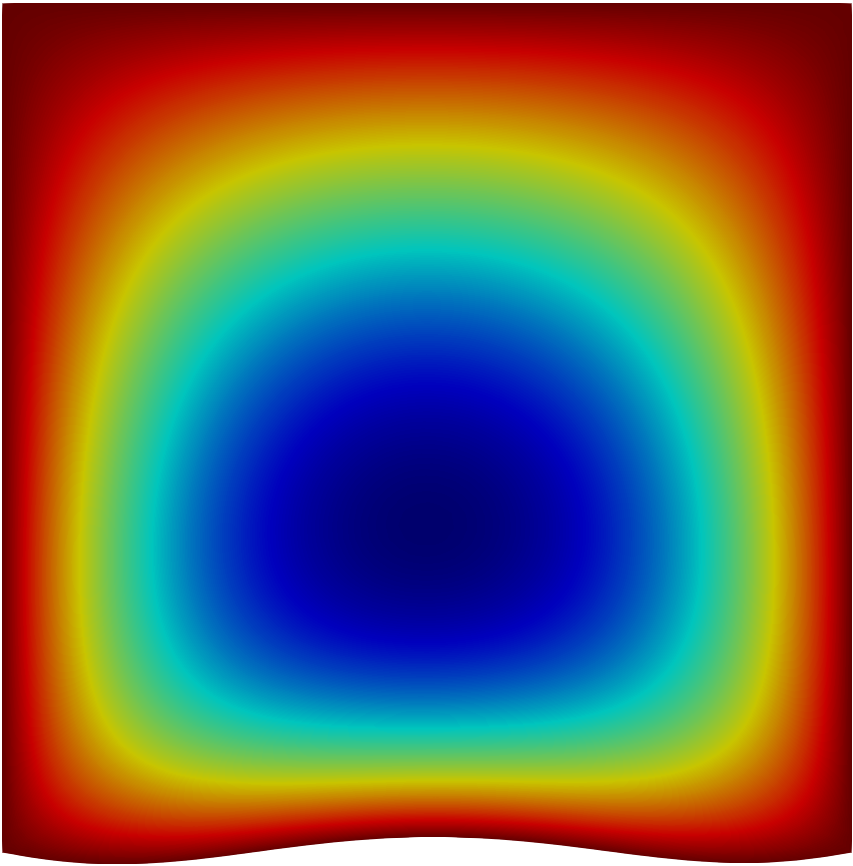}
    &\includegraphics[width=0.3\linewidth]{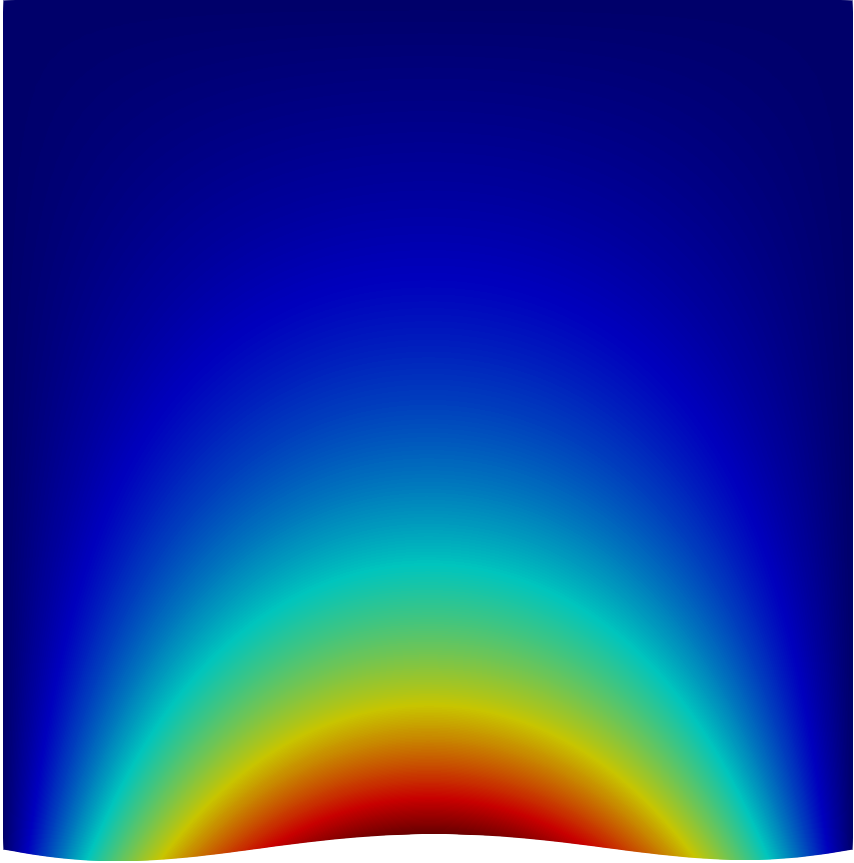}
    &\includegraphics[width=0.3\linewidth]{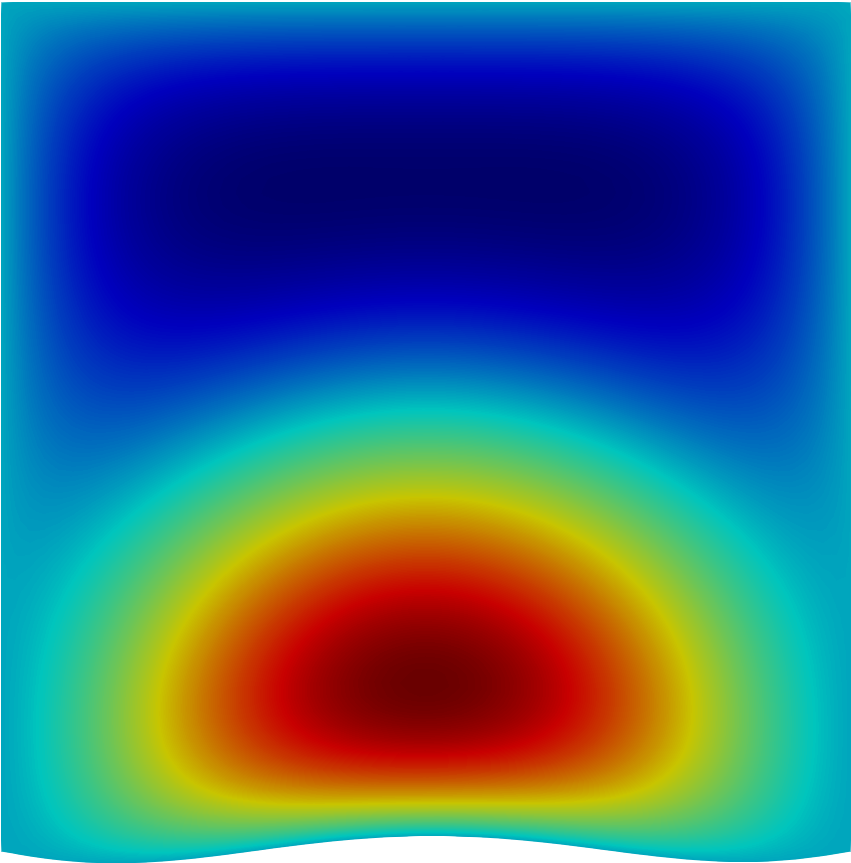}\\
    (a). $\hat{T}$ & (b). $T=\hat{T}+T_d$ & (c). $S$
    \end{tabular}
    \caption{Case\ref{Sect:Num-2}: Plot for the numerical solution on the final domain $\Omega_{5600}.$}
    \label{fig:Num-2-5}
\end{figure}

\begin{figure}[H]
    \centering
    \begin{tabular}{cccc}
         \includegraphics[width=0.21\linewidth]{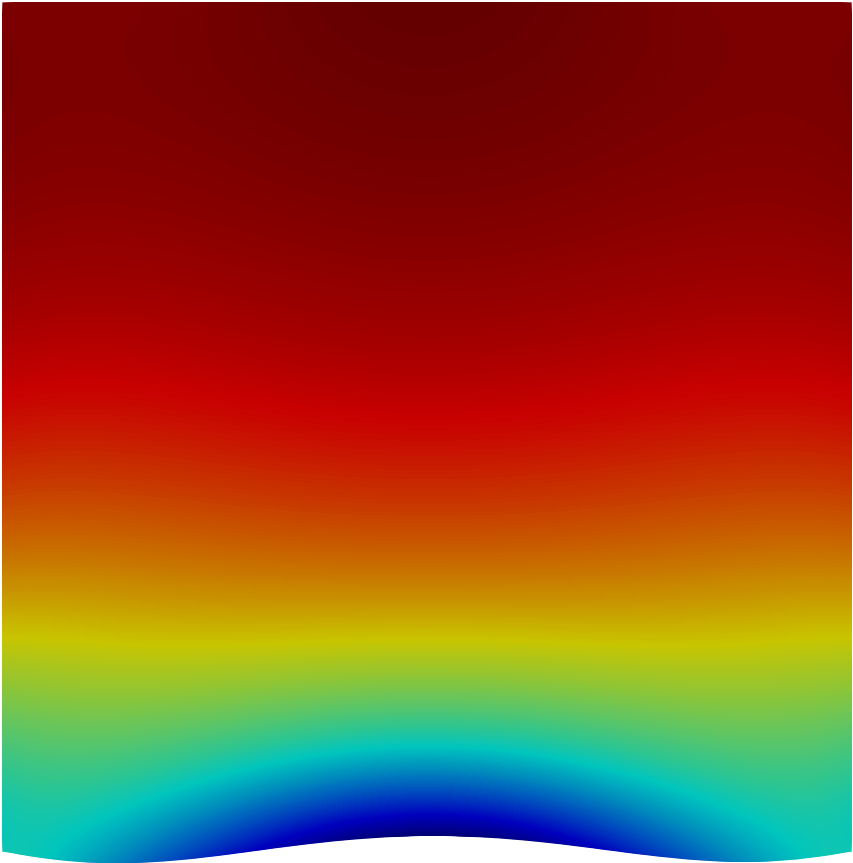}
    &\includegraphics[width=0.21\linewidth]{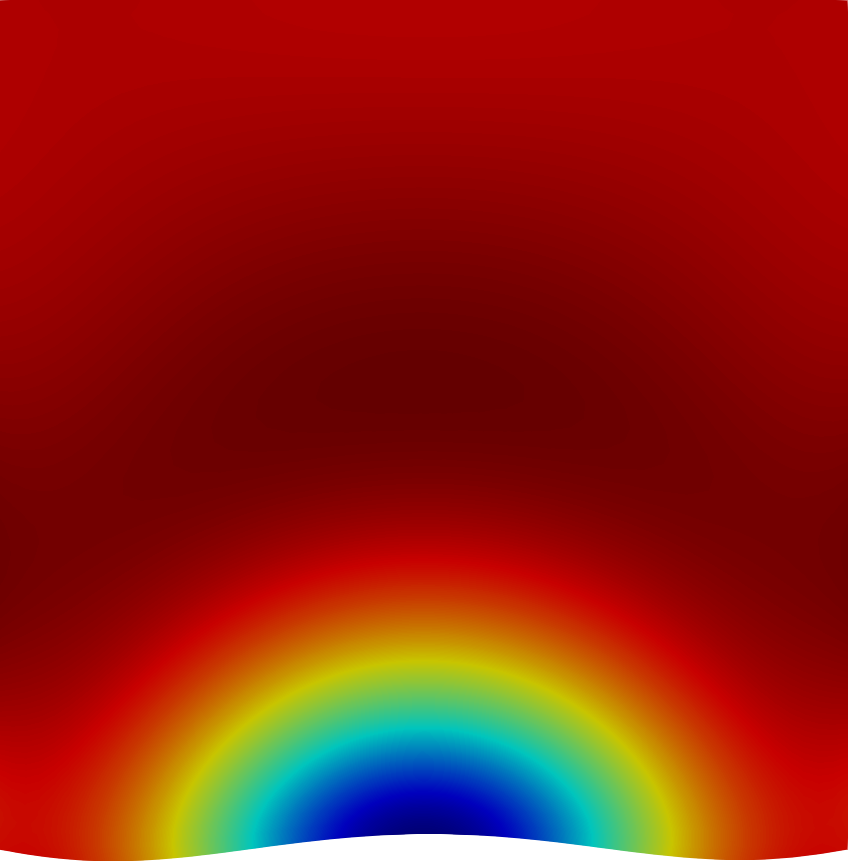}
        &\includegraphics[width=0.21\linewidth]{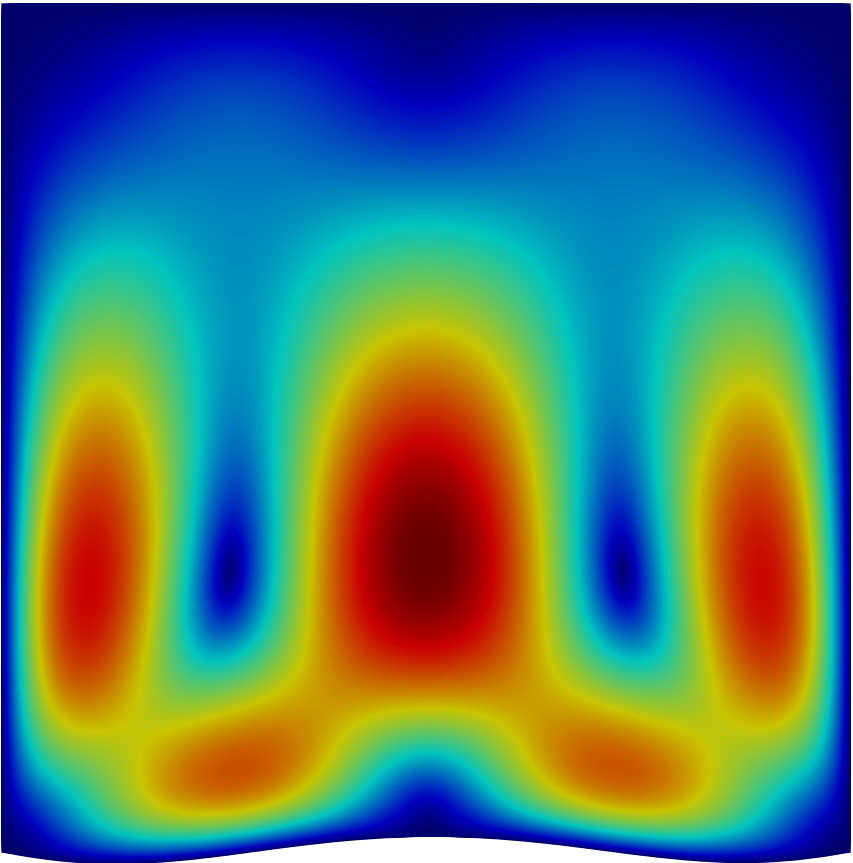}
   & \includegraphics[width=0.21\linewidth]{./figure/Test1_wmag_final}
      \\
  (a). $p$ & (b). $q$ &(c). $|{\bf v}|$ &(d). $|{\bf w}|$
    \end{tabular}
    \caption{Case\ref{Sect:Num-2}: Plot for the numerical solution on the final domain $\Omega_{5600}.$}
    \label{fig:Num-2-6}
\end{figure}

\subsection{Case 3: initial state with $\gamma: y = -0.1\sin(5\pi x)$}\label{Sect:Num-3}

In this test, we shall start with a curvy bottom with more oscillations $\gamma_0 = -0.1\sin(5\pi x)$. The numerical solutions have been plotted in Fig.~\ref{fig:Num-3-1}-Fig.~\ref{fig:Num-3-2}. The plot of numerical solutions shows similar patterns as before.

\begin{figure}[H]
    \centering
    \begin{tabular}{ccc}
    \includegraphics[width=0.3\linewidth]{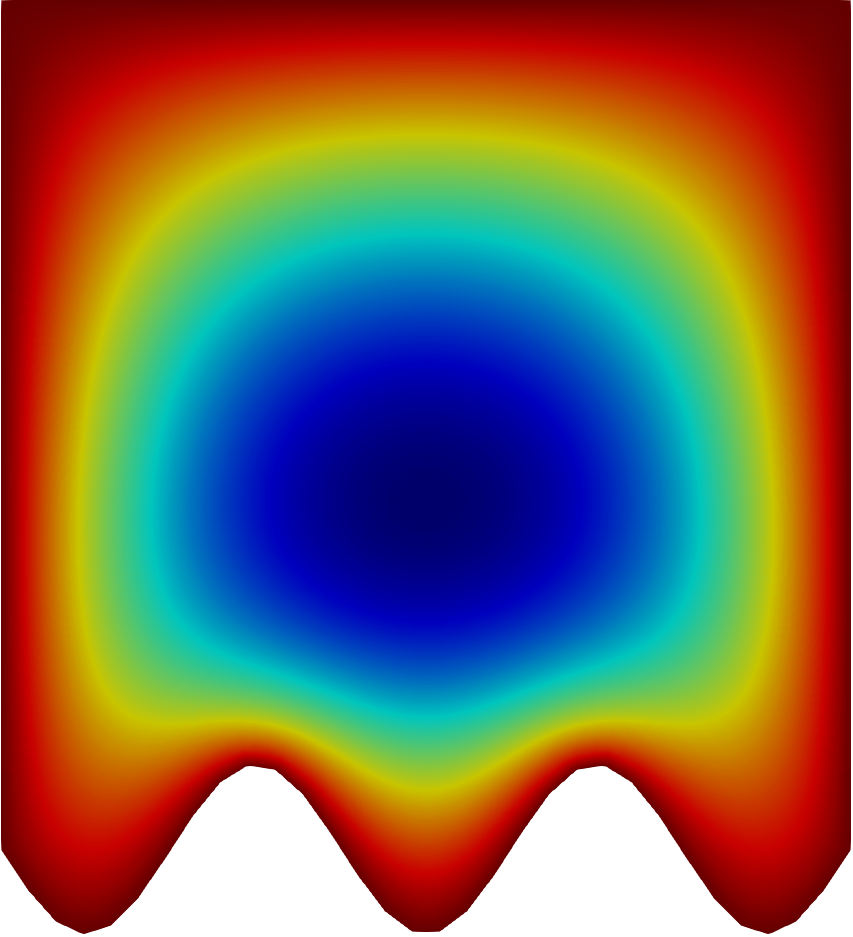}
    &\includegraphics[width=0.3\linewidth]{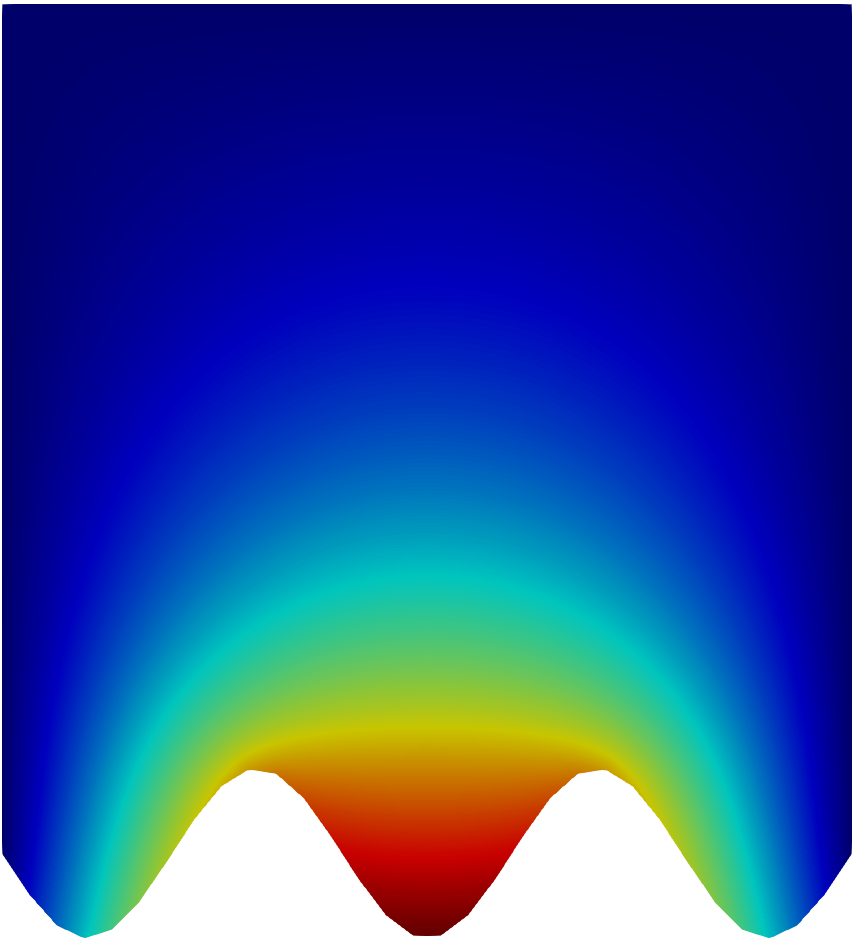}
    &\includegraphics[width=0.3\linewidth]{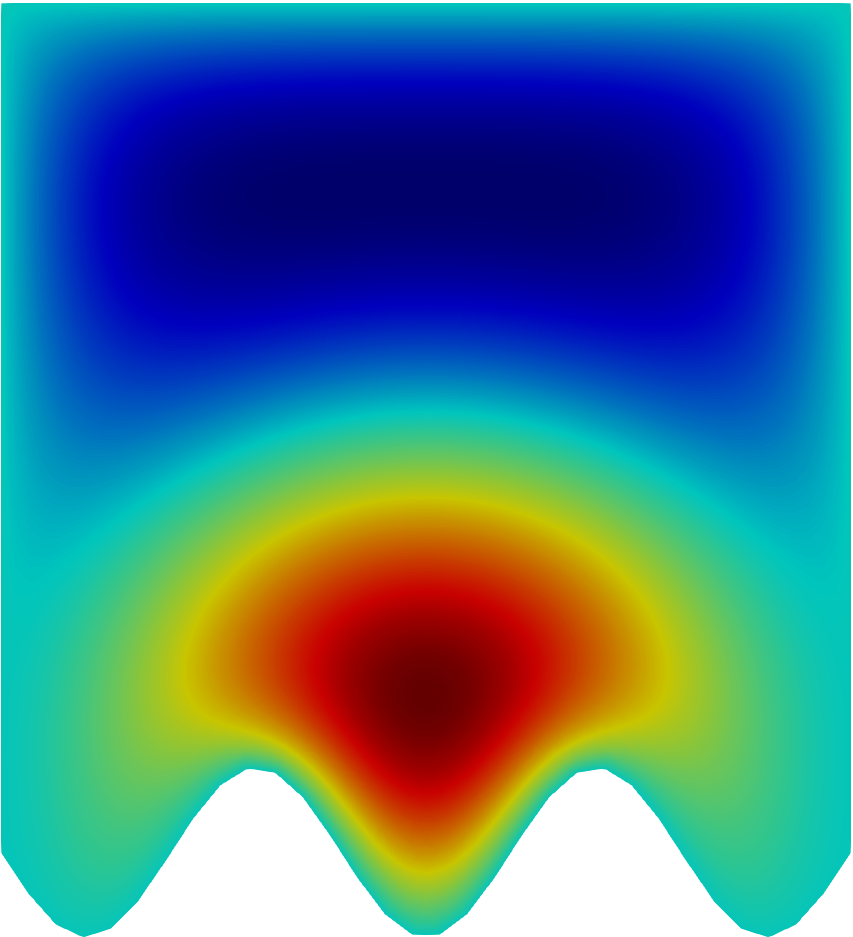}
    \\
    (a). $\hat{T}$ & (b). $T = \hat{T} + T_d$ &(c). $S$
    \end{tabular}
    \caption{Case\ref{Sect:Num-3}: Plot for the numerical solution on the initial domain $\Omega_0.$}
    \label{fig:Num-3-1}
\end{figure}

\begin{figure}[H]
    \centering
    \begin{tabular}{cccc}         \includegraphics[width=0.21\linewidth]{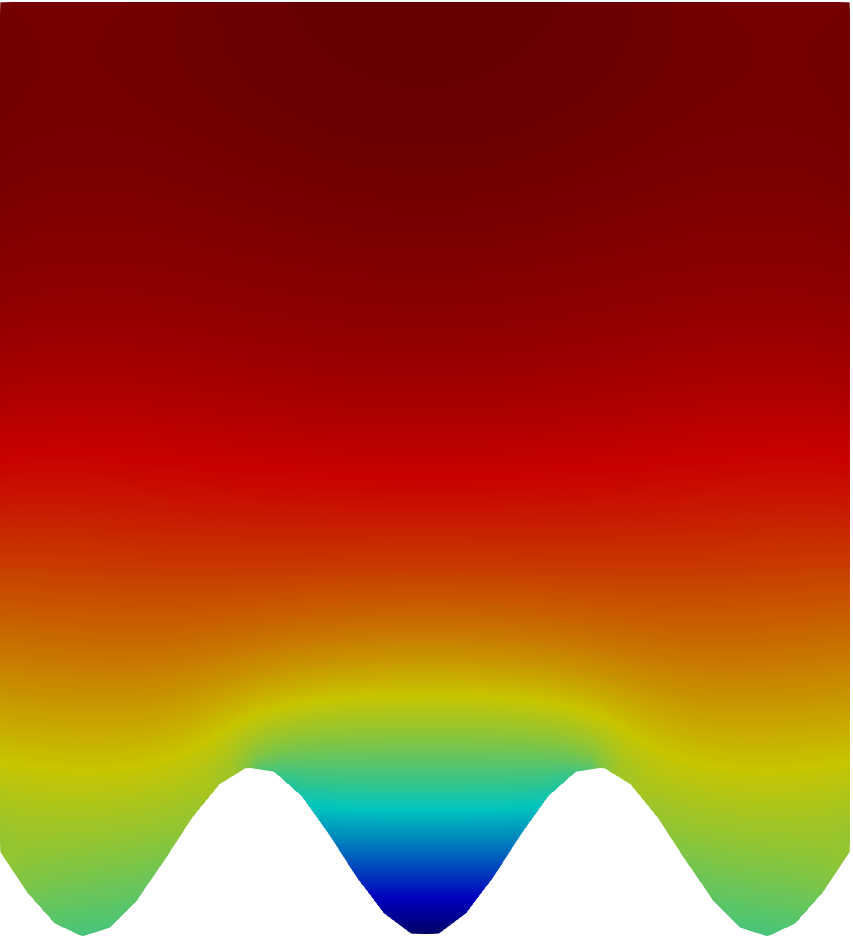}
    &\includegraphics[width=0.21\linewidth]{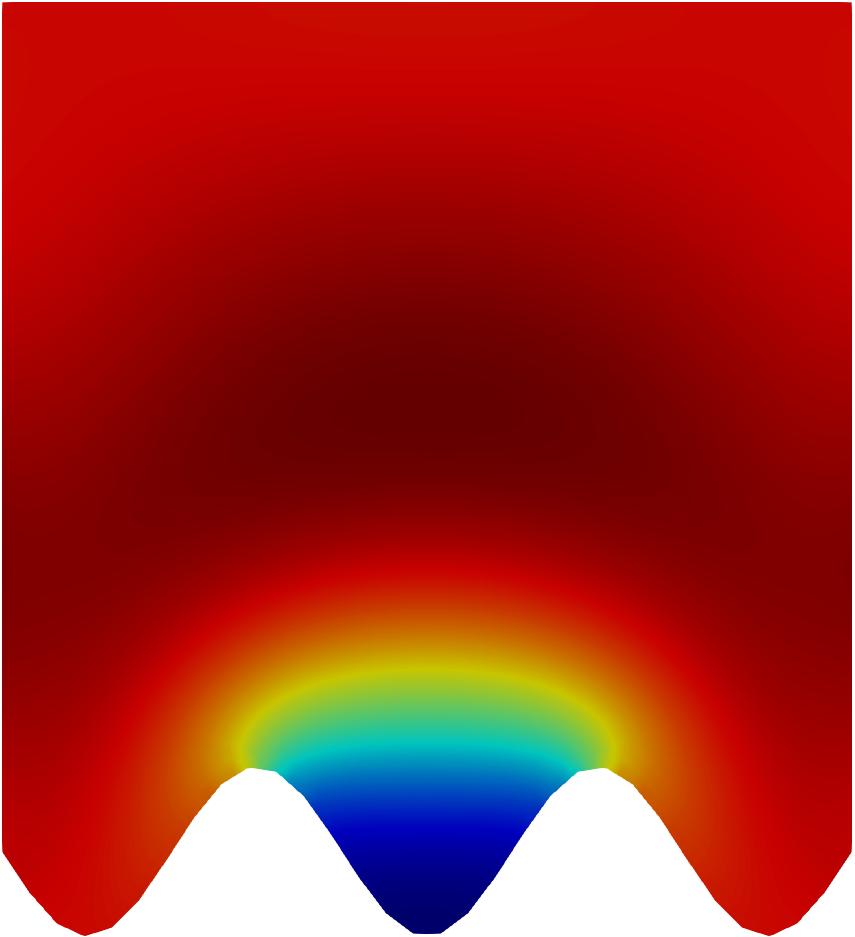}
 &\includegraphics[width=0.21\linewidth]{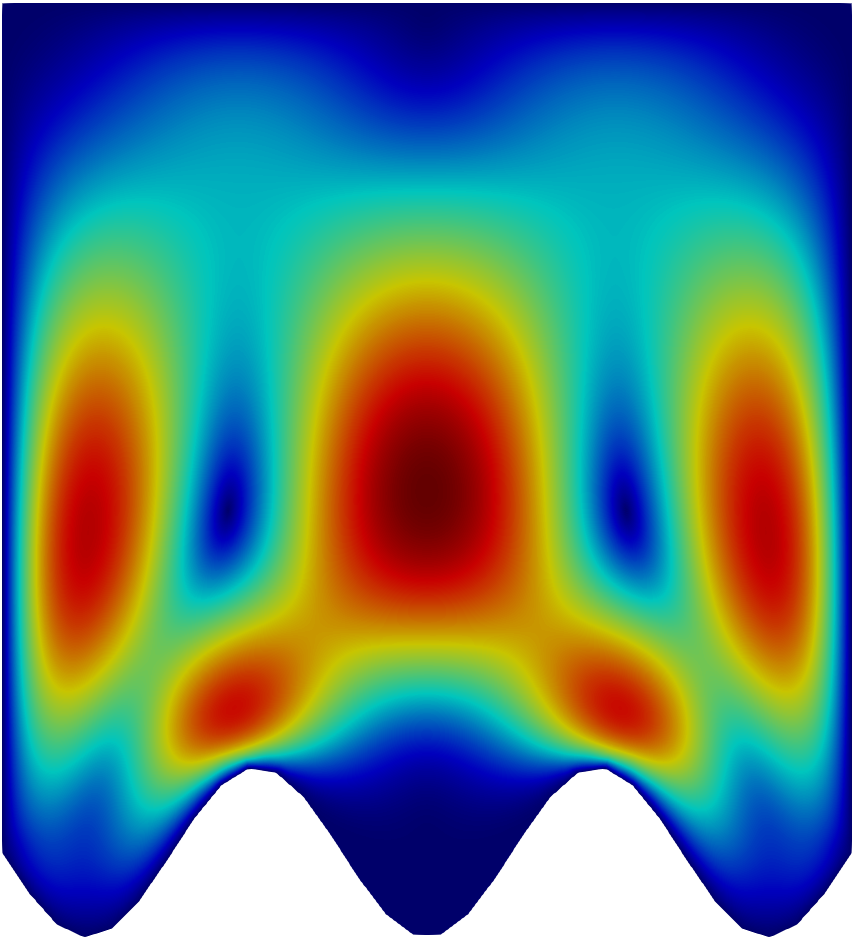}
   & \includegraphics[width=0.21\linewidth]{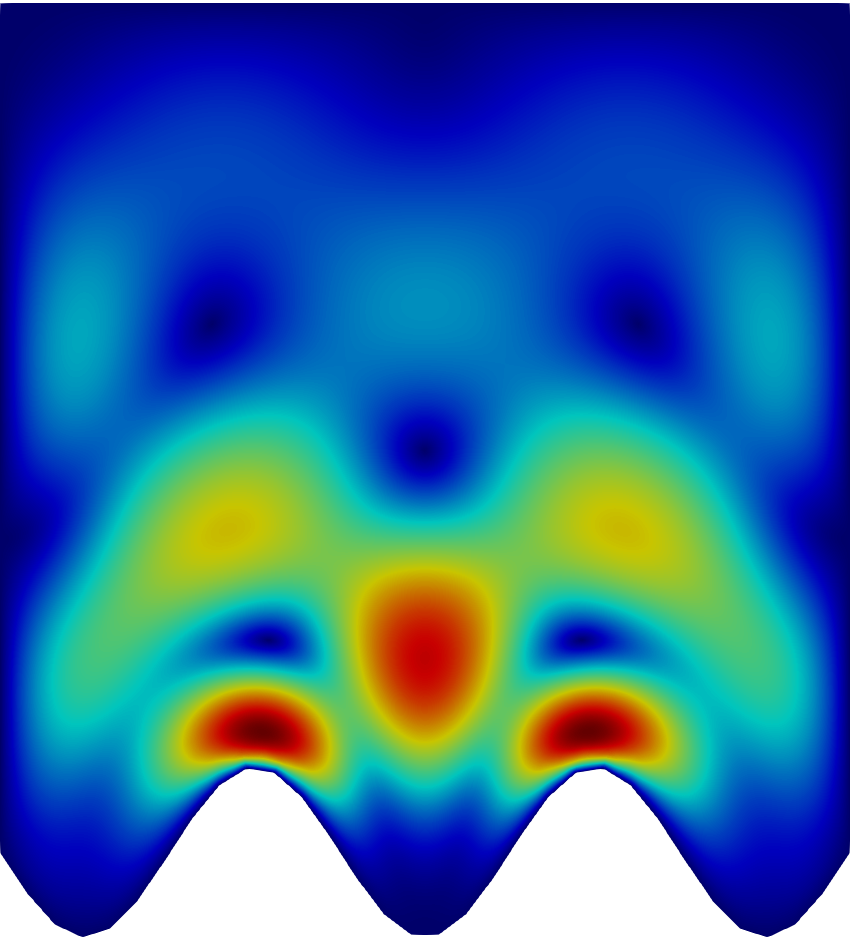}
      \\
  (a). $p$ & (b). $q$ &(c). $|{\bf v}|$ &(d). $|{\bf w}|$
    \end{tabular}
    \caption{Case\ref{Sect:Num-3}: Plot for the numerical solution on the initial domain $\Omega_0.$}
    \label{fig:Num-3-2}
\end{figure}

\begin{figure}[H]
    \centering
    \begin{tabular}{cc}
    \includegraphics[width=0.45\linewidth]{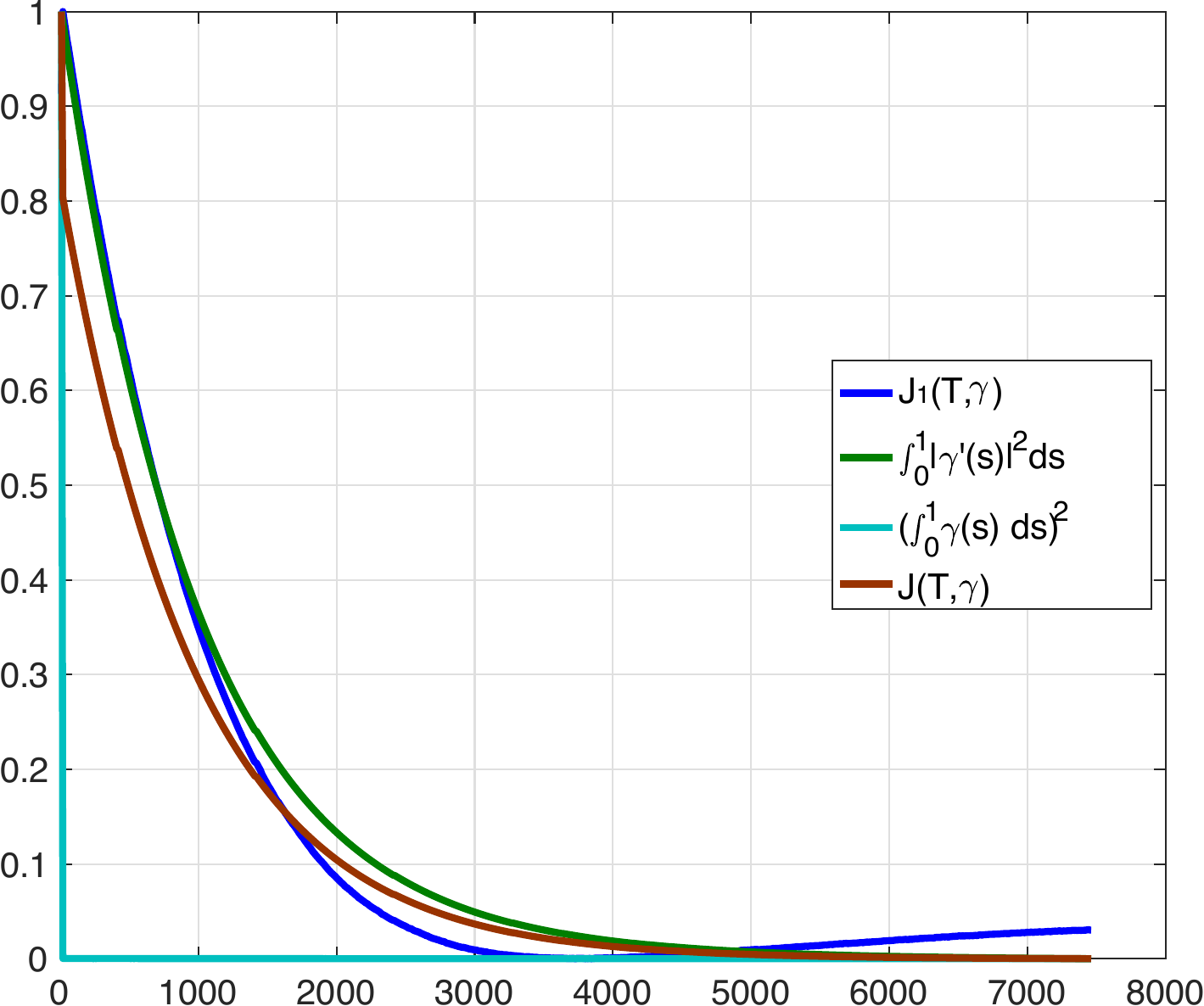}
    &\includegraphics[width=0.45\linewidth]{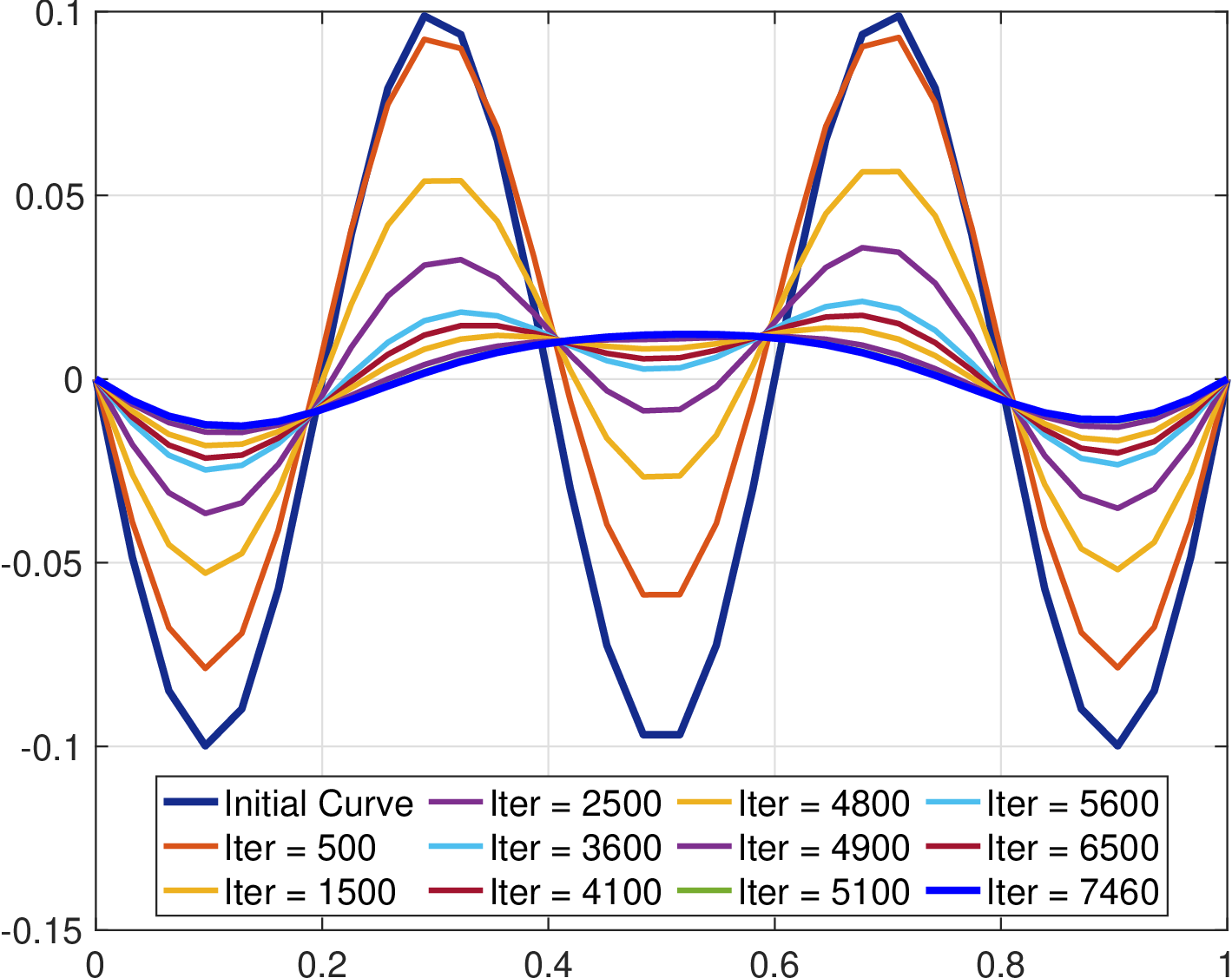}\\
    (a). Cost functional & (b). $\gamma_n$
    \end{tabular}
    \caption{Case~\ref{Sect:Num-3}: Convergence test of the cost functional and optimization curves. 
    }
    \label{fig:Num-3-3}
\end{figure}

Next, we plot the convergence behavior for the cost functional and the optimization iterations in Fig.~\ref{fig:Num-3-3}. In this test, all the terms in the cost functional are decaying until Iter = 5000. After this iteration, the other terms are keeping decaying but the $J_1$ term will be increasing a bit. In fact, the optimization process will contribute to both lifting and lower the initial curves. Thus curvature contribution needs to be optimized in order to derive the final curve. The convergence of curve optimization for several iterations has been plotted in Fig.~\ref{fig:Num-3-3}b. The numerical solutions on the final optimized domain $\Omega_{7460}$ have patterns similar to those in the above sections, and we will omit these plots.

\subsection{Case 4: initial state with $\gamma: y = -0.01\sin(7\pi x)$}\label{Sect:Num-4}
Lastly, we take the initial bottom curve is chosen with less amplitude. Several $\sin$ functions have been tried to select the least value in the cost functional. We shall take  $\gamma_0: y = -0.01\sin(7\pi x)$ for this testing. The numerical solutions on the initial domain have been plotted in Fig.~\ref{fig:Num-4-1}-Fig.~\ref{fig:Num-4-2}.

\begin{figure}[H]
    \centering
    \begin{tabular}{ccc}
    \includegraphics[width=0.3\linewidth]{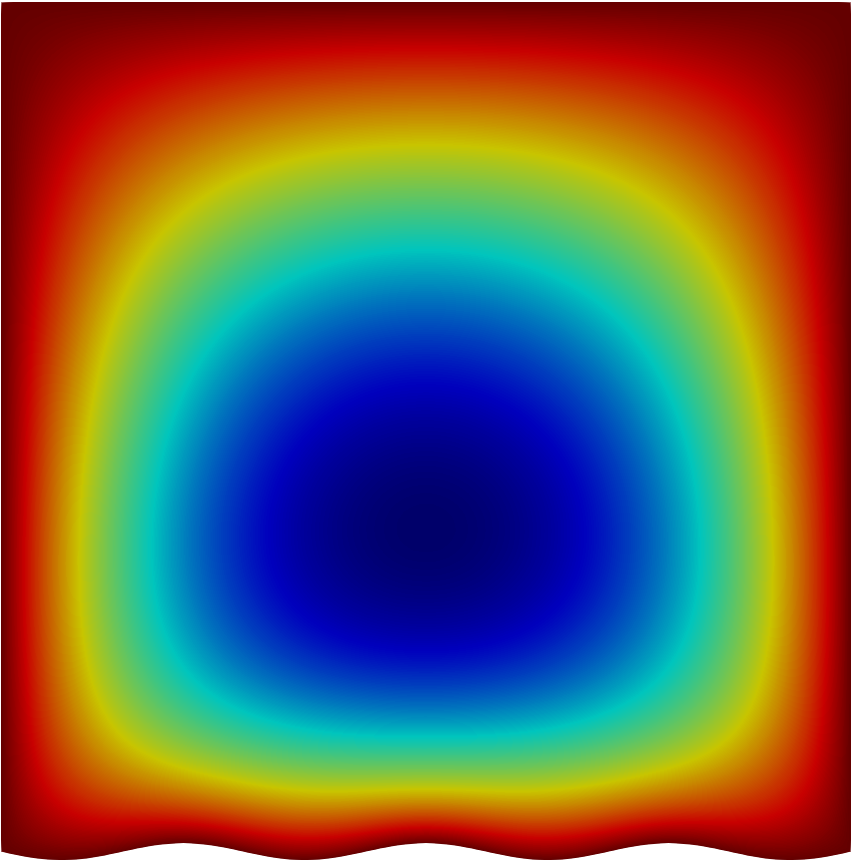}
    &\includegraphics[width=0.3\linewidth]{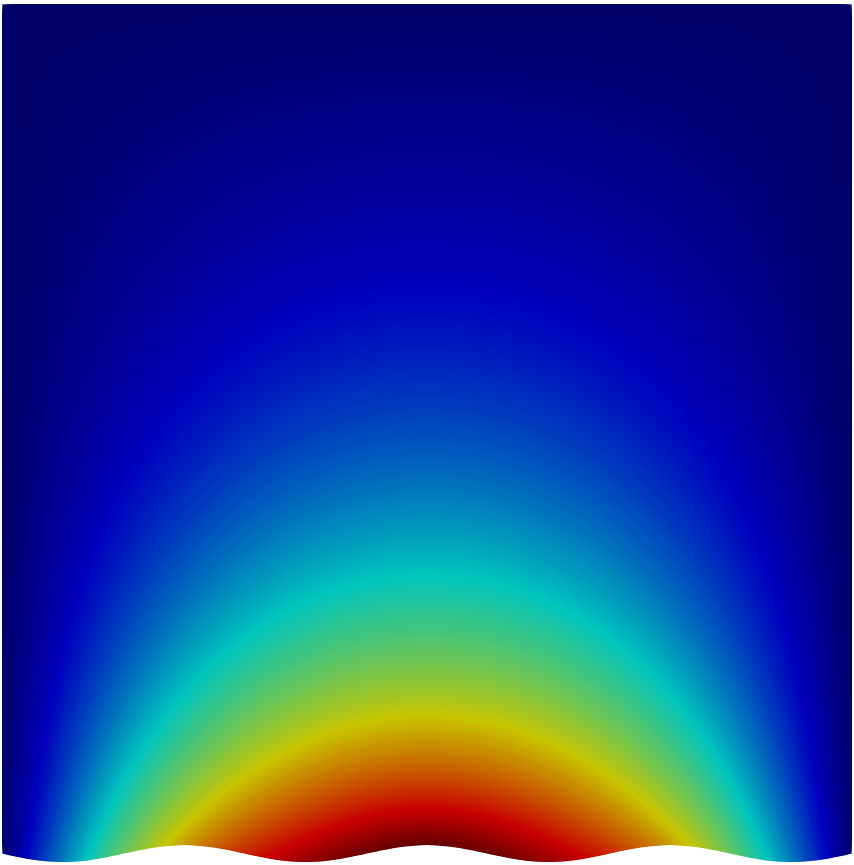}
    &\includegraphics[width=0.3\linewidth]{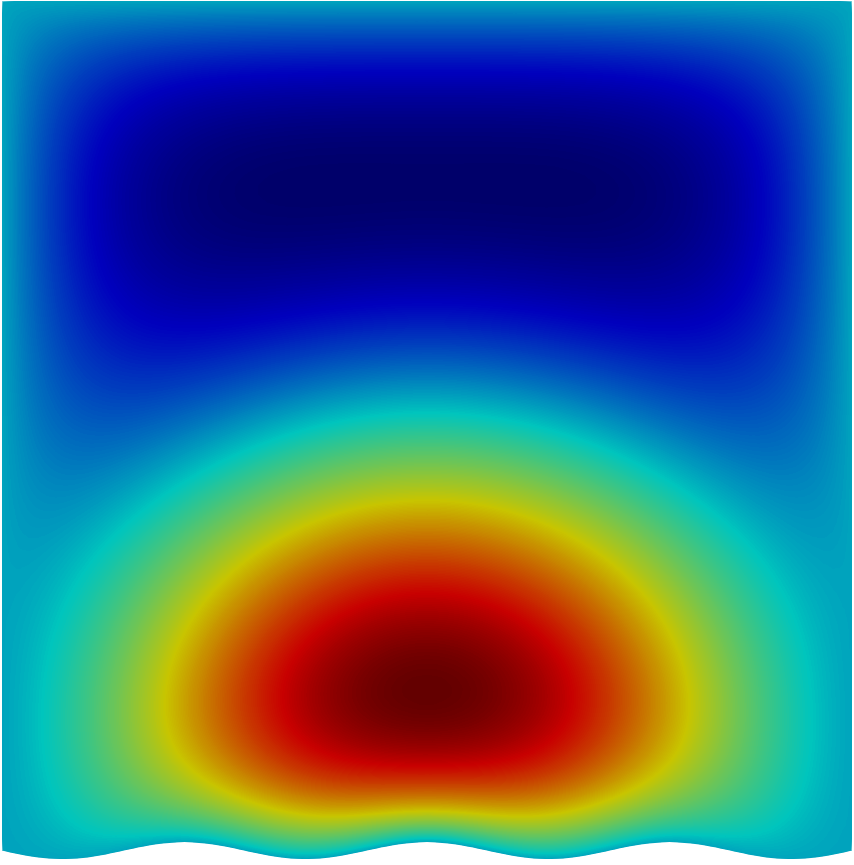}
    \\
    (a). $\hat{T}$ & (b). $T = \hat{T} + T_d$ &(c). $S$
    \end{tabular}
    \caption{Case\ref{Sect:Num-4}: Plot for the numerical solution on the initial domain $\Omega_0.$}
    \label{fig:Num-4-1}
\end{figure}

\begin{figure}[H]
    \centering
    \begin{tabular}{cccc}
         \includegraphics[width=0.21\linewidth]{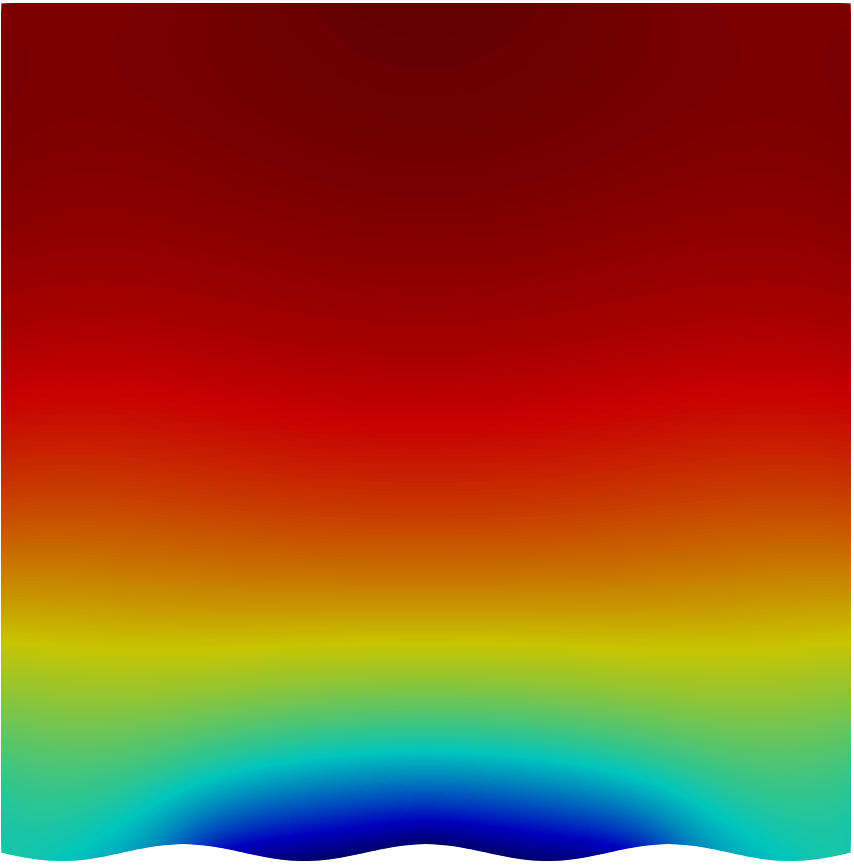}
    &\includegraphics[width=0.21\linewidth]{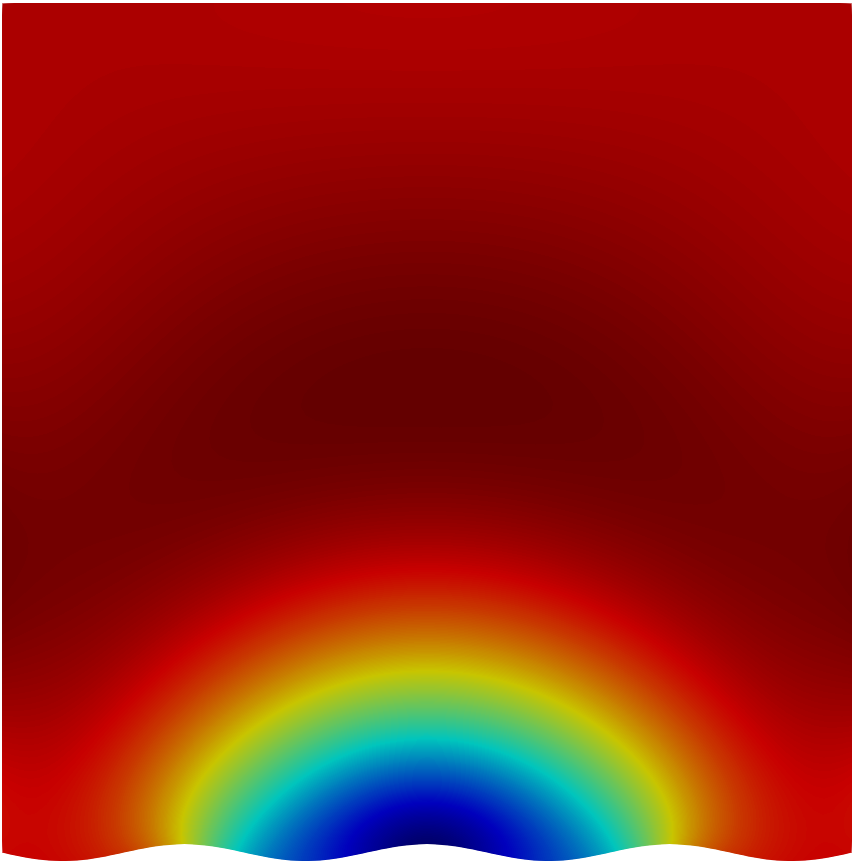}
        &\includegraphics[width=0.21\linewidth]{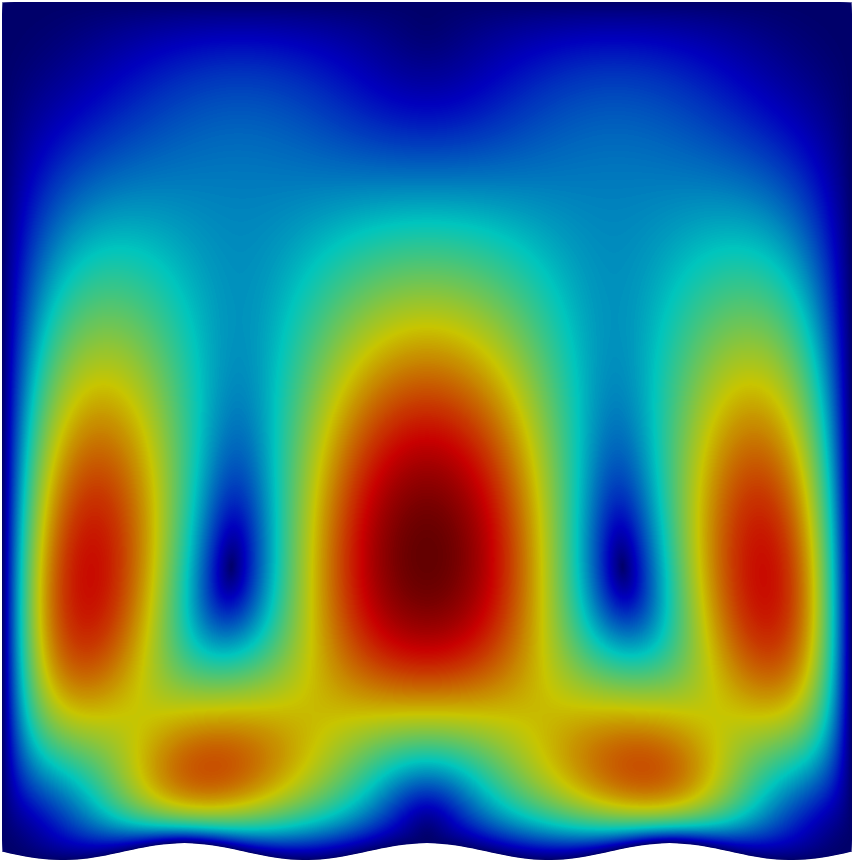}
   & \includegraphics[width=0.21\linewidth]{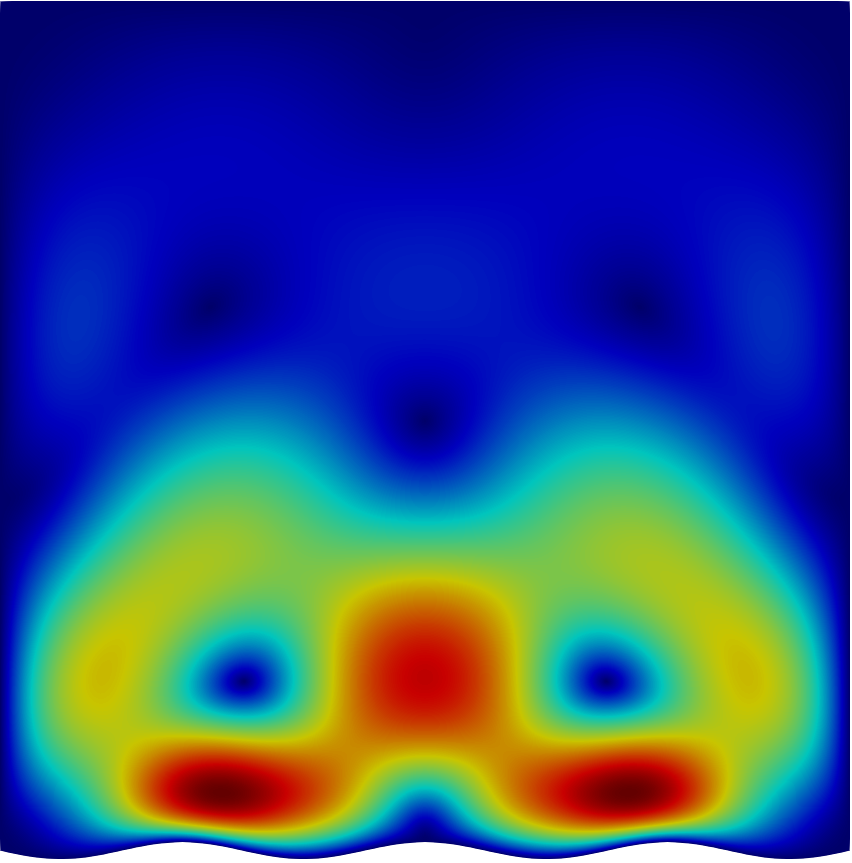}
      \\
  (a). $p$ & (b). $q$ &(c). $|{\bf v}|$ &(d). $|{\bf w}|$
    \end{tabular}
    \caption{Case\ref{Sect:Num-4}: Plot for the numerical solution on the initial domain $\Omega_0.$}
    \label{fig:Num-4-2}
\end{figure}

\begin{figure}[H]
    \centering
    \begin{tabular}{cc}
    \includegraphics[width=0.45\linewidth]{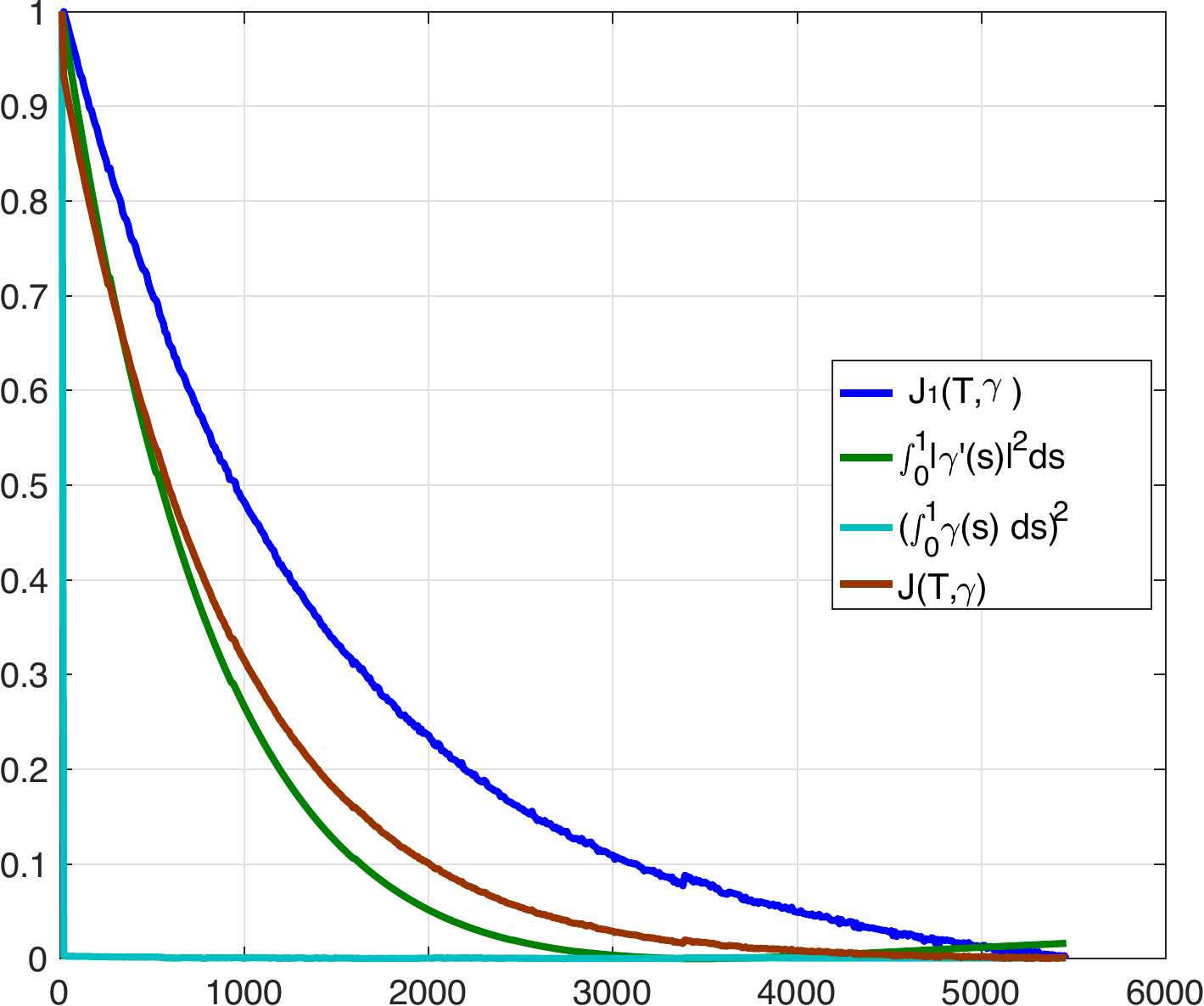}
    &\includegraphics[width=0.45\linewidth]{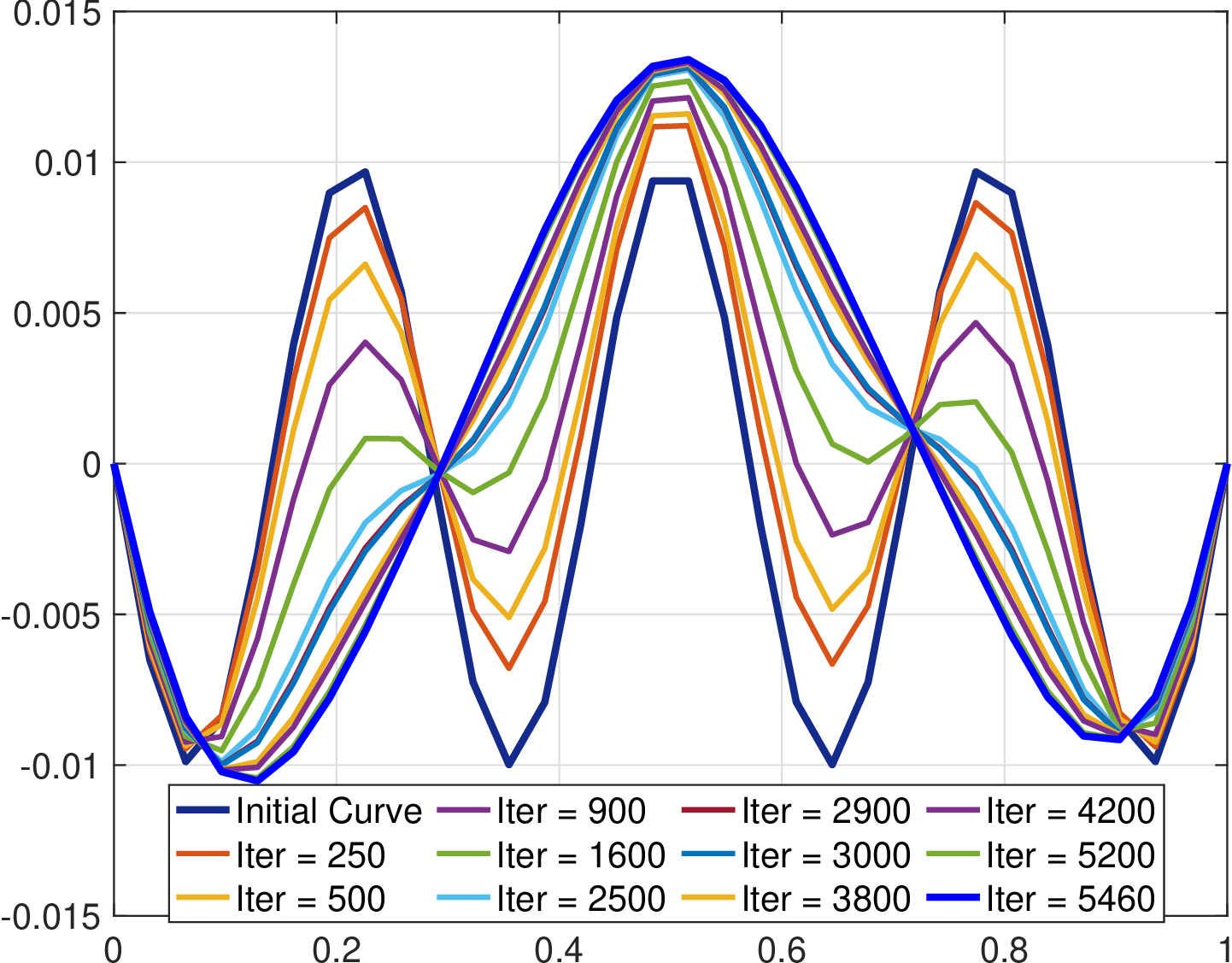}\\
    (a). Cost functional & (b). $\gamma_n$
    \end{tabular}
    \caption{Case~\ref{Sect:Num-4}: Convergence test of the cost functional and optimization curves.}
    \label{fig:Num-4-3}
\end{figure}

Again, we plot the convergence behavior for the cost functional and the optimization iterations in Fig.~\ref{fig:Num-4-3}. In this test, all the terms in the cost functional are decaying. One can also observe the oscillating behavior for the cost functional. The cost functional decays significantly in the first 2000 iterations, but slows down after Iter = 4000.
We believe that such slowing down is due to our uniform choice of the step size $\tau = 1E-3$. We may consider adopting more advanced optimization technologies including adaptive step sizing or a linear search approach to refine our algorithm, which will be left for our future research. The convergence of the curve has been plotted in Fig.~\ref{fig:Num-4-3}b. Interestingly, we observe that the optimization process will lower two bumps in the $\sin$ function but adjust the middle bump with the appropriate amplitude, which is shown in blue color. Here we omit the plot for the numerical solutions in the final domain $\Omega_{5460}$.

\subsection{Case 5: initial state with $\gamma:y = -0.1\sin(5\pi x)\exp(-3 x)$}\label{Sect:Num-5}
Finally, we test our optimization procedure by choosing a non-symmetric curvy bottom given by $y = -0.1\sin(5\pi x)\exp(-3 x)$. The results, presented from Fig. \ref{fig:Num-5-1} to \ref{fig:Num-5-5}, are similar to those obtained for the symmetric cases. We do not discuss the details to avoid repetition.

\begin{figure}[H]
    \centering
    \begin{tabular}{ccc}
    \includegraphics[width=0.3\linewidth]{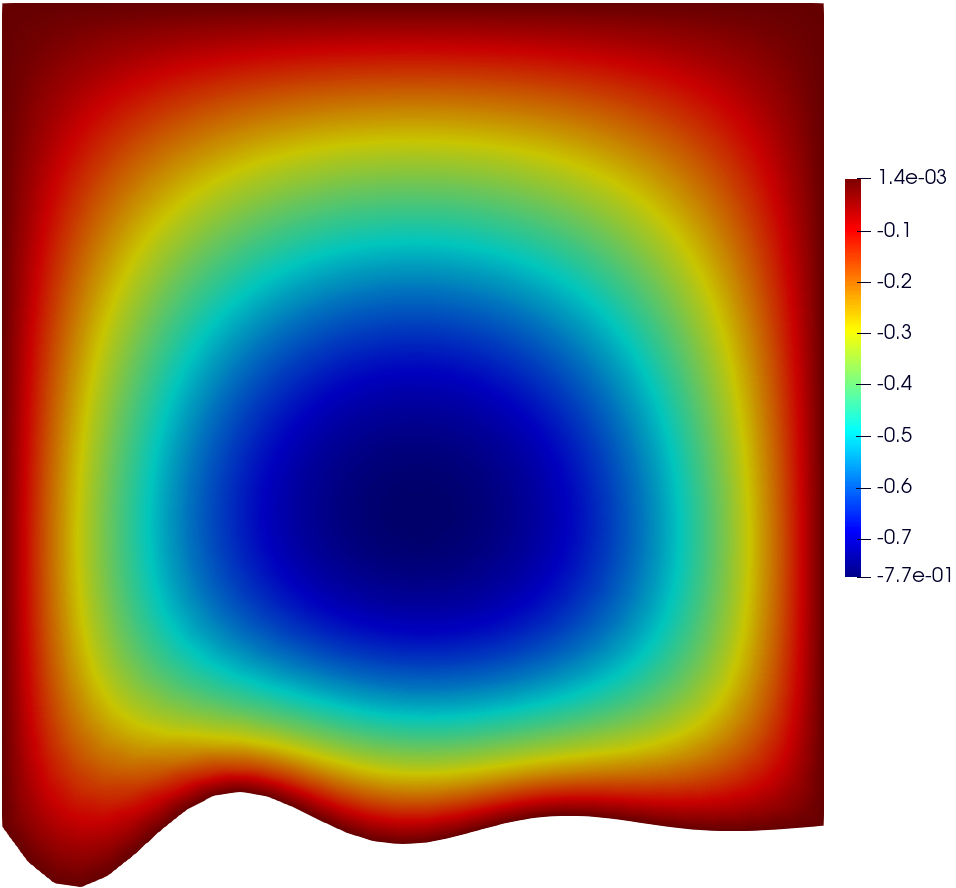}
    &\includegraphics[width=0.3\linewidth]{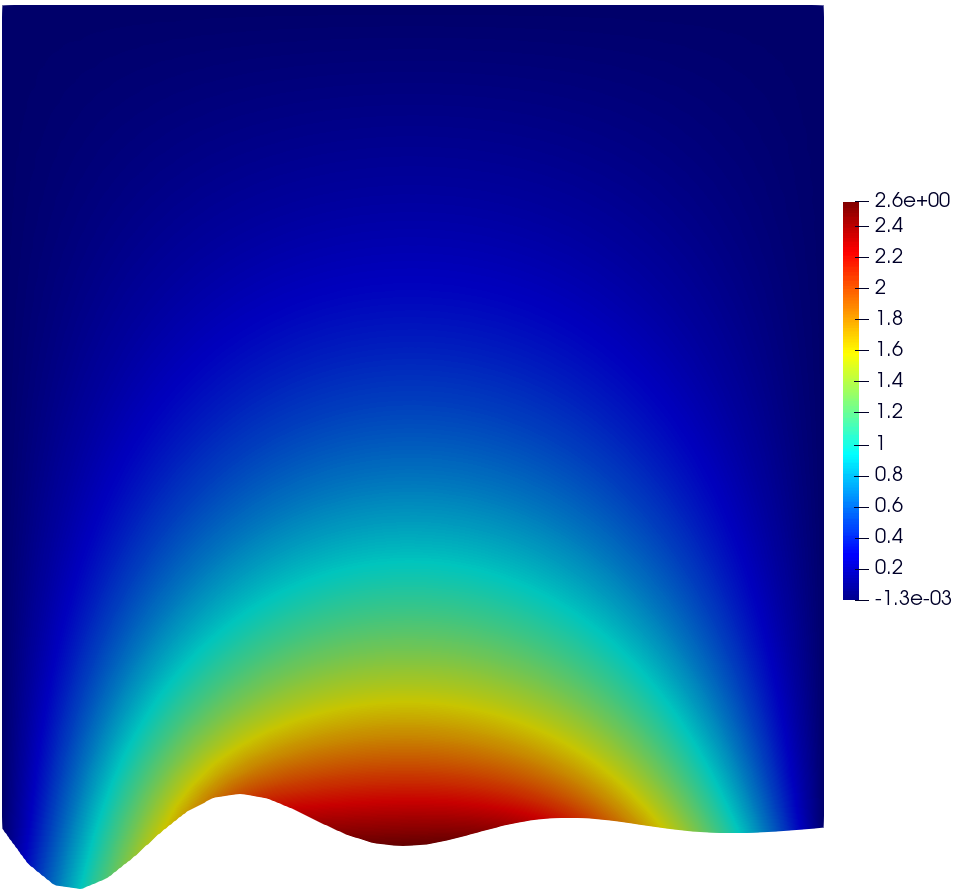}
    &\includegraphics[width=0.3\linewidth]{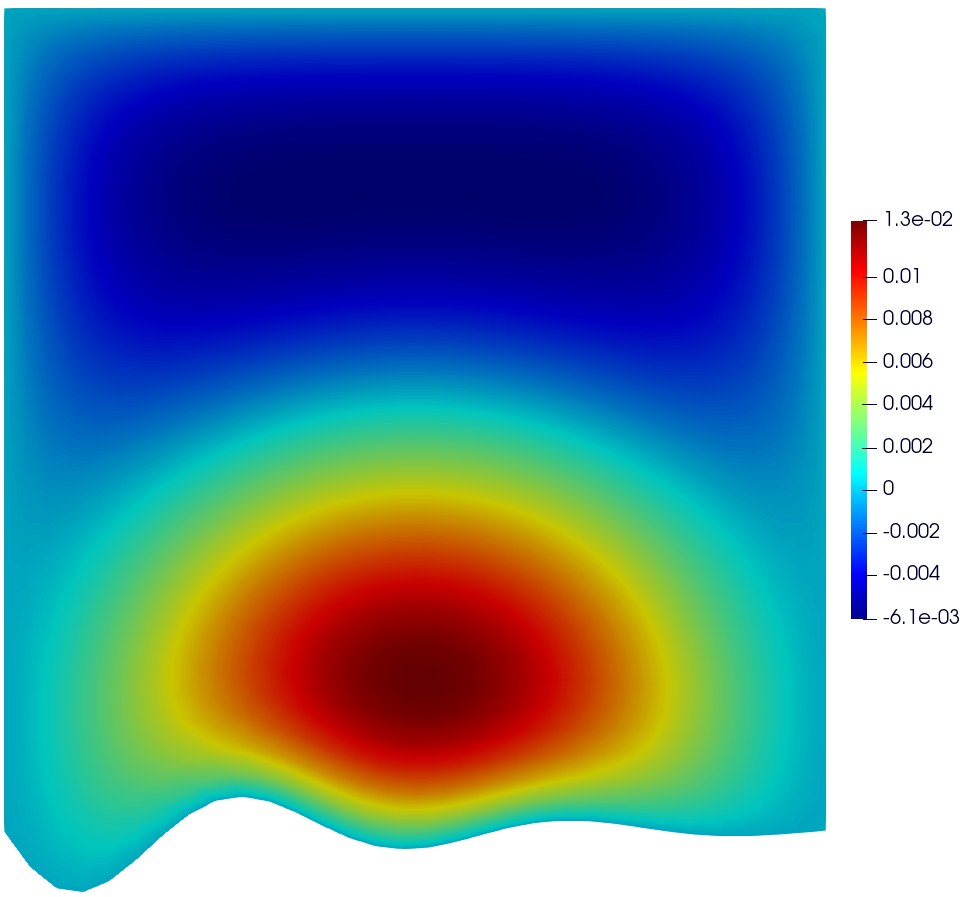}
    \\
    (a). $\hat{T}$ & (b). $T = \hat{T} + T_d$ &(c). $S$
    \end{tabular}
    \caption{Case\ref{Sect:Num-5}: Plot for the numerical solution on the initial domain $\Omega_0.$}
    \label{fig:Num-5-1}
\end{figure}

\begin{figure}[H]
    \centering
    \begin{tabular}{cccc}         \includegraphics[width=0.21\linewidth]{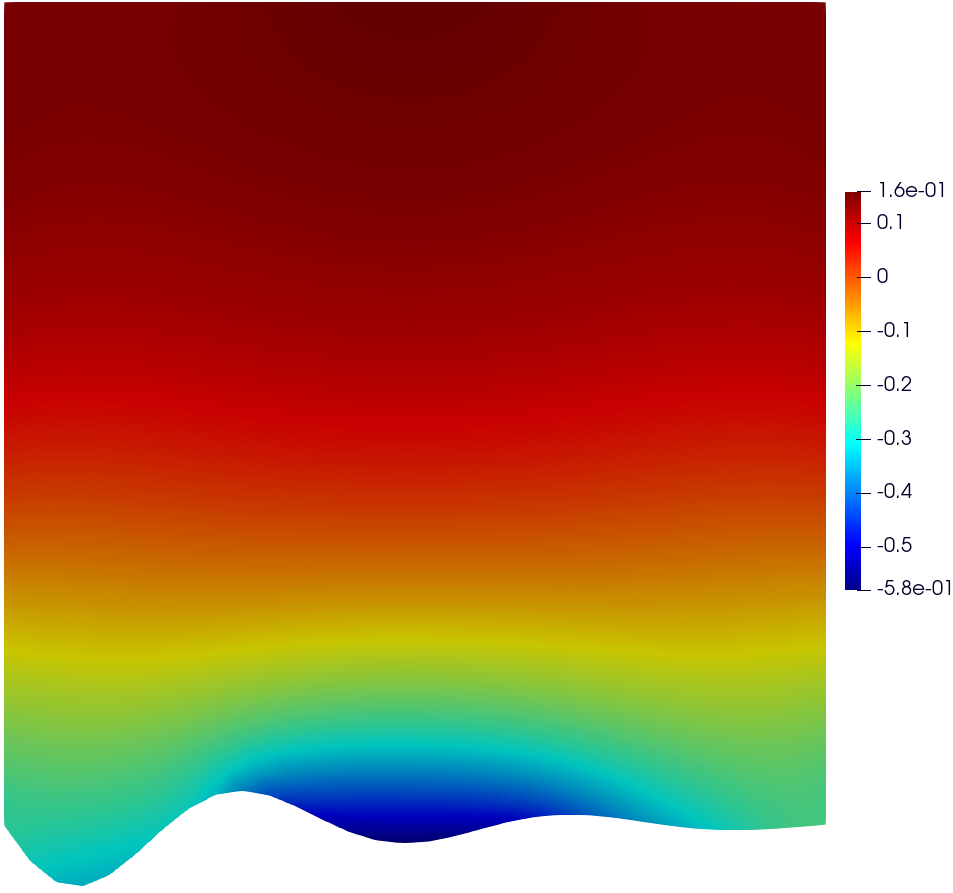}
    &\includegraphics[width=0.21\linewidth]{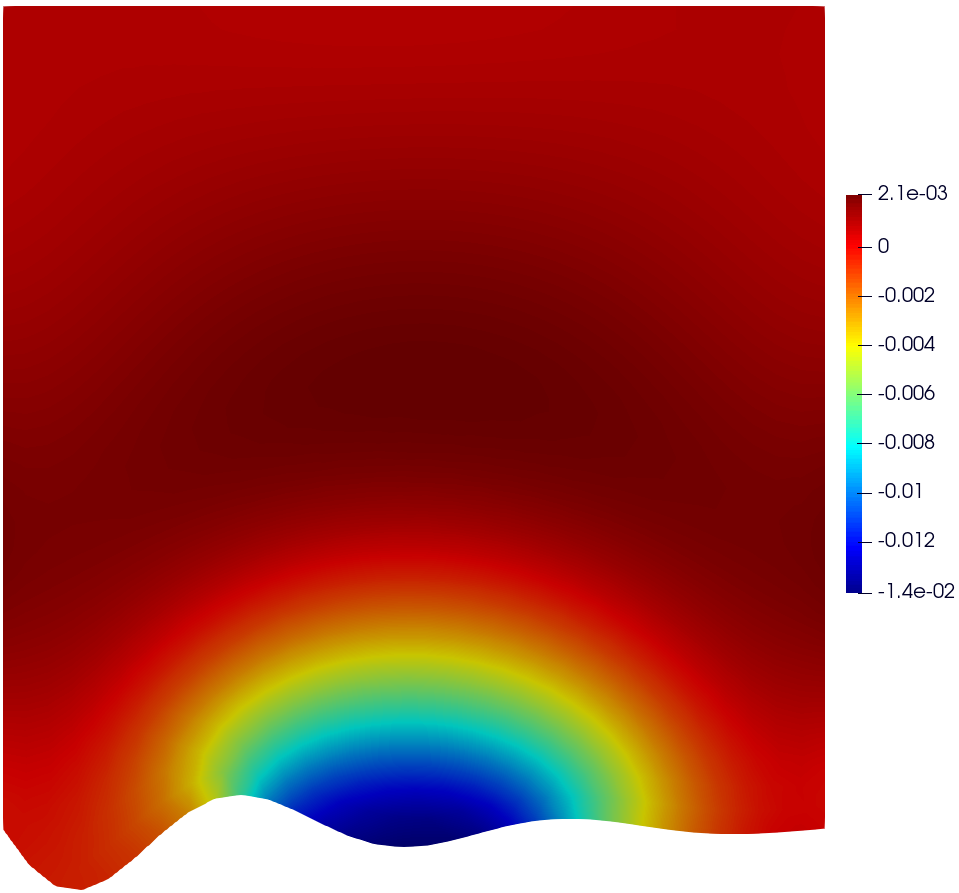}
 &\includegraphics[width=0.21\linewidth]{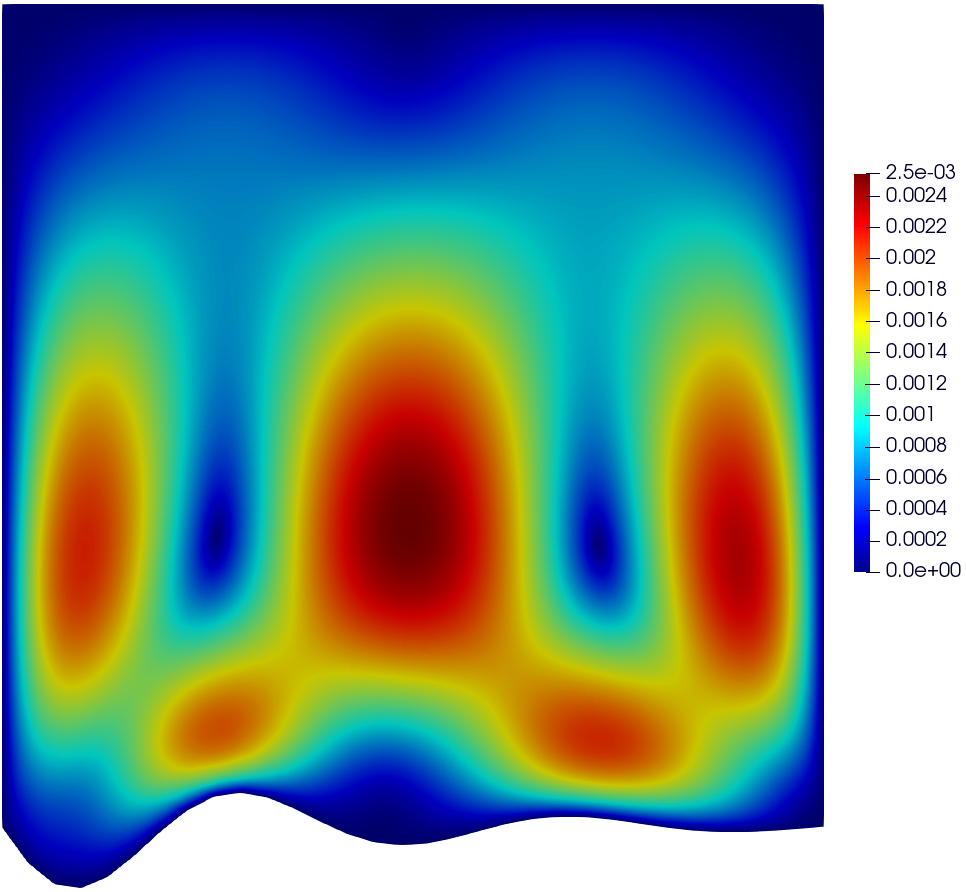}
   & \includegraphics[width=0.21\linewidth]{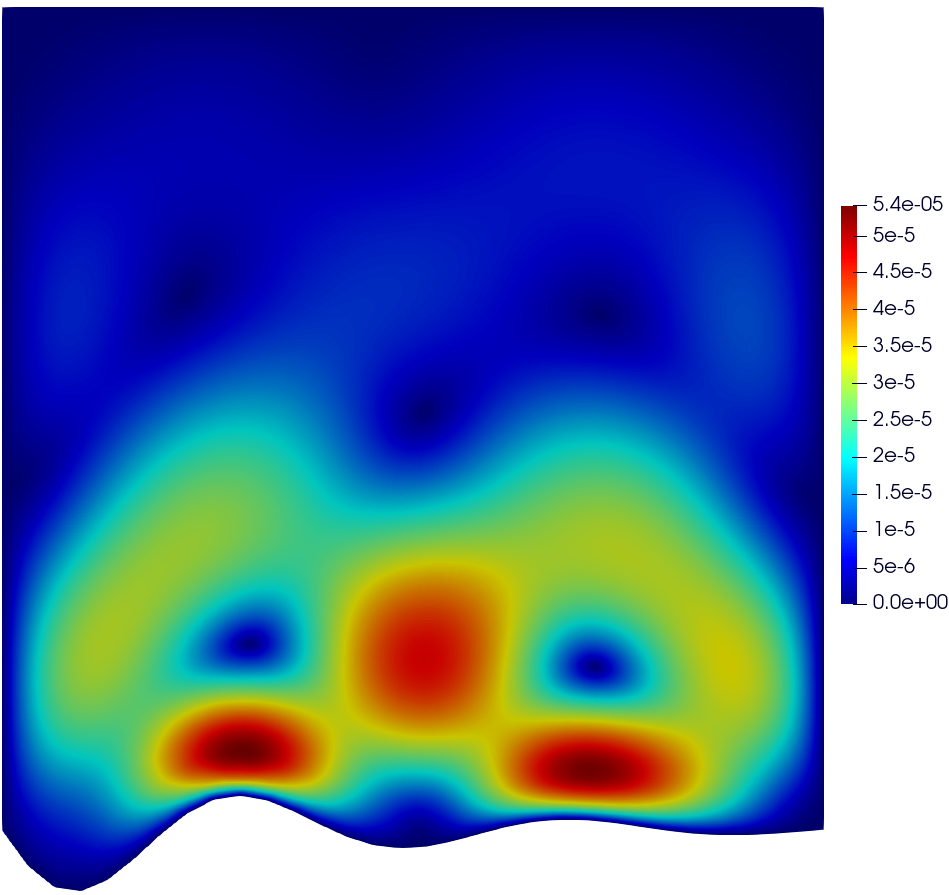}
      \\
  (a). $p$ & (b). $q$ &(c). $|{\bf v}|$ &(d). $|{\bf w}|$
    \end{tabular}
    \caption{Case\ref{Sect:Num-5}: Plot for the numerical solution on the initial domain $\Omega_0.$}
    \label{fig:Num-5-2}
\end{figure}

\begin{figure}[H]
    \centering
    \begin{tabular}{cc}
    \includegraphics[width=0.45\linewidth]{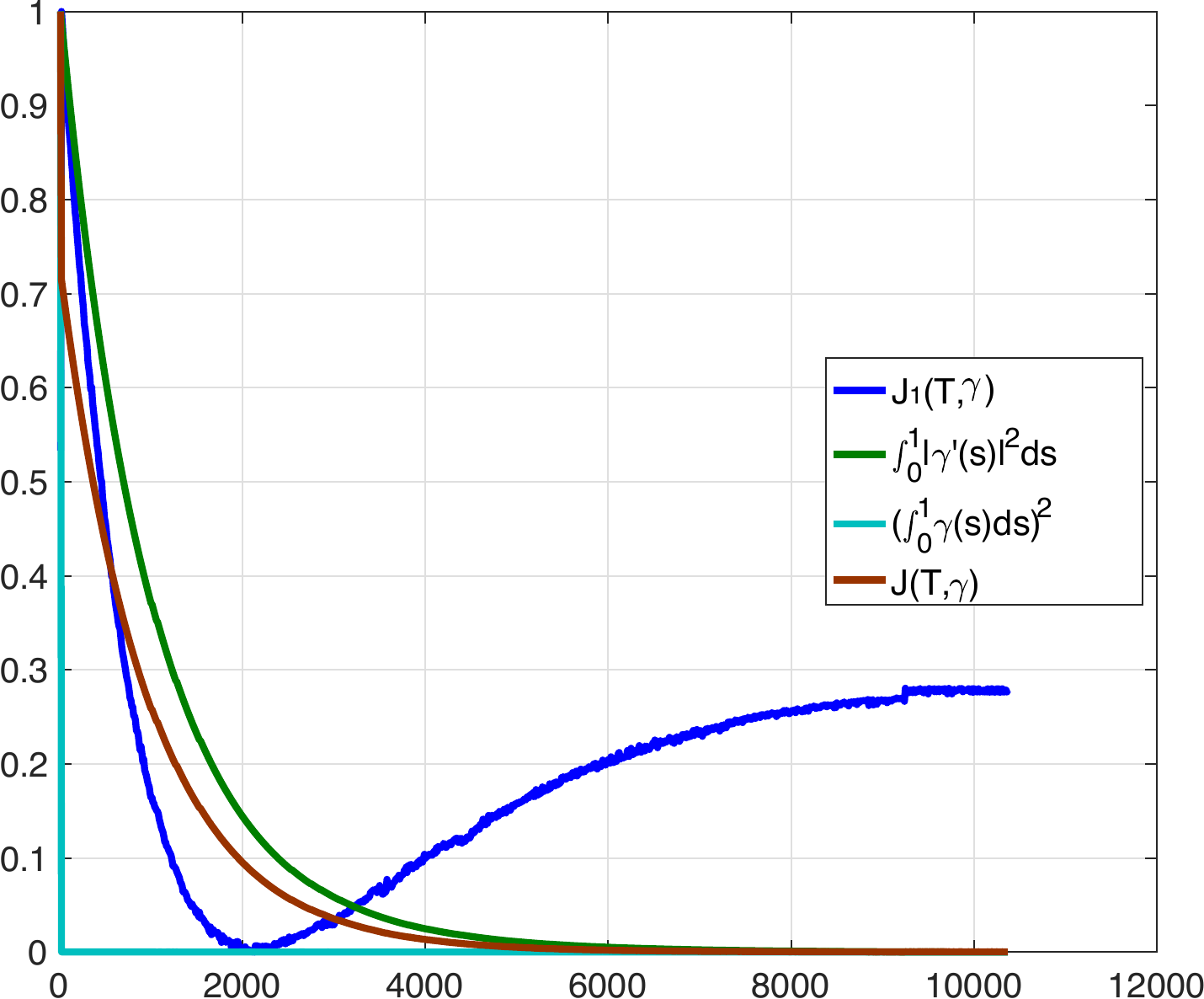}
    &\includegraphics[width=0.45\linewidth]{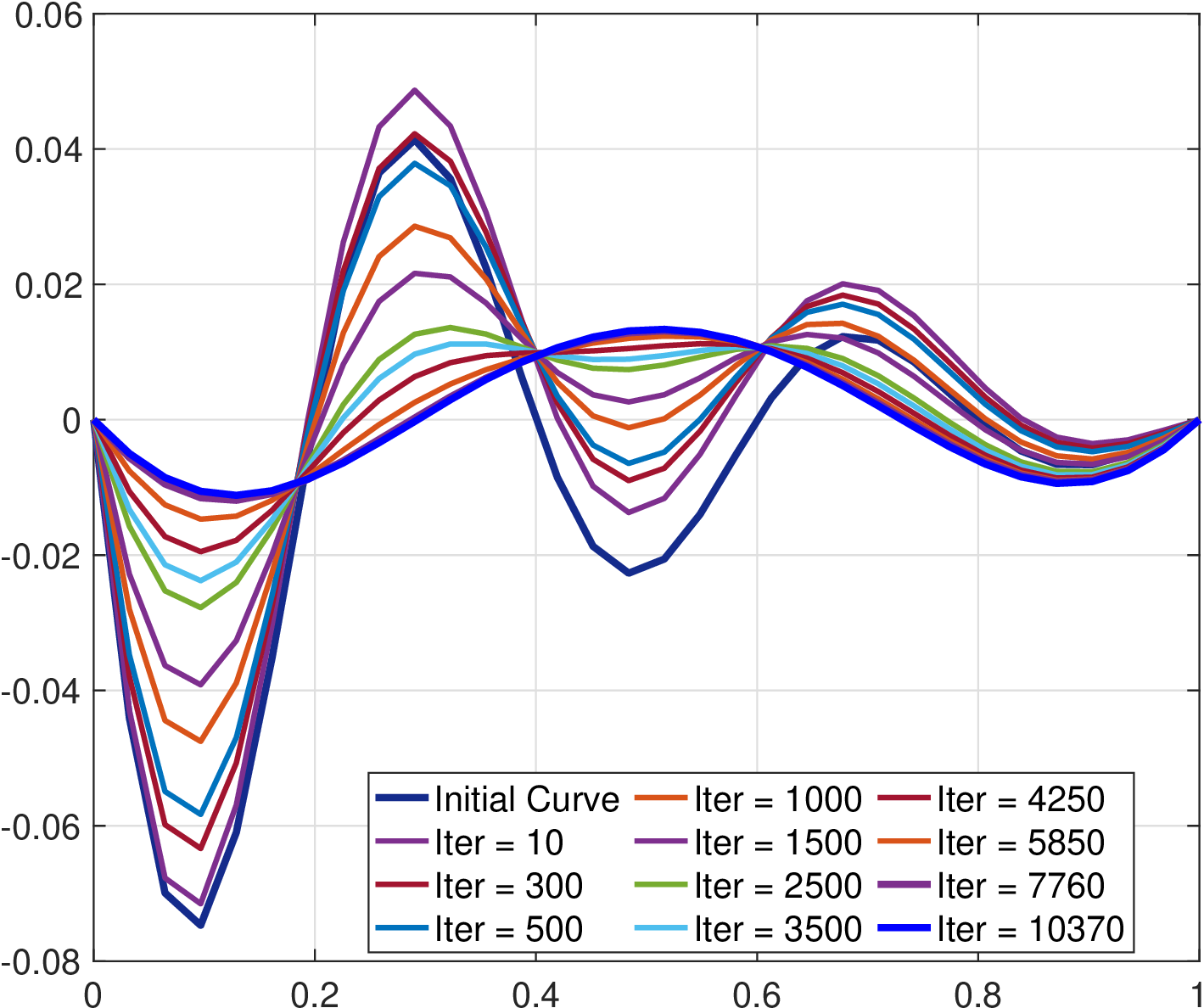}\\
    (a). Cost functional & (b). $\gamma_n$
    \end{tabular}
    \caption{Case~\ref{Sect:Num-5}: Convergence test of the cost functional and optimization curves.}
    \label{fig:Num-5-3}
\end{figure}

\begin{figure}[H]
    \centering
    \begin{tabular}{ccc}
    \includegraphics[width=0.3\linewidth]{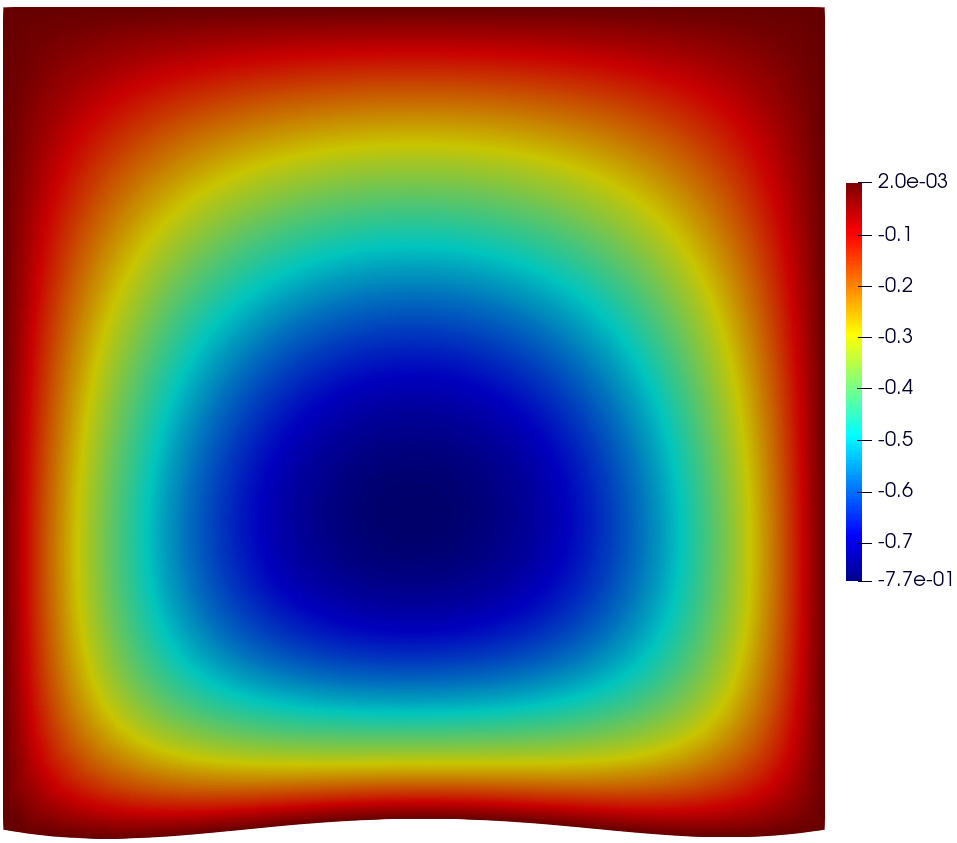}
    &\includegraphics[width=0.3\linewidth]{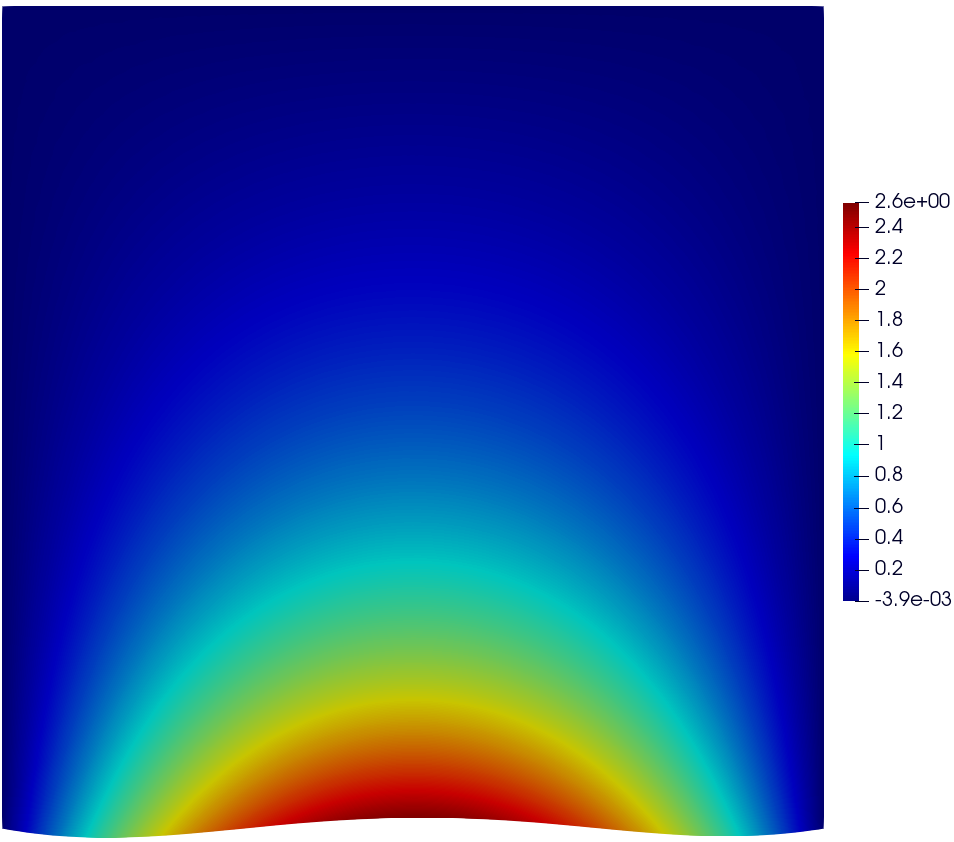}
    &\includegraphics[width=0.3\linewidth]{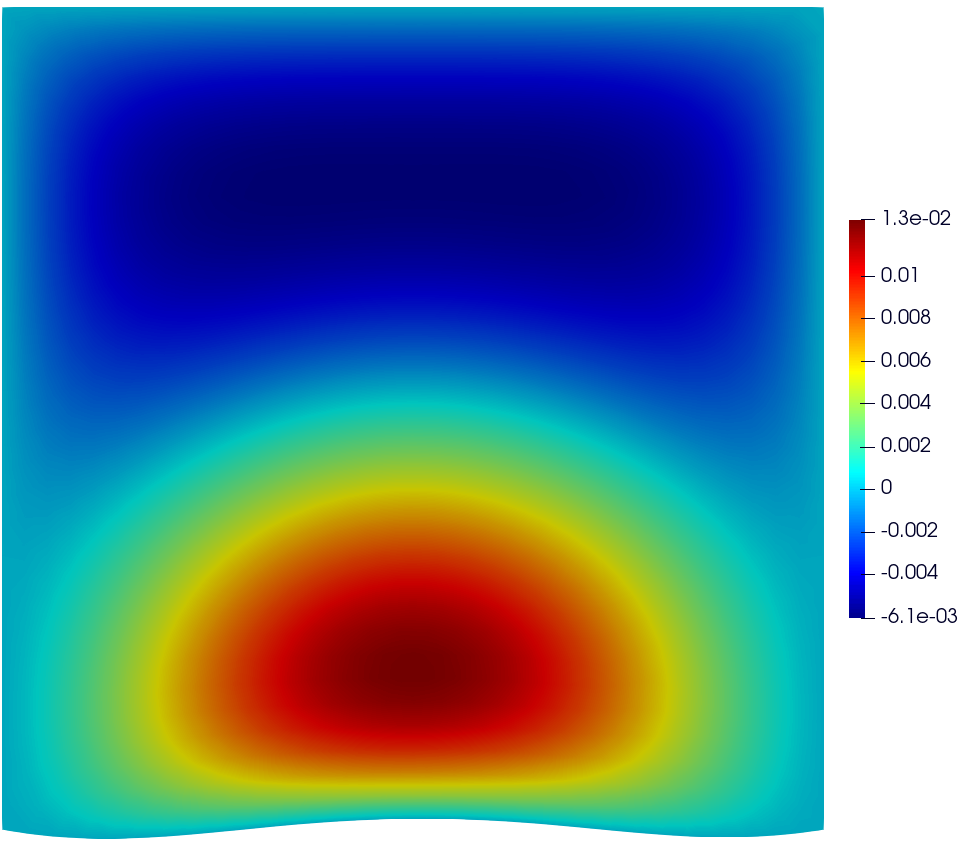}
    \\
    (a). $\hat{T}$ & (b). $T = \hat{T} + T_d$ &(c). $S$
    \end{tabular}
    \caption{Case\ref{Sect:Num-5}: Plot for the numerical solution on the final domain $\Omega_{10370}.$}
    \label{fig:Num-5-4}
\end{figure}

\begin{figure}[H]
    \centering
    \begin{tabular}{cccc}         \includegraphics[width=0.21\linewidth]{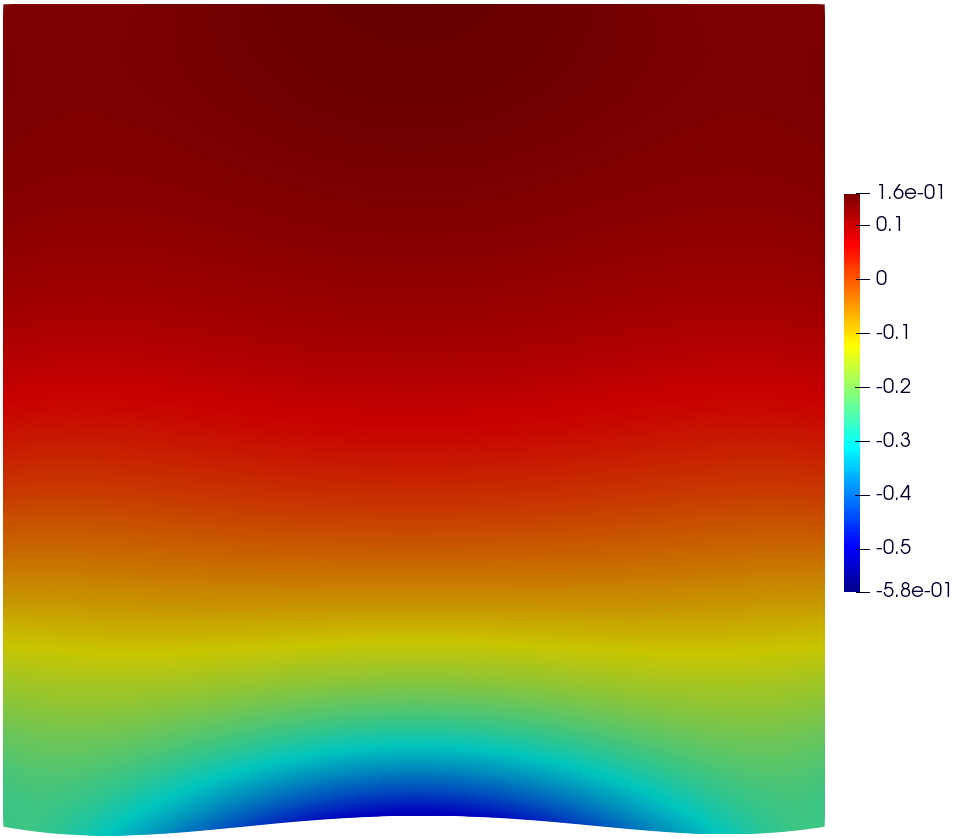}
    &\includegraphics[width=0.21\linewidth]{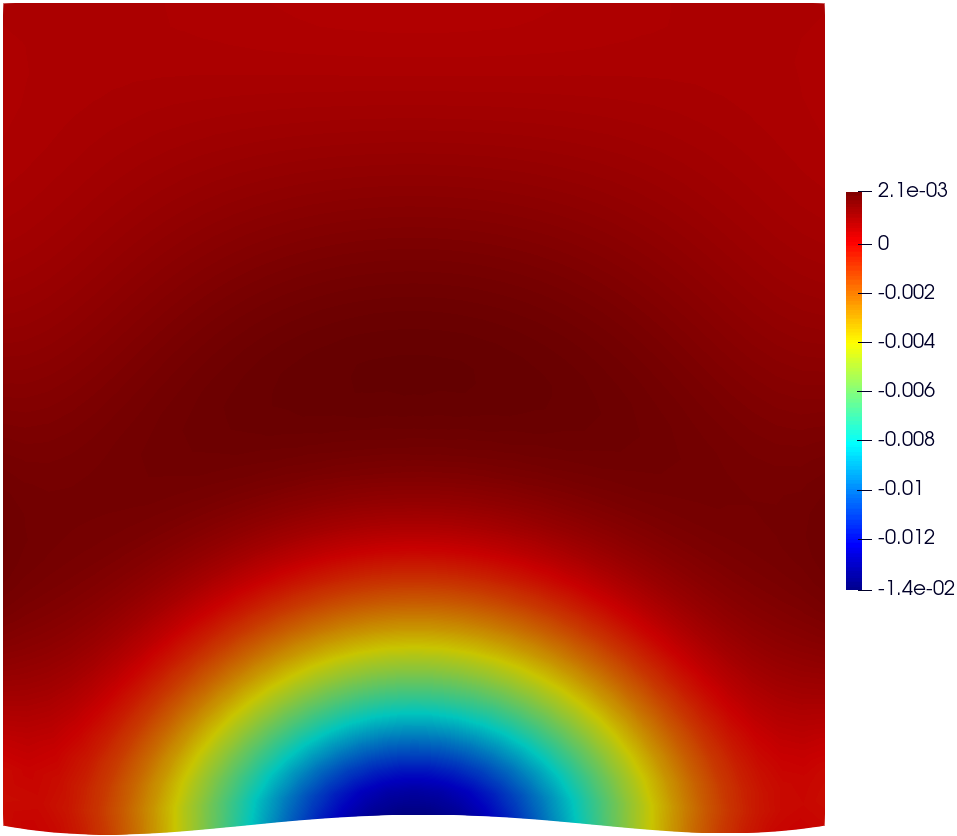}
 &\includegraphics[width=0.21\linewidth]{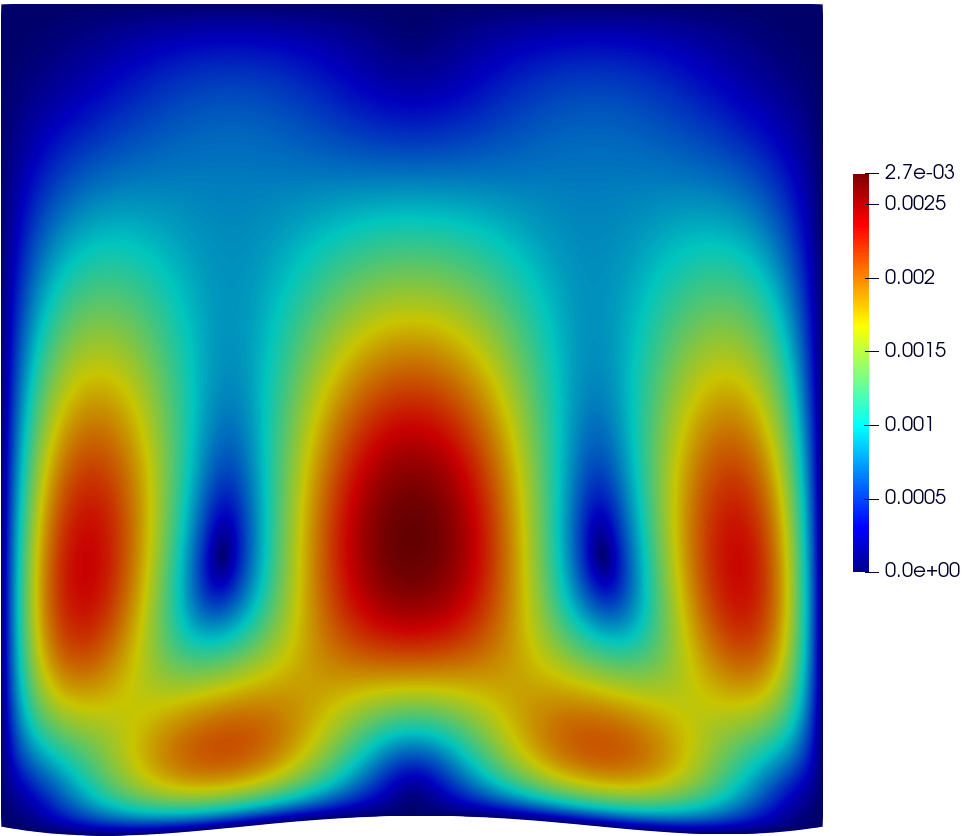}
   & \includegraphics[width=0.21\linewidth]{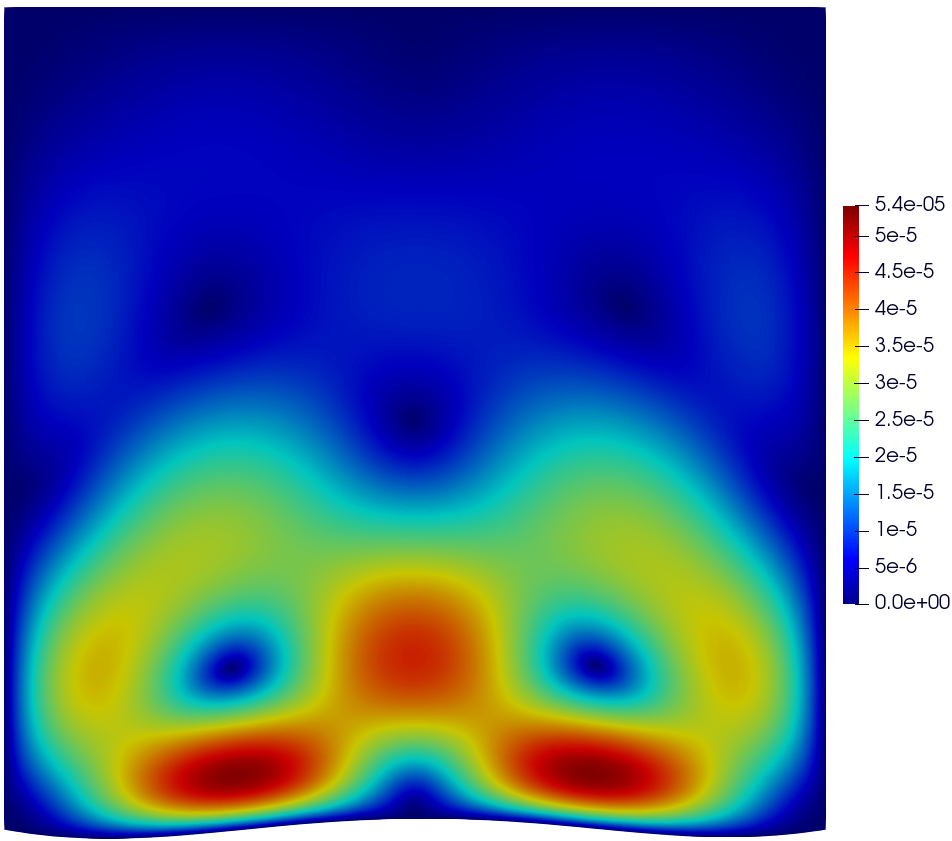}
      \\
  (a). $p$ & (b). $q$ &(c). $|{\bf v}|$ &(d). $|{\bf w}|$
    \end{tabular}
    \caption{Case\ref{Sect:Num-5}: Plot for the numerical solution on the final domain $\Omega_{10370}.$}
    \label{fig:Num-5-5}
\end{figure}
It is worth mentioning that a similar pattern emerges for the final curves from Section~\ref{Sect:Num-1}-Section~\ref{Sect:Num-4}. The plot of these curves is shown in Fig.~\ref{fig:finalcurve-allcases}. Observe that 
the curves corresponding to Case 1 and Case 4 are similar, with almost aligning with each other. The final curve to Case 3 has a fatter middle bump.

\begin{figure}[H]
    \centering
    \includegraphics[width=0.5\linewidth]{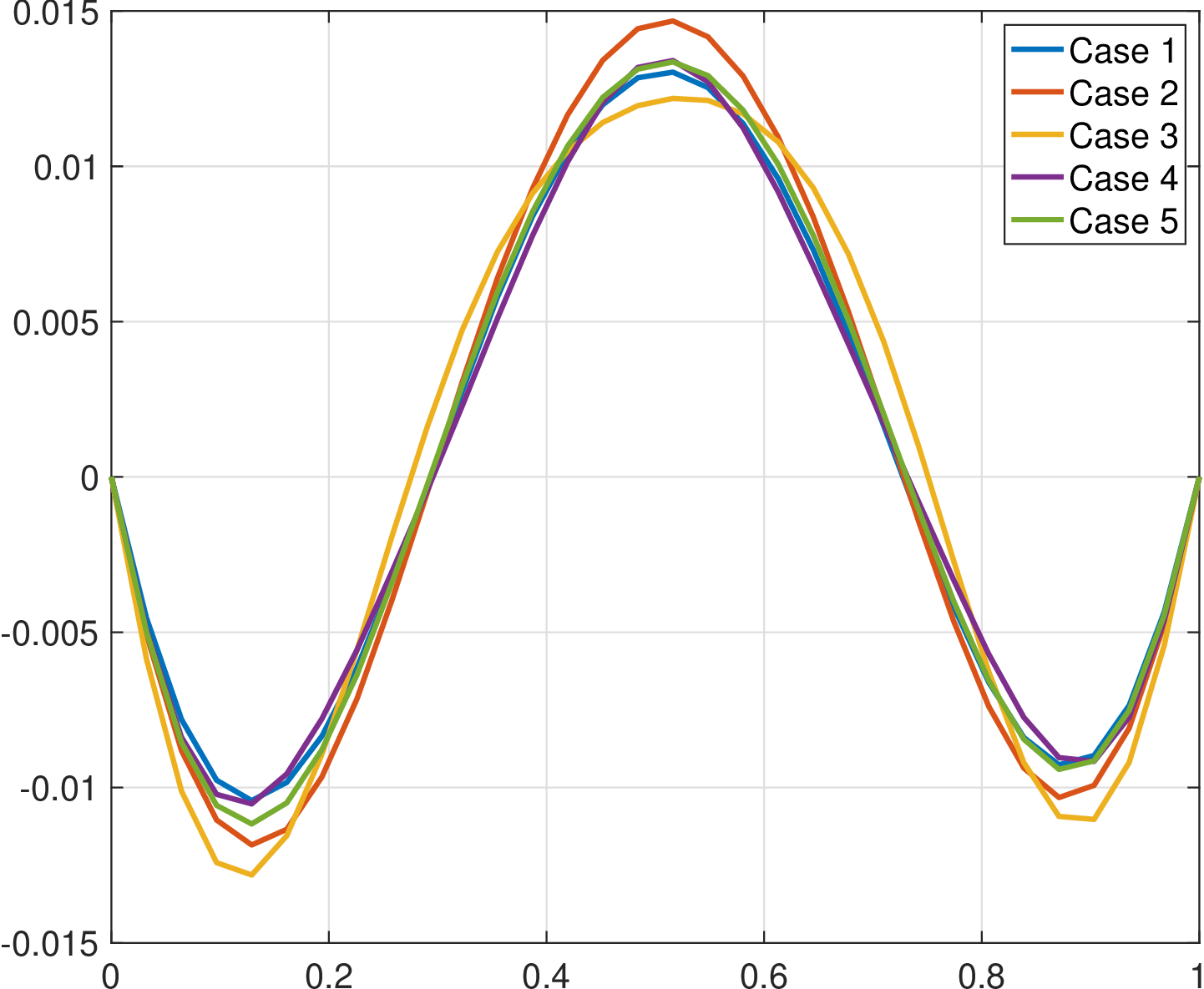}
    \caption{Comparison of the final curves.}
    \label{fig:finalcurve-allcases}
\end{figure}

The values of the cost functional corresponding to these final curves are reported in Table~\ref{tab:val_cost}. In fact, the cost functional has similar values for Case 1 and Case 4, which can explain the similarities in the final curves in these two cases. However, in Case 3, the cost functional $J(T,\gamma)$ is the greatest and it may explain the shape difference among other cases. The results presented so far suggest that all final curves converge to the same pattern. We plan to investigate this in our future research plan by using more advanced optimization techniques.
\begin{table}[H]
    \centering
    \begin{tabular}{c|c|c|c||c}
& $J_1(T,\gamma)$ & $\int_0^1|\gamma'(\xi)|d\xi$ & $(\int_0^1\gamma(\xi) d\xi)^2$ &$J(T,\gamma)$  \\ \hline
         Section~\ref{Sect:Num-1}& 1.43868e-1 &0 & 0 & 967.88582\\
         Section~\ref{Sect:Num-2}&1.34080e-1 &4.40314e-1 & 2.81108e-5 &971.32364\\
         Section~\ref{Sect:Num-3}&1.67868e-1 &1.20756 &1.00939e-5 &968.27665\\
         Section~\ref{Sect:Num-4}&1.48466e-1 &8.39608e-2 &6.19020e-4 & 972.55404\\ 
         Section~\ref{Sect:Num-5}&1.41794e-1 &1.99301e-1 &2.64846e-6 & 967.95342\\ \hline
    \end{tabular}
    \caption{Values of initial cost functional.}
    \label{tab:val_cost_ini}
\end{table}

\begin{table}[H]
    \centering
    \begin{tabular}{c|c|c|c||c}
         & $J_1(T,\gamma)$ & $\int_0^1|\gamma'(\xi)|d\xi$ & $(\int_0^1\gamma(\xi) d\xi)^2$ &$J(T,\gamma)$  \\ \hline
         Section~\ref{Sect:Num-1}& 1.41462e-1 &5.46374e-3 & 1.85319e-9 & 967.88477\\
         Section~\ref{Sect:Num-2}&1.41184e-1 &6.91638e-3 &1.83922e-9 &967.88486\\
         Section~\ref{Sect:Num-3}&1.41181e-1 &7.36001e-3 &1.78769e-9 &967.88497\\
         Section~\ref{Sect:Num-4}&1.41469e-1 &5.50588e-3 &1.85346e-9 & 967.88479\\ 
         Section~\ref{Sect:Num-5}&1.41370e-1 &5.92158e-3 &1.85248e-9 & 967.88480\\ \hline
    \end{tabular}
    \caption{Values of final cost functional. }
    \label{tab:val_cost}
\end{table}

\section{Conclusion}\label{Sec:Conclusion}
In this work, we explored the problem of optimizing the shape of a fluid container to force a uniform temperature distribution. Our analysis began with the investigation of the state Boussinesq system, establishing its well-posedness along with boundary regularity properties for the weak solution. We then introduced a shape optimization framework and proved the existence of an optimal domain. To characterize such optimal shapes, we derived a first-order optimality condition using the adjoint method, supported by a detailed analysis of the adjoint system. A crucial step involves handling domain perturbations in a rigorous way to ensure the existence of directional derivatives of the objective functional. Finally, numerical experiments were conducted to validate the theoretical results. The setting explored in this paper represents a canonical example of fluid systems where a quantity (such as temperature or material concentration) is transported through both diffusion and convection. Our broader goal is to develop a solid theoretical framework for control and design strategies in such systems, along with efficient numerical methods for their practical implementation. The results presented here contribute to this long-term objective. Future directions include the integration of additional active components and the consideration of more complex domain modifications to enhance mixing efficiency. These may involve shape deformations of boundary regions and topological changes, such as introducing internal holes to create channels that promote mixing. We are particularly interested in active design mechanisms based on boundary or distributed controls for both temperature and flow, as well as dynamic boundary elements (such as flexible, controllable membranes) that can adapt to optimize performance.
\bibliographystyle{plain} 
\bibliography{References}

\end{document}